\documentclass[11pt,twoside]{amsart}

\newtheorem{theorem}{Theorem}[section]  
\newtheorem{proposition}{Proposition}[section]  
\newtheorem{definition}{Definition}[section]  
\newtheorem{lemma}{Lemma}[section]  
\newtheorem{remark}{Remark}[section]  
\newtheorem{corollary}{Corollary}[section]  
\usepackage{amsmath,amsfonts,amssymb}
\usepackage{graphicx}    

\def\Id{{\rm Id}\,}
\def\da{\delta\!a}

\def\du{\delta\!u}

\def\dv{\delta\!v}

\def\dr{\delta\!\rho}

\def\dF{\delta\!F}
\def\dG{\delta\!G}
\def\d{\partial}
\def\ddj{\dot \Delta_j}

\def\ddk{\dot \Delta_k}

\def\divA{\, \hbox{\rm div}_A\,  }
\def\divx{\, \hbox{\rm div}_x\,  }
\def\divy{\, \hbox{\rm div}_y\,  }

\def\tilde{\widetilde}
\def\hat{\widehat}
\def\wh{\widehat}
\def\wt{\widetilde}

\newcommand\R{\mathbb{R}}

\newcommand\T{\mathbb{T}}
\newcommand\Z{\mathbb{Z}}
\newcommand{\ds}{\displaystyle}
\newcommand{\N}{\mathbb{N}}
\newcommand{\ep}{\varepsilon}

\newcommand{\Supp}{\hbox{Supp}\,} 
\newcommand{\Tr}{\hbox{\rm{Tr}\,}}
\newcommand{\adj}{\hbox{\rm{adj}\,}}
\renewcommand{\Re}{\,{\rm{Re}}\,}

\def\divx{\, \hbox{\rm div}_x\,  }
\def\divy{\, \hbox{\rm div}_y\,  }

\renewcommand{\div}{\mbox{\rm div}\;\!}

\newcommand{\bftau}{\mbox{\boldmath $\tau$}}

\def\eqdefa{\buildrel\hbox{\footnotesize def}\over =}

\def\cA{{\mathcal A}}

\def\cC{{\mathcal C}}
\def\cD{{\mathcal D}}

\def\cF{{\mathcal F}}

\def\cL{{\mathcal L}}

\def\cO{{\mathcal O}}
\def\cP{{\mathcal P}}
\def\cQ{{\mathcal Q}}
\def\cS{{\mathcal S}}

\newcommand{\Int}{\displaystyle \int}
\newcommand{\Frac}{\displaystyle \frac}

\def\d{\partial}

\def\ep{\varepsilon}
\def\eps{\varepsilon}
\def\tilde{\widetilde}
\def\hat{\widehat}

\begin{document}

\title{Fourier analysis methods for the compressible Navier-Stokes equations}

\author{Rapha\"el Danchin  
}


\address{Universit\'{e} Paris-Est,  LAMA UMR 8050, UPEMLV, UPEC, CNRS,\\
61 avenue du G\'en\'eral de Gaulle, 94010 Cr\'eteil Cedex\\ 
danchin@univ-paris12.fr}



\maketitle

\begin{abstract}
In the last three decades, Fourier analysis methods have known a growing importance in the study of linear and nonlinear PDE's. In particular, techniques based on Littlewood-Paley decomposition and paradifferential calculus have proved to be very efficient for 
investigating evolutionary  fluid mechanics equations in the whole space or in the torus.
We here give an overview of results that we can get by Fourier analysis and paradifferential calculus, for  the compressible 
Navier-Stokes equations.  We focus on the Initial Value Problem 
in the case where the fluid domain is $\R^d$ (or the torus $\T^d$) with $d\geq2,$  and also 
establish some asymptotic properties of global small solutions.
The time decay estimates in the critical regularity framework that are stated at the end of the survey
are new,   to the best of our knowledge. 
\end{abstract}

\section{Introduction}

In the Eulerian description,  a general compressible   fluid evolving in some open set $\Omega$ of  $\R^d$ is characterized 
at every  material point~$x$ in~$\Omega$ and time  $t\in\R$
by its  {\it velocity field} $u= u(t,x)\in\R^d,$   {\it  density} $\varrho=\varrho(t,x)\in\R_+,$
{\it pressure} $p=p(t,x)\in\R,$
{\it internal energy  by unit mass}  $e= e(t,x)\in\R,$
{\it entropy by unit mass} $s=s(t,x)$ and  {\it absolute temperature} $T=T(t,x).$
In the absence of external forces, those quantities are governed by:
\begin{itemize}
\item The mass balance:
$$ 
\partial_t\varrho+\div(\varrho u)=0.
$$
\item The momentum balance\footnote{With the convention that $\bigl(\div(a\otimes b)\bigr)^j\eqdefa \sum_i\d_i(a^i\,b^j).$}:
$$
\partial_t(\varrho u)+\div(\varrho u\otimes u)=\div\bftau-\nabla p,
$$
where $\bftau$ stands for the \emph{viscous stress tensor}.
\item The energy balance:
$$
\partial_t\Bigl(\varrho\bigl(e+{\textstyle\frac{| u|^2}2}\bigr)\Bigr)
+\div\Bigl(\varrho\bigl(e+{\textstyle\frac{| u|^2}2}\bigr) u\Bigr)
=\div(\bftau\cdot u+pu)-\div  q,
$$
where $q$ is the \emph{heat flux} vector. 
\item  The  entropy inequality:
\begin{equation}\label{ineq:entropy}\d_t(\varrho s)+\div(\varrho s u)\geq -\div\bigl(\textstyle{\frac{q}T}\bigr)\cdotp
\end{equation}
\end{itemize} 
In what follows, we concentrate on so-called \emph{Newtonian fluids}\footnote{That is to say:
 the viscous stress tensor $\bftau$ is a linear function of  $D_xu,$ invariant under rigid transforms,
there is no internal  mass couples (and thus the {\it angular momentum} is conserved), and  the fluid is isotropic ({\it viz.} the physical quantities
depend only on $(t,x)$).}. Hence   (see e.g.\cite{batchelor})  $\bftau$ is given by:
$$
\bftau\eqdefa \lambda\div u\,\Id+2\mu D( u),
$$
where the real numbers $\lambda$ and $\mu$ are the {\it viscosity coefficients}
and $D( u)\eqdefa \frac12(D u+{}^T\!D u)$ is the {\it deformation tensor.} 
\smallbreak
If we assume  in addition that the {\it Fourier law} $q=-k\nabla T$ is satisfied  then we get the following system of equations:
\begin{equation}\label{eq:fluid}
\left\{\begin{array}{l}
\partial_t\varrho+\div(\varrho u)=0,\\
\partial_t(\varrho u)+\div(\varrho u\otimes u)-2\div(\mu D( u)) - \nabla(\lambda\div u) +\nabla p=0,\\
\partial_t(\varrho e)+\div(\varrho e u)+p\div u-\div(k\nabla T)=2\mu D( u):D( u)+\lambda(\div u)^2.
\end{array}\right.
\end{equation}
We shall further postulate that the entropy $s$ is interrelated with $p,$ $T$ and $e$ through 
the so-called {\it Gibbs relation}
$$
Tds=de+p\, d\Bigl(\frac1\varrho\Bigr),
$$
and thus we get the following evolution equation for $s$:
\begin{equation}\label{eq:entropy}
T\bigl(\d_t(\varrho s)+\div(\varrho s u)\bigr)=\bftau\cdot D( u)-\div q.
\end{equation}
For  the entropy inequality  to be satisfied, a necessary and sufficient condition is thus
$$\bftau:D( u)+k\frac{|\nabla T|^2}{T}\geq0,$$
which yields the following  constraints on $\lambda,$ $\mu$ and $k$:
$$
k\geq0,\quad\mu\geq 0\quad\hbox{and}\quad
2\mu+d\lambda\geq0.
$$
In order to close System \eqref{eq:fluid} which is  composed of  $d+2$ equations for $d+4$ unknowns (namely $\varrho,$ $e,$ $p,$ $T$ and $ u^1,\cdots, u^d$), we need another two  {\it state equations}
interrelating  $p$,~$\varrho$,~$e$,~$s$ and~$T.$
In this survey, for simplicity  we shall  focus on \emph{barotropic gases} that is $p$ depends only on the density and 
$\lambda$ and $\mu$ are independent of $T.$ Therefore the system constituted by the first
two equations in \eqref{eq:fluid},  the so-called 
  {\it barotropic compressible Navier-Stokes system}~:
\begin{equation}\label{eq:nsbaro}
\left\{\begin{array}{l}
\partial_t\varrho+\div(\varrho u)=0,\\
\partial_t(\varrho u)+\div(\varrho u\otimes u)-\div\bigl(2\mu D(u)+\lambda\,\div u\, \Id\bigr)+\nabla p= 0
\end{array}\right.
\end{equation}
where  $p\eqdefa P(\varrho)$ for some given  smooth function $P,$  is closed.
\bigbreak
Our main goal is to solve the Initial Value Problem (or Cauchy Problem) for \eqref{eq:nsbaro}
supplemented with initial data $(\varrho_0,u_0)$ at time $t=0$ \emph{in the case where the fluid domain $\Omega$ is the whole space or the torus.}
We will  concentrate on the   local well-posedness issue for large data with no vacuum, 
on   the global well-posedness issue for small perturbations of a constant stable equilibrium, and will give exhibit some of the qualitative properties
of the constructed solutions.
As regards global results,  the concept of \emph{critical regularity} is fundamental. 
Indeed, experience shows that whenever the PDE system under consideration 
possesses some scaling invariance with respect to space and time dilations (which is in general the case when it comes from mathematical physics) then appropriate so-called critical norms or quantities essentially control the (possible) finite time blow-up and the asymptotic properties of the solutions.  
\medbreak
 If modifying the pressure law accordingly then the barotropic system \eqref{eq:nsbaro} we are here considering has the following  scaling invariance:
 \begin{equation}\label{scaling}
\varrho(t,x)\leadsto \varrho(\ell ^2t,\ell x),\qquad
u(t,x)\leadsto \ell u(\ell^2t,\ell x), \qquad\ell>0.
\end{equation}
More precisely, $(\varrho,u)$ is a solution to \eqref{eq:nsbaro} if and only if so does  $(\varrho_\ell,u_\ell),$ with pressure function $\ell^2 P.$
This means that we expect optimal solution spaces (and norms) for \eqref{eq:nsbaro} to have the
scaling invariance pointed out above.
\medbreak
The rest of these notes unfolds as follows. 
In the next section, we present the basic tools and estimates that will be needed  to 
study System \eqref{eq:nsc}. 
Then we concentrate on  the local well-posedness issue for \eqref{eq:nsc} in critical spaces.
Section \ref{s:global} is dedicated to solving \eqref{eq:nsbaro} globally for small data. 
The last section  is devoted to asymptotic results for  the system.
 We concentrate on the low Mach number limit and
on the decay rates of global solutions in the critical regularity framework.


\section{The Fourier analysis toolbox}

We here shortly introduce  the  Fourier analysis tools needed
in this survey, then  state estimates for the heat and transport equations that will play
a fundamental role.
 For the sake of conciseness, some proofs are just sketched or omitted. Unless otherwise specified, 
the reader will find details in \cite{BCD}, Chap. 2 or 3.

\subsection{The Littlewood-Paley decomposition}

The Littlewood-Paley decomposition  is a  dyadic localization procedure in the frequency space
for tempered distributions over $\R^d.$
One of the main motivations for  using it when dealing with PDEs   is that the derivatives
 act almost as   dilations  on distributions with  Fourier
transform  supported in  a ball or an annulus, as regards $L^p$ norms. 
This is exactly what is stated in the following proposition:
 \begin{proposition}[Bernstein inequalities] Let $0<r<R.$
 \begin{itemize}
\item  There exists a constant~$C$ so that, for any $k\in\N$, couple~$(p,q)$ 
in~$[1,\infty]^2$ with  $q\geq p\geq 1$ 
and  function $u$ of~$L^p$ with $\wh u$ supported in the  ball $B(0,\lambda R)$ of $\R^d$ for some $\lambda>0,$ we have
$D^ku\in L^q$ and 
$$
\|D^k u\|_{L^q} \leq  C^{k+1}\lambda^{k+d(\frac{1}{p}-\frac{1}{q})}\|u\|_{L^p}.
 $$
\item  There exists a constant~$C$ so that
for any  $k\in\N,$ $p\in[1,\infty]$ and  function $u$ of $L^p$ with 
$\Supp\,\hat u \subset \{\xi\in\R^d\,/\, r\lambda\leq|\xi|\leq R\lambda\}$
for some $\lambda>0,$ we have
$$
\lambda^k\|u\|_{L^p}
\leq C^{k+1}\|D^k u\|_{L^p}.
$$
\end{itemize}
\end{proposition}
As general solutions to nonlinear PDE's need not  be spectrally localized in annuli,  we want 
 a device  for splitting any function or distribution into a sum of spectrally localized functions.
 To this end,   fix some  smooth radial non increasing function $\chi$
with ${\rm Supp}\, \chi\subset B(0,\frac 43)$ and 
$\chi\equiv1$ on $B(0,\frac34),$ then set
$\varphi(\xi)=\chi(\xi/2)-\chi(\xi).$ We thus have 
$$
\sum_{j\in\Z}\varphi(2^{-j}\cdot)=1\ \hbox{ in }\ \R^d\setminus\{0\}\cdotp
$$
The homogeneous  dyadic blocks  $\ddj$ are defined by 
$$\ddj u\eqdefa\varphi(2^{-j}D)u\eqdefa\cF^{-1}(\varphi(2^{-j}\cdot)\cF u)=2^{jd}h(2^j\cdot)\star u
\quad\hbox{with}\quad h\eqdefa\cF^{-1}\varphi.
$$
We also introduce  the low frequency cut-off operator $\dot S_j$:
$$
\dot S_ju\eqdefa\chi(2^{-j}D)u\eqdefa\cF^{-1}(\chi(2^{-j}\cdot)\cF u)=2^{jd}\check h(2^j\cdot)\star u
\quad\hbox{with}\quad \check h\eqdefa\cF^{-1}\chi.
$$
Let us emphasize that operators $\ddj$ and $\dot S_j$ are continuous 
on $L^p,$ with norm \emph{independent} of $j,$ a property that would fail if taking a rough cut-off
function $\chi$ (unless $p=2$ of course).
The price to pay for smooth cut-off is that  $\ddj$ is not an  $L^2$ orthogonal projector. 
However the following important quasi-orthogonality property is fulfilled:
\begin{equation}\label{eq:quasi1}
\ddj\ddk=0\ \hbox{ if }\  |j-k|>1.
\end{equation}
The homogeneous   Littlewood-Paley decomposition for $u$ reads
\begin{equation}\label{eq:LP}
u=\sum_{j}\ddj u.
\end{equation}
This equality holds  \emph{modulo polynomials} only.
In order to have equality in the distributional sense, one may consider 
  the set  $\cS'_h$ of tempered distributions $u$ such that 
$$
\lim_{j\rightarrow-\infty}\|\dot S_ju\|_{L^\infty}=0.
$$
As  distributions of $\cS'_h$  tend to $0$ at infinity, one can easily conclude that 
 \eqref{eq:LP} holds true  in $\cS'$ whenever $u$ is in $\cS'_h.$


\subsection{Besov spaces} 

It is obvious that for all $s\in\R,$ we have 
\begin{equation}\label{eq:b1}
C^{-1}\|u\|_{\dot H^s}^2\leq\sum_{j\in\Z} 2^{2js}\|\ddj u\|_{L^2}^2\leq C\|u\|_{\dot H^s}^2,
\end{equation}
and it is  also not very difficult to prove that for $s\in(0,1),$
\begin{equation}\label{eq:b2}
C^{-1}\|u\|_{\dot C^{0,s}}\leq \sup_{j\in\Z} 2^{js}\|\ddj u\|_{L^\infty}\leq C\|u\|_{\dot C^{0,s}},\qquad s\in(0,1).
\end{equation}
In \eqref{eq:b1} and \eqref{eq:b2}, we observe  that three parameters come into play: 
the regularity  parameter $s,$ the Lebesgue exponent that is used for bounding $\ddj u$
and the type of summation that is done over $\Z.$
This  motivates the following definition: 
\begin{definition}
 For $s\in\R$ and  
$1\leq p,r\leq\infty,$ we set
$$
\|u\|_{\dot B^s_{p,r}}\eqdefa\bigg(\sum_{j\in\Z} 2^{rjs}
\|\ddj  u\|^r_{L^p}\bigg)^{\frac{1}{r}}\ \text{ if }\ r<\infty
\quad\!\text{and}\quad\!
\|u\|_{\dot B^s_{p,\infty}}\eqdefa\sup_{j\in\Z} 2^{js}\|\ddj  u\|_{L^p}.
$$
We then define the homogeneous Besov space $\dot B^s_{p,r}$ to be the 
subset of  distributions $u\in {\cS}'_h$ such  that
$\|u\|_{\dot B^s_{p,r}}<\infty.$
\end{definition}
We shall often use the following classical  properties: 
\begin{itemize}
\item \emph{Scaling invariance}: 
For any $s\in\R$ and $(p,r)\in[1,+\infty]^2$ there exists a constant $C$ such that
for all  $\lambda>0$ and $u\in\dot B^s_{p,r},$ we have
\begin{equation}\label{eq:scaling}
C^{-1}\lambda^{s-\frac dp} \|u\|_{\dot B^s_{p,r}}\leq 
\|u(\lambda\cdot)\|_{\dot B^s_{p,r}}\leq C\lambda^{s-\frac dp}  \|u\|_{\dot B^s_{p,r}}.
\end{equation}
\item \emph{Completeness}: $\dot B^s_{p,r}$ is a Banach space whenever $s<d/p$ or $s\leq d/p$ and $r=1.$
\item \emph{Fatou property}: if $(u_n)_{n\in\N}$ is a bounded sequence of functions of $\dot B^s_{p,r}$
that converges in  $\cS'$  to some $u\in\cS'_h$ 
then $u\in \dot B^s_{p,r}$ and 
$\|u\|_{\dot B^s_{p,r}}\leq C\liminf \|u_n\|_{\dot B^s_{p,r}}.$
\item \emph{Duality}:  If $u$ is  in~$\cS'_h$ then  we have
$$
\|u\|_{\dot B^s_{p,r}} \leq C \sup_\phi \langle u,\phi\rangle
$$
where the supremum is taken over  those~$\phi$ in~$\cS\cap \dot B^{-s}_{p',r'}$
 such that~$\|\phi\|_{\dot B^{-s}_{p',r'}} \leq 1.$
  \item \emph{Interpolation}: The following inequalities are  satisfied for\footnote{With the convention that
  $A\lesssim B$ means that $A\leq CB$ for some `harmless' positive constant $C.$}
  all $1\leq p,r_1,r_2,r\leq\infty,$ $s_1\not=s_2$ and $\theta\in(0,1)$:
  $$
  \|u\|_{\dot B^{\theta s_2+(1-\theta)s_1}_{p,r}}\lesssim\|u\|_{\dot B^{s_1}_{p,r_1}}^{1-\theta}
  \|u\|_{\dot B^{s_2}_{p,r_2}}^\theta.
  $$  
  \item \emph{Action of Fourier multipliers}: If $F$ is a smooth homogeneous  of degree $m$ function on $\R^d\setminus\{0\}$ then 
\begin{equation}\label{eq:fm}
F(D):\dot B^s_{p,r}\to\dot B^{s-m}_{p,r}.\end{equation}
 In particular, the gradient operator maps $\dot B^s_{p,r}$
 in $\dot B^{s-1}_{p,r}.$  
   \end{itemize}
\begin{proposition}[Embedding for Besov spaces on $\R^d$]
\begin{enumerate}
\item For any $p\in[1,\infty]$ we have the  continuous embedding
$
\dot B^0_{p,1}\hookrightarrow L^p\hookrightarrow \dot B^0_{p,\infty}.$
\item 
If  $s\in\R,$ $1\leq p_1\leq p_2\leq\infty$ and $1\leq r_1\leq r_2\leq\infty,$
  then $\dot B^{s}_{p_1,r_1}\hookrightarrow
  \dot B^{s-d(\frac1{p_1}-\frac1{p_2})}_{p_2,r_2}.$
  \item   For $s'<s$ and any  $1\leq p, r_1, r_2\leq\infty,$
  the embedding  of $\dot B^s_{p,r_1}$ in $\dot B^{s'}_{p,r_2}$   
  is locally compact, i.e.   for any $\varphi\in\cS,$
  the map $u\mapsto\varphi u$ is compact from   $\dot B^{s}_{p_1,r_1}$ to  $\dot B^{s'-d(\frac1{p_1}-\frac1{p_2})}_{p_2,r_2}.$  
  \item The space  $\dot B^{\frac dp}_{p,1}$ is continuously embedded in   the set  of
bounded  continuous functions (going to $0$ at infinity if, additionally,   $p<\infty$).
  \end{enumerate}
  \end{proposition}


\subsection{Paraproduct and nonlinear estimates}

 Formally, the  product  of two  tempered distributions $u$ and $v$   may be decomposed
into 
\begin{equation}\label{eq:bony}
uv=T_uv+R(u,v)+T_vu\end{equation}
with
$$T_uv\eqdefa\sum_j\dot S_{j-1}u\,\ddj v
\ \hbox{ and }\ 
R(u,v)\eqdefa\sum_j\sum_{|j'-j|\leq1}\ddj u\,\dot\Delta_{j'}v.
$$
The above operator $T$ is called ``paraproduct'' whereas
$R$ is called ``remainder''.
The decomposition \eqref{eq:bony} has been first introduced 
by J.-M. Bony in \cite{Bony}.   We observe that  in Fourier variables the sum in $T_uv$
 is \emph{locally finite}, hence  $T_uv$ is  always defined. We shall see however that  it cannot be smoother than what is given by high frequencies,
namely $v.$
As for the remainder, it may be not defined, but if it is  then 
the regularity exponents add up. 
All that is detailed below:
\begin{proposition}\label{p:op}
Let $(s,r)\in\R\times[1,\infty]$  and $1\leq p,p_1,p_2\leq\infty$ with  $1/p=1/p_1+1/p_2.$
\begin{itemize}
\item We have:
$$
\|T_uv\|_{\dot B^s_{p,r}}\lesssim \|u\|_{L^{p_1}}\|v\|_{\dot B^{s}_{p_2,r}}\quad\hbox{and}\quad
\|T_uv\|_{\dot B^{s+t}_{p,r}}\lesssim \|u\|_{\dot B^t_{p_1,\infty}}\|v\|_{\dot B^{s}_{p_2,r}},\quad
\hbox{if }\ t<0.
$$
\item If
$s_1+s_2>0$ and $1/r=1/r_1+1/r_2\leq1$ then
 $$\|R(u,v)\|_{\dot B^{s_1+s_2}_{p,r}}\lesssim\|u\|_{\dot B^{s_1}_{p_1,r_1}}
\|v\|_{\dot B^{s_2}_{p_2,r_2}}.$$
\item If $s_1+s_2=0$  and $1/r_1+1/r_2\geq1$ then
 $$\|R(u,v)\|_{\dot B^{0}_{p,\infty}}\lesssim\|u\|_{\dot B^{s_1}_{p_1,r_1}}
\|v\|_{\dot B^{s_2}_{p_2,r_2}}.$$
\end{itemize}
\end{proposition}
Putting together decomposition \eqref{eq:bony} and the above 
results, one may    get the following product estimate that  depends only \emph{linearly} 
on the highest norm of $u$ and $v$:
\begin{corollary}\label{c:op}
Let $u$ and $v$ be in $L^\infty\cap \dot B^{s}_{p,r}$ for some $s>0$
and $(p,r)\in[1,\infty]^2.$  Then there exists a constant $C$ depending only on $d,$ $p$ and $s$
and such that 
$$
\|uv\|_{\dot B^s_{p,r}}\leq C\bigl(\|u\|_{L^\infty}\|v\|_{\dot B^s_{p,r}}+\|v\|_{L^\infty}\|u\|_{\dot B^{s}_{p,r}}\bigr).
$$
\end{corollary}

\begin{remark}\label{r:algebra}
 Because  $\dot B^{\frac dp}_{p,1}$
is  embedded in $L^\infty,$ we deduce that whenever  $p<+\infty,$
  the product of two functions in   $\dot B^{\frac dp}_{p,1}$
is also in  $\dot B^{\frac dp}_{p,1}$ and that  for some constant $C=C(p,d)$: 
$$
\|uv\|_{\dot B^{\frac dp}_{p,1}}\leq C\|u\|_{\dot B^{\frac dp}_{p,1}}\|v\|_{\dot B^{\frac dp}_{p,1}}.
$$
\end{remark}
Let us finally state a composition result.
\begin{proposition}\label{p:comp}
Let $F:\R\rightarrow\R$ be  smooth 
with $F(0)=0.$ 
For  all  $1\leq p,r\leq\infty$ and  $s>0$ we have
$F(u)\in \dot B^s_{p,r}\cap L^\infty$  for  $u\in \dot B^s_{p,r}\cap L^\infty,$  and   
\begin{equation}\label{eq:comp}
\|F(u)\|_{\dot B^s_{p,r}}\leq C\|u\|_{\dot B^s_{p,r}}
\end{equation}
with $C$ depending only on $\|u\|_{L^\infty},$ $F'$ (and higher derivatives),  $s,$ $p$ and $d.$ 
\end{proposition}


\subsection{Endpoint maximal regularity for the linear heat equation}

This paragraph is dedicated to \emph{maximal regularity} issues for  the basic heat equation
\begin{equation}\label{eq:heateq1}
\d_tu-\Delta u=f,\qquad u_{|t=0}=u_0.
\end{equation}
In the case $u_0\equiv0,$  we say that  the  functional space $X$ endowed
with norm $\|\cdot\|_X$  has the maximal regularity property if
\begin{equation}\label{eq:heat}
\|\d_tu,D^2_xu\|_{X}\leq C\|f\|_{X}.
\end{equation}
Fourier-Plancherel theorem implies that  \eqref{eq:heat} holds true for $X=L^2(\R_+\times\R^d).$
{}From more complicated tools based on singular integrals and heat kernel estimates, one may 
gather that  \eqref{eq:heat} is
true for $X=L^r(\R_+;L^p(\R^d))$ whenever $1<p,r<+\infty.$ However, the endpoint cases where one of the exponents $p$ or $r$
is $1$ or $+\infty$ are false. 

One of the keys to the approach presented in these notes  is that \eqref{eq:heat} is true for $X=L^1(\R_+;\dot B^s_{p,1}(\R^d)),$
 a consequence of the following lemma:
\begin{lemma}\label{l:heat} There exist two positive constants $c$ and $C$ such that for any $j\in\Z,$ $p\in[1,\infty]$ and $\lambda\in\R_+,$ we have for all $u\in\cS'$ with $\ddj u$ in $L^p$~:
$$\|e^{\lambda\Delta}\ddj u\|_{L^p}\leq Ce^{-c_0\lambda2^{2j}}\|\ddj u\|_{L^p}.$$
\end{lemma}
\begin{proof}
A  suitable change of  variable
reduces the proof to the case $j=0.$ 
Then consider a  function~$\phi$
in~$\cD(\R^d\setminus\{0\})$ with value~$1$ on a neighborhood of
  $\Supp\varphi.$ We have
$$\begin{array}{lll}
e^{\lambda\Delta}\dot\Delta_0u&=&\cF^{-1}\left(\phi e^{-\lambda|\cdot|^2}
\hat{\dot\Delta_0u}\right)\\
 & = &  g_\lambda \star \dot\Delta_0u\quad\hbox{with}\quad
 g_\lambda(x) \eqdefa
(2\pi)^{-d} \Int  e^{i(x|\xi)}\phi(\xi)
e^{-\lambda|\xi|^2}d\xi.
\end{array}
$$
Integrating  by parts yields
$$
g_\lambda(x)=(1+|x|^2)^{-d}\int_{\R^d}e^{i(x|\xi)}
 (\Id-\Delta_\xi)^d  \Bigl(\phi(\xi)e^{-\lambda|\xi|^2}\Bigr)d\xi.
$$
Combining Leibniz and Fa\'a-di-Bruno's formulae to bound the integrant, we  get
$$|g_\lambda(x))| \leq C(1+|x|^2)^{-d}e^{-c_0\lambda},$$ 
and thus
\begin{equation}\label{eq:g} 
\|g_\lambda\|_{L^1}\leq Ce^{-c_0\lambda}.
\end{equation}
Now the desired inequality (with $j=0$) just follows from   $L^1\star L^p\rightarrow L^p.$ 
\end{proof}
\begin{theorem}\label{th:heat} Let $u$ satisfy \eqref{eq:heateq1}. Then
for any $p\in[1,\infty]$ and $s\in\R$ the following inequality holds true for all $t>0:$
\begin{equation}\label{eq:maxreg1}
\|u(t)\|_{\dot B^s_{p,1}}+\int_0^t\|\nabla^2u\|_{\dot B^s_{p,1}}\,d\tau
\leq C\biggl(\|u_0\|_{\dot B^s_{p,1}}+\int_0^t\|f\|_{\dot B^s_{p,1}}\,d\tau\biggr)\cdotp
\end{equation}
\end{theorem}
\begin{proof}
If $u$ satisfies \eqref{eq:heateq1}  then we have for any $j\in\Z,$
 $$
\ddj u(t)=e^{t\Delta}\ddj u_0+\int_0^te^{(t-\tau)\Delta}\ddj f(\tau)\,d\tau.
$$
Taking advantage of Lemma \ref{l:heat}, we thus have
\begin{equation}\label{eq:heat2}
\|\ddj u(t)\|_{L^p}\lesssim e^{-c_02^{2j}t}\|\ddj u_0\|_{L^p}
+\int_0^t e^{-c_02^{2j}(t-\tau)}\|\ddj f(\tau)\|_{L^p}\,d\tau.
\end{equation}
Multiplying by $2^{js}$ and summing  up over $j$ yields
$$
\sum_j2^{js}\|\ddj u(t)\|_{L^p}\lesssim
\sum_je^{-c_02^{2j}t}2^{js}\|\ddj u_0\|_{L^p}
+\int_0^te^{-c_02^{2j}(t-\tau)}\sum_j\|\ddj f(\tau)\|_{L^p}\,d\tau
$$
whence 
$$
\|u\|_{L^\infty(0,t;\dot B^s_{p,1})}\lesssim
\|u_0\|_{\dot B^s_{p,1}}+\|f\|_{L^1(0,t;\dot B^s_{p,1})}.
$$
Note that integrating \eqref{eq:heat2} with respect to time  also yields
$$
2^{2j}\|\ddj u\|_{L^1(0,t;L^p)}\lesssim \Bigl(1-e^{-c_02^{2j}t}\Bigr)
\Bigl(\|\ddj u_0\|_{L^p}+\|\ddj f\|_{L^1(0,t;L^p)}\Bigr).
$$
Therefore, multiplying by $2^{js},$ using Bernstein inequality and summing up over $j$ yields
\begin{equation}\label{eq:slightlybetter}
\|\nabla^2u\|_{L^1(0,t;\dot B^{s}_{p,1})}\lesssim
\sum_j \Bigl(1-e^{-c_02^{2j}t}\Bigr)2^{js}\Bigl(\|\ddj u_0\|_{L^p}+\|\ddj f\|_{L^1(0,t;L^p)}\Bigr),
\end{equation}
which is  even slightly better than what we wanted to prove
\end{proof}
\begin{remark}\label{r:heat}  
Starting from \eqref{eq:heat2} and using general convolution inequalities in $\R_+$
gives a whole family of estimates for the heat equation. 
However, as time integration has been performed \emph{before}
summation over $j,$ the norms that naturally appear are 
  $$\|\cdot\|_{\tilde L_t^a(\dot B^\sigma_{b,c})}\eqdefa\Bigl\|2^{j\sigma}\|\cdot\|_{L_t^a(L^b)}\Bigr\|_{\ell^c}
  \quad\hbox{where}\quad \|\cdot\|_{L_t^a(L^b)}\eqdefa\|\cdot\|_{L^a((0,t);L^b(\R^d))}.$$
With this notation,  \eqref{eq:heat2} implies that
$$\|u\|_{\tilde L_t^{\rho_1}(\dot B^{s+\frac2{\rho_1}}_{p,r})}\lesssim
\|u_0\|_{\dot B^s_{p,r}}+\|f\|_{\tilde L^{\rho_2}_t(\dot B^{s-2+\frac2{\rho_2}}_{p,r})}\quad\hbox{for }\ 
1\leq\rho_2\leq\rho_1\leq\infty.
$$
The relevancy of the above norms in the maximal regularity estimates
has been first noticed (in a particular case)
in the pioneering work by J.-Y. Chemin and N. Lerner \cite{ChL}, then extended
to general Besov spaces in \cite{Chemin}. They will play a fundamental role in
the proof of decay estimates, at the end of the paper. 

Let us point out that results in the spirit of Propositions \ref{p:op} and \ref{p:comp} 
may be easily proved for $\wt L_t^\rho(\dot B^\sigma_{p,r})$ spaces, the general rule being just that the time exponents 
behave according to H\"older inequality.  
\end{remark}


\subsection{The linear transport equation}

Here we give estimates in Besov spaces for   the following
\emph{transport equation}:
\begin{equation}\label{eq:transport}
\left\{\begin{array}{lll}\d_ta+v\cdot\nabla_x a+\lambda a=f&\quad\hbox{in }&\R\times\R^d\\
a_{|t=0}=a_0&\quad\hbox{in }&\R^d,\end{array}\right. 
\end{equation}
where the initial data  $a_0=a_0(x),$ the source term $f=f(t,x),$
the damping coefficient $\lambda\geq0$  and the time dependent
transport field $v=v(t,x)$  are given. 

Assuming that $a_0\in X$ and   $f\in L^1_{loc}(\R_+;X),$    the relevant assumptions on $v$
for \eqref{eq:transport} to be uniquely solvable depend on the nature of  the Banach space $X.$
Broadly speaking,  in the classical theory based on Cauchy-Lipschitz theorem, 
 $v$ has to be at least  integrable in time with values in the set of Lipschitz functions,  
 so that it  has a flow $\psi.$  This allows to get the following explicit  solution for \eqref{eq:transport}: 
\begin{equation}\label{eq:transportsol}
a(t,x)=e^{-\lambda t}a_0(\psi_t^{-1}(x))+\int_0^te^{-\lambda(t-\tau)}f(\tau,\psi_\tau(\psi_t^{-1}(x)))\,d\tau.
\end{equation}
\begin{theorem}\label{th:transport}
Let $1\leq p\leq p_1\leq\infty,$ $1\leq r\leq\infty$ and $s\in\R$ satisfy
$$-\min\Bigl(\frac d{p_1},\frac d{p'}\Bigr) < s < 1+\frac d{p_1}\cdotp
$$
Then any smooth enough solution to \eqref{eq:transport}
fulfills
$$
\|a\|_{\tilde L_t^\infty(\dot B^s_{p,r})}+
\lambda \|a\|_{\tilde L_t^1(\dot B^s_{p,r})}\leq e^{CV(t)}
\Bigl(\|a_0\|_{\dot B^s_{p,r}}+\|f\|_{\tilde L_t^1(\dot B^s_{p,r})}\Bigr)
$$
with   $$\ds V(t)\eqdefa\int_0^t\|\nabla v(\tau)\|_{\dot B^{\frac d{p_1}}_{p_1,\infty}\cap L^\infty}\,d\tau.$$ 
In the  case $s=1+\frac d{p_1}$ and $r=1,$ the above inequality is true
with $V'(t)=\|\nabla v(t)\|_{\dot B^{\frac d{p_1}}_{p_1,1}}.$
\end{theorem}
\begin{proof} Applying $\ddj$ to \eqref{eq:transport} gives
\begin{equation}\label{eq:transport1}
\d_t\ddj a+v\cdot\nabla\ddj a+\lambda\ddj a=\ddj f+\dot R_j\quad\hbox{with }\ 
\dot R_j\eqdefa[v\cdot\nabla,\ddj]a.
\end{equation}
Therefore, from classical $L^p$ estimates for the transport equation, we get 
\begin{multline}\label{eq:aLp}
\|\ddj a(t)\|_{L^p}+\lambda\|\ddj a\|_{L_t^1(L^p)} \leq \|\ddj a_0\|_{L^p}\\
+\int_0^t\Bigl(\|\ddj f\|_{L^p}
+\|\dot R_j\|_{L^p}+\frac{\|\div v\|_{L^\infty}}p
\|\ddj a\|_{L^p}\Bigr)\,d\tau.
\end{multline}
We claim that the remainder term $\dot R_j$ satisfies
\begin{equation}\label{eq:transport2}
\|\dot R_j(t)\|_{L^p}\leq Cc_j(t)2^{-js}\|\nabla v(t)\|_{\dot B^{\frac d{p_1}}_{p_1,\infty}\cap L^\infty}
\|a(t)\|_{\dot B^s_{p,r}}
\quad\hbox{with }\ \|(c_j(t))\|_{\ell^r}=1.
\end{equation}
Indeed, from Bony's decomposition, we infer that (with the summation convention over repeated indices):
\begin{equation}\label{eq:Rj}
\dot R_j= [T_{v^{k}},\dot\Delta_{j}]\partial_{k}a
+T_{\partial_{k}\dot\Delta_{j}a}v^{k}
-\dot\Delta_{j}T_{\partial_{k}a}v^{k}
+R(v^{k},\partial_k\dot\Delta_{j}a)
-\dot\Delta_{j}R(v^{k},\d_ka).
\end{equation}
The first term is the only one where having a commutator improves the estimate.    
 Indeed, owing to the properties of spectral localization of the Littlewood-Paley decomposition, we have
  $$ [T_{v^{k}},\dot\Delta_{j}]\partial_{k}a=\sum_{|j-j'|\leq 
4}[\dot S_{j'-1}v^{k},\dot\Delta_{j}]\partial_{k}\dot\Delta_{j'}a.$$
Now, remark that
$$
[\dot S_{j'-1}v^{k},\dot\Delta_{j}]\partial_{k}\dot\Delta_{j'}a(x) =  
2^{jd}\!\int_{\R^d}\! h(2^j(x-y))
\bigl(\dot S_{j'-1}v^{k}(x)-\dot S_{j'-1}v^{k}(y)\bigr)\d_k\dot\Delta_{j'} a(y)\,dy.
$$
Hence using   the mean value formula and Bernstein inequalities yields
$$
\| [T_{v^{k}},\dot\Delta_{j}]\partial_{k}a\|_{L^p}\lesssim  2^{-j}\|\dot S_{j'-1}\nabla v\|_{L^\infty}
\sum_{|j'-j|\leq4}\|{\d_k\dot\Delta_{j'}a}\|_{L^p}
\lesssim  \|{\nabla v}\|_{L^\infty}
\sum_{|j'-j|\leq4}\|{\dot\Delta_{j'}a}\|_{L^p}.$$
Bounding the  third and last term in \eqref{eq:Rj}  follows from  Proposition \ref{p:op}.
Regarding the second term, we may write
$$
T_{\partial_{k}\dot\Delta_{j}a}v^{k}=\sum_{j'\geq j-3}\dot S_{j'-1}\partial_k\ddj a\,\dot\Delta_{j'}v^k,
$$
hence using Bernstein inequality, 
$$
\|T_{\partial_{k}\dot\Delta_{j}a}v^{k}\|_{L^p} \lesssim \sum_{j'\geq j-3}
2^{j-j'}\|\ddj a\|_{L^p}\,\|\nabla\dot\Delta_{j'}v\|_{L^\infty},
$$
and  convolution inequality for series thus ensures \eqref{eq:transport2} for that term.
\medbreak
Finally, we have
$$
\ddj R(v^k,\d_ka)=\sum_{|j'-j|\leq1} \d_k\dot\Delta_{j'} a\: (\dot\Delta_{j'-1}+\dot\Delta_{j'}+\dot\Delta_{j'+1})v^k,
$$
whence, by virtue of Bernstein inequality,
$$
\|R(v^k,\d_ka)\|_{L^p} \lesssim  \sum_{|j'-j|\leq1} \|\dot\Delta_{j'} a\|_{L^p} \|\nabla v\|_{L^\infty},
$$
and that term is thus also bounded by the r.h.s. of \eqref{eq:transport2}.
\medbreak
Let us resume to \eqref{eq:transport1}. 
Using \eqref{eq:aLp} and \eqref{eq:transport2},
multiplying by $2^{js}$ then summing up over $j$ and using the notations of Remark \ref{r:heat} yields 
$$ \|a\|_{\tilde L_t^\infty(\dot B^s_{p,r})}+\lambda \|a\|_{\tilde L_t^1(\dot B^s_{p,r})}
\leq \|a_0\|_{\dot B^s_{p,r}}+\|f\|_{\wt L^1_t(\dot B^s_{p,r})}+C\int_0^tV'\|a\|_{\dot B^s_{p,r}}\,d\tau.$$
Then applying Gronwall's lemma gives the desired inequality for $a.$
\end{proof}


\section{The local existence in critical spaces}\label{s:local}

This section is dedicated to solving \eqref{eq:nsbaro} locally in time, in critical spaces.
For simplicity, we focus on the case where the density goes to $1$ at $\infty.$ 
Setting   $a\eqdefa\varrho-1$ and   looking for reasonably smooth solutions with positive density,  System \eqref{eq:nsbaro} is equivalent to 
\begin{equation}\label{eq:nsc}
  \left\{
    \begin{aligned}
      & \d_ta+u\cdot\nabla a=-(1+a)\div u,\\
      &\d_t u + u\cdot\nabla u-\frac{\cA u}{1+a} +\nabla  G(a)
      =\frac1{1+a}\div\bigl(2\wt\mu(a) D(u)+\wt\lambda(a)\div u\:\Id\bigr),
        \end{aligned} \right.
\end{equation}
where\footnote{In our approach, the exact value of functions $G,$ $\wt\lambda$ and $\wt\mu$ will not matter. We shall only  need enough 
smoothness,  and vanishing at $0$ for $\wt\lambda$ and $\wt\mu.$}
  $G'(a)\eqdefa\frac{P'(1+a)}{1+a},$ $\cA\eqdefa\mu\Delta+(\lambda+\mu)\nabla\div$  with $\lambda\eqdefa\lambda(1)$ and $\mu\eqdefa\mu(1),$ 
$$\wt\mu(z)\eqdefa\mu(1+z)-\mu(1)\ \hbox{ and }\ \wt\lambda(z)\eqdefa\lambda(1+z)-\lambda(1).$$ 
The scaling invariance properties for $(a,u)$ are those  pointed out
for  \eqref{eq:nsbaro}. 
Critical norms  for the initial data are thus invariant for all $\ell>0$ by 
$$
a_0(x)\leadsto a_0(\ell x)\quad\hbox{and}\quad 
u_0(x)\leadsto \ell u_0(\ell x).
$$
In all that follows, we shall only consider \emph{homogeneous Besov spaces} 
 having the above scaling invariance and last
index $1.$ There are good reasons for that choice, which will be explained throughout. 
Remembering \eqref{eq:scaling}, we thus  take 
\begin{equation}\label{eq:regularity}
a_0\in \dot B^{\frac d{p}}_{p,1}\quad\hbox{and}\quad
u_0\in \dot B^{\frac d{p}-1}_{p,1}.
\end{equation}
In order to guess what is the relevant solution space, we just  use the fact that
  $a$ is governed by a transport equation and that 
 $u$ may be seen as the solution to  the following   \emph{Lam\'e equation}:
\begin{equation}\label{eq:lame}
\d_tu-\cA u=f,\qquad u|_{t=0}=u_0.
\end{equation}
In the whole space case, the  solutions to \eqref{eq:lame} also satisfy 
the estimates of Theorem \ref{th:heat} and Remark \ref{r:heat}  whenever the following ellipticity 
condition is fulfilled:
\begin{equation}\label{eq:ellipticity}
\mu>0\quad\hbox{and}\quad \nu\eqdefa \lambda+2\mu>0.
\end{equation}
Indeed, if we denote by $\cP$ and $\cQ$ the orthogonal projectors over divergence-free and potential vector fields, 
then we have
$$
\d_t\cP u-\mu\Delta\cP u=\cP f\quad\hbox{and}\quad \d_t\cQ u-\nu\Delta\cQ u=\cQ f.
$$
In particular,  applying Theorem \ref{th:heat} yields for all $t\geq0$:
$$\displaylines{
\|\cP u(t)\|_{\dot B^s_{p,1}}+\mu\int_0^t\|\nabla^2\cP u\|_{\dot B^s_{p,1}}\,d\tau
\leq C\biggl(\|\cP u_0\|_{\dot B^s_{p,1}}+\int_0^t\|\cP f\|_{\dot B^s_{p,1}}\,d\tau\biggr),\cr
\|\cQ u(t)\|_{\dot B^s_{p,1}}+\nu\int_0^t\|\nabla^2\cQ u\|_{\dot B^s_{p,1}}\,d\tau
\leq C\biggl(\|\cQ u_0\|_{\dot B^s_{p,1}}+\int_0^t\|\cQ f\|_{\dot B^s_{p,1}}\,d\tau\biggr)\cdotp}
$$
As  $\cP$ and $\cQ$ are  continuous on  $\dot B^s_{p,1}$ (being
$0$ order multipliers) we conclude that
\begin{equation}\label{eq:maxreglame}
\|u(t)\|_{\dot B^s_{p,1}}+\min(\mu,\nu)\int_0^t\|\nabla^2u\|_{\dot B^s_{p,1}}\,d\tau
\leq C\biggl(\|u_0\|_{\dot B^s_{p,1}}+\int_0^t\|f\|_{\dot B^s_{p,1}}\,d\tau\biggr)\cdotp
\end{equation}
So, in short,  starting from $u_0\in\dot B^{\frac dp-1}_{p,1},$ we expect  
$u$ in \eqref{eq:nsc} to be in  $$
E_p(T)\eqdefa\bigl\{u\in\cC([0,T];\dot B^{\frac dp-1}_{p,1}), 
\ \d_tu,\nabla ^2u\in L^1(0,T;\dot B^{\frac dp-1}_{p,1})\bigr\}\cdotp
$$
In particular $\nabla u$ has exactly the regularity needed in Theorem \ref{th:transport} 
to ensure the conservation of the initial regularity for $a,$ and we thus expect $a\in\cC([0,T];\dot B^{\frac dp}_{p,1}).$
\medbreak
The rest of this section is devoted to  the proof of  the following  statement: 
\begin{theorem}\label{th:small1} Let the viscosity coefficients $\lambda$ and $\mu$ depend smoothly on $\varrho,$ and satisfy \eqref{eq:ellipticity}.
Assume that $(a_0,u_0)$ fulfills \eqref{eq:regularity} for some  $1\leq p<2d,$ and that  $d\geq2.$ If in addition $1+a_0$ is bounded away from $0$ then 
 \eqref{eq:nsc} has  a unique local-in-time 
 solution\footnote{Owing to Remark \ref{r:heat} and to Theorem \ref{th:transport}
 the constructed solution will have the additional property that 
 $u\in\wt L_T^\infty(\dot B^{\frac dp-1}_{p,1})$ and that $a\in\wt L_T^\infty(\dot B^{\frac dp}_{p,1}).$} 
  and  $(a,u)$ with  
$a$ in $\cC([0,T];\dot B^{\frac dp}_{p,1})$ and
$u$ in $E_p(T).$
\end{theorem}
We propose two different proofs for Theorem \ref{th:small1}. 
The first one uses an iterative scheme for approximating \eqref{eq:nsc} which consists
in solving a linear transport equation and the Lam\'e equation with appropriate right-hand sides. 
Taking advantage of  Theorem \ref{th:transport} and of \eqref{eq:maxreglame},  it is easy 
to bound the sequence in the expected solution space on some  fixed time interval $[0,T]$ with $T>0.$
However, because the whole system is not fully parabolic, the strong convergence of the sequence is shown 
 for a weaker norm corresponding to `a loss of one derivative'. For that reason, that
approach works only\footnote{Let us emphasize however that one may modify the iterative scheme for constructing solutions, then resort to compactness arguments
 to get existence for the full range $1\leq p<2d$ and $d\geq2$. }  if $1\leq p<d$ and $d\geq3.$ 
The same restriction occurs as regards the uniqueness issue, although 
the limit case $p=d,$ or $d=2$ and $p\leq2$ is tractable  by taking advantage of a logarithmic interpolation inequality
combined with Osgood lemma (see the end of this section).

\smallbreak
The second proof consists in rewriting \eqref{eq:nsc} in \emph{Lagrangian coordinates.} 
Then the density becomes essentially time independent, and one just has to concentrate on the
velocity equation that is of parabolic type for small enough time, and can thus be solved by the
contracting mapping argument.

\subsection{The classical  proof in Eulerian coordinates}

We here present the direct approach for solving \eqref{eq:nsc}. Our 
proof covers only the case $d\geq3$ and $1\leq p<d$ as regards existence, 
and $1\leq p\leq d$ with  $d\geq2$ for uniqueness (variations on the method would allow to get existence for the full range $1\leq p<2d$ with $d\geq2,$ though). 
To simplify the presentation, we  assume that $\lambda$ and $\mu$ are density independent
so that \eqref{eq:nsc} rewrites
$$
\left\{\begin{array}{l}
\d_ta+u\cdot\nabla a=-(1+a)\div u,\\
\d_tu-\cA u=-u\cdot\nabla u- I(a)\cA u-\nabla(G(a)),\end{array}
\right.
$$
with $I(a)\eqdefa \frac a{1+a}$ and $G'(a)\eqdefa \frac{P'(1+a)}{1+a}\cdotp$
\medbreak
Furthermore, we suppose that for a small enough  constant $c=c(p,d,G),$
\begin{equation}\label{eq:a0}
\|a_0\|_{\dot B^{\frac dp}_{p,1}}\leq c.
\end{equation}

\subsubsection*{Step 1: An iterative scheme} 

We set $a_0^n\eqdefa \dot S_na_0$ and $u_0^n\eqdefa\dot S_nu_0,$ and
define the first term $(a^0,u^0)$ of the sequence of approximate solutions to be
$$
a^0\eqdefa a^0_0\quad\hbox{and}\quad u^0\eqdefa e^{t\cA} u_0^0,
$$
where $(e^{t\cA})_{t\geq0}$ stands for the semi-group of operators associated to 
\eqref{eq:lame}.\medbreak
Next, once $(a^n,u^n)$ has been constructed, we define
$a^{n+1}$ and $u^{n+1}$ to be the solutions to the following \emph{linear} transport and Lam\'e equations:
\begin{equation}\label{eq:nscn}
\left\{\begin{array}{l}
\d_ta^{n+1}+u^n\cdot\nabla a^{n+1}=-(1+a^n)\div u^n,\\
\d_tu^{n+1}-\cA u^{n+1}=-u^n\cdot\nabla u^n- I(a^n)\cA u^n-\nabla(G(a^n)),\end{array}
\right. \end{equation}
supplemented with initial data   $a_0^{n+1}$ and $u_0^{n+1}.$

\subsubsection*{Step 2: Uniform estimates in the case $1\leq p<2d$ and $d\geq2$} 

As the data are smooth, it is not difficult to check (by induction) that
$a^n$ and $u^n$ are smooth and globally defined.
We claim that there exists some $T>0$ such that $(a^n)_{n\in\N}$
is bounded in $\cC([0,T];\dot B^{\frac dp}_{p,1})$ and $(u^n)_{n\in\N}$ is bounded in
the space $E_p(T).$
Indeed, Theorem \ref{th:transport} and the fact  that  $\dot B^{\frac dp}_{p,1}$ is stable by product  imply that  for some $C\geq1,$
$$\displaylines{
\|a^{n+1}(T)\|_{\dot B^{\frac dp}_{p,1}}\leq  \|a_0^{n+1}\|_{\dot B^{\frac dp}_{p,1}}+C\int_0^T
(1+\|a^n\|_{\dot B^{\frac dp}_{p,1}})\|\div u^n\|_{\dot B^{\frac dp}_{p,1}}\,dt\hfill\cr\hfill
+C\int_0^T\|\nabla u^n\|_{\dot B^{\frac dp}_{p,1}}\|a^{n+1}\|_{\dot B^{\frac dp}_{p,1}}\,dt.}
$$
Let  $U^n(T)\eqdefa\|\nabla u^n\|_{L^1_T(\dot B^{\frac dp}_{p,1})}.$
Applying Gronwall's lemma  and  using the definition of $a_0^{n+1},$ we thus get
\begin{equation}\label{eq:an+1}
\|a^{n+1}\|_{L^\infty_T(\dot B^{\frac dp}_{p,1})}\leq Ce^{CU^n(T)}\|a_0\|_{\dot B^{\frac dp}_{p,1}}+
(1+\|a^n\|_{L^\infty_T(\dot B^{\frac dp}_{p,1})})\Bigl(e^{CU^n(T)}-1\Bigr)\cdotp
\end{equation}
Therefore, assuming  that
$a_0$ fulfills \eqref{eq:a0} with some small enough $c,$ that
\begin{equation}\label{eq:an0}
\|a^{n}\|_{L^\infty_T(\dot B^{\frac dp}_{p,1})}\leq 4Cc
\end{equation}
and that 
\begin{equation}\label{eq:Un}
CU^n(T)\leq\log(1+c),
\end{equation}
we conclude from \eqref{eq:an+1} that 
$a^{n+1}$ also satisfies \eqref{eq:an0} \emph{for the same $T$}.
At this point, let us observe that as $\dot B^{\frac dp}_{p,1}$ is continuously embedded
in $L^\infty,$ one may take $c$ so small as 
\begin{equation}\label{eq:pasvide}
\|a^{n}\|_{L^\infty_T(\dot B^{\frac dp}_{p,1})}\leq 4Cc
\quad\hbox{implies}\quad
\|a^{n}\|_{L^\infty([0,T]\times\R^d)}\leq 1/2.
\end{equation}
Let us now prove estimates for the velocity. {}From \eqref{eq:maxreglame},  we get for some
constant $C$ depending only on $\lambda$ and $\mu,$
$$
\|u^{n+1}\|_{E_p(T)}
\leq  C\biggl(\|u_0\|_{\dot B^{\frac dp-1}_{p,1}}+
\int_0^T\|u^n\cdot\nabla u^n+I(a^n)\cA u^n+\nabla(G(a^n))\|_{\dot B^{\frac dp-1}_{p,1}}\,dt\biggr)\cdotp
$$
The terms in the r.h.s. may be bounded by means of Propositions \ref{p:op}
 and \ref{p:comp} (remembering \eqref{eq:pasvide})   if $d\geq2$ and $1\leq p< 2d$. We get for some 
  $C'=C'(p,d,G)$:
$$\displaylines{\|u^{n+1}\|_{E_p(T)}\leq  C'\biggl( \|u_0\|_{\dot B^{\frac dp-1}_{p,1}}
+\bigl(\|a^n\|_{L^\infty_T(\dot B^{\frac dp}_{p,1})}+\|u^n\|_{L^\infty_T(\dot B^{\frac dp-1}_{p,1})}\bigr)\|\nabla u^n\|_{L^1_T(\dot B^{\frac dp}_{p,1})}
\hfill\cr\hfill+T\|a^n\|_{L_T^\infty(\dot B^{\frac dp}_{p,1})}\biggr)\cdotp}
$$
Using \eqref{eq:an0} and the definition of $U^n,$ this implies that
$$
\|u^{n+1}\|_{E_p(T)}\leq  C'\bigl( \|u_0\|_{\dot B^{\frac dp-1}_{p,1}}+(4Cc+U^n(T))\|u^n\|_{E_p(T)}
+4CcT\bigr).
$$
Therefore, assuming that \eqref{eq:Un} is fulfilled and taking smaller $c$ if needed,
we get
$$
\|u^{n+1}\|_{E_p(T)}\leq\frac12\|u^n\|_{E_p(T)}+C'(\|u_0\|_{\dot B^{\frac dp-1}_{p,1}}+4CcT),
$$
and thus, if
\begin{equation}\label{eq:un1}
\|u^n\|_{E_p(T)}\leq 2C'\bigl(\|u_0\|_{\dot B^{\frac dp-1}_{p,1}}+4cCT\bigr)
\end{equation}
then $u^{n+1}$ also satisfies \eqref{eq:un1}.
\medbreak
To complete the proof,  we still have to justify that
\eqref{eq:Un} is fulfilled at rank $n+1.$
From the definition of $\|\cdot\|_{E_p(T)},$ embedding and \eqref{eq:un1} (at rank $n+1$), we know that
there exists some constant $C''$ so that 
$$
U^{n+1}(T)\leq C''\bigl(\|u_0\|_{\dot B^{\frac dp-1}_{p,1}} +cT\bigr).
$$
Hence there exists a constant $c'>0$ such that if $T$ and $u_0$ satisfy
\begin{equation}\label{eq:smallu0}
\|u_0\|_{\dot B^{\frac dp-1}_{p,1}} +cT\leq c'
\end{equation}
then both \eqref{eq:Un} and \eqref{eq:un1} are  fulfilled at rank $n+1.$
\medbreak
If  $\|u_0\|_{\dot B^{\frac dp-1}_{p,1}}\geq c'$  then 
 we split $u^n$ into $u_L^n+\tilde u^n$ with $u_L^n(t)\eqdefa e^{t\cA}u_0^n.$  
   Denoting $u_L\eqdefa e^{t\cA}u_0$ and observing that $u_L^n=\dot S_n u_L,$ we discover that
$$
U^n(T)\leq \|\nabla u^n_L\|_{L^1_T(\dot B^{\frac dp}_{p,1})}
+\|\nabla\tilde u^n\|_{L^1_T(\dot B^{\frac dp}_{p,1})}\leq
 C\|\nabla u_L\|_{L^1_T(\dot B^{\frac dp}_{p,1})}+\|\nabla\tilde u^n\|_{L^1_T(\dot B^{\frac dp}_{p,1})}.
 $$
 The  term with $u_L$ goes  to $0$ for $T$ tending to $0$  with a speed of convergence that may be described according to \eqref{eq:slightlybetter}. 
To handle   the second  term, we  observe that  $\tilde u^{n+1}$ satisfies  
 $$\displaylines{\partial_t\tilde u^{n+1}-\cA\tilde u^{n+1}=-\tilde u^n\cdot\nabla u^n-u_L^n\cdot\nabla\tilde u^n
-u_L^n\cdot\nabla u_L^n-\nabla(G(a^n))-I(a^n)\cA u^n.}$$
Because $\wt u^{n+1}(0)=0,$  combining \eqref{eq:maxreglame}, product laws in Besov spaces and \eqref{eq:a0}, we   get
\begin{multline}\label{eq:ubar}
\|\tilde u^{n+1}\|_{E_p(T)}\leq C\biggl(\int_0^T\|\tilde u^n\|_{\dot B^{\frac dp-1}_{p,1}}\|\nabla u^n\|_{\dot B^{\frac dp}_{p,1}}\,dt+
\int_0^T\!\|u_L\|_{\dot B^{\frac dp}_{p,1}}\|\tilde u^n\|_{\dot B^{\frac dp}_{p,1}}\,dt\\
+\|u_L\|_{L_T^1(\dot B^{\frac dp+1}_{p,1})}\|u_L\|_{L_T^\infty(\dot B^{\frac dp-1}_{p,1})}
+\|a^n\|_{L_T^\infty(\dot B^{{\frac dp}}_{p,1})}\|u^n\|_{L_T^1(\dot B^{{\frac dp}-1}_{p,1})}+T\|a^n\|_{L_T^\infty(\dot B^{{\frac dp}}_{p,1})}\biggr)\cdotp
\end{multline}
Arguing by interpolation yields for any $\beta>0,$
$$
\int_0^T\!\!\|u_L\|_{\dot B^{\frac dp}_{p,1}}\|\tilde u^n\|_{\dot B^{\frac dp}_{p,1}}\,dt
\leq \beta \|u_L\|_{L_T^\infty(\dot B^{\frac dp-1}_{p,1})}
\|\tilde u^n\|_{L_T^1(\dot B^{\frac dp\!+\!1}_{p,1})}+C\beta^{-1}\|u_L\|_{L_T^1(\dot B^{\frac dp\!+\!1}_{p,1})}
\|\tilde u^n\|_{L_T^\infty(\dot B^{\frac dp-1}_{p,1})}.
$$
Besides,  as $ \|u_L\|_{L_T^\infty(\dot B^{\frac dp-1}_{p,1})}\leq C \|u_0\|_{\dot B^{\frac dp-1}_{p,1}},$
Inequality \eqref{eq:ubar} implies that
$$
\displaylines{
\|\tilde u^{n+1}\|_{E_p(T)}\leq C\biggl(\bigl(U^n(T)+\beta \|u_0\|_{\dot B^{\frac dp-1}_{p,1}}+\beta^{-1}\|u_L\|_{L_T^1(\dot B^{\frac dp+1}_{p,1})}
+\|a^n\|_{L^\infty_T(\dot B^{\frac dp})}\bigr)\|\tilde u^{n}\|_{E_p(T)}\hfill\cr\hfill+\|u_0\|_{\dot B^{\frac dp-1}_{p,1}}\|u_L\|_{L_T^1(\dot B^{\frac dp+1}_{p,1})}
+(T+\|u_L\|_{L_T^1(\dot B^{\frac dp+1}_{p,1})})\|a^n\|_{L_T^\infty(\dot B^{{\frac dp}}_{p,1})}\biggr)\cdotp}
$$
Choosing $\beta=1/(4C\|u_0\|_{\dot B^{\frac dp-1}_{p,1}}),$ remembering 
\eqref{eq:an0} and \eqref{eq:Un} (taking $c$ smaller if needed), and that 
\begin{equation}\label{eq:uL}
(1+\|u_0\|_{\dot B^{\frac dp-1}_{p,1}})\|u_L\|_{L_T^1(\dot B^{\frac dp+1}_{p,1})}\leq c,
\end{equation}
we conclude that there exists $C'''$ so that
$$
\|\tilde u^{n+1}\|_{E_p(T)}\leq \frac12\|\tilde u^{n}\|_{E_p(T)}+C'''c.
$$
Hence $\|\tilde u^{n}\|_{E_p(T)}\leq 2C'''c$ implies that $\|\tilde u^{n+1}\|_{E_p(T)}\leq2C'''c,$
and thus \eqref{eq:un1} is fulfilled at rank $n+1$ if $c$ has been chosen small enough. 
\medbreak
Finally, let us notice that there exists some $T>0$ so that 
$$
\sum_j\Bigl(1-e^{-c_02^{2j}T}\Bigr)2^{j(\frac dp-1)}\|\ddj u_0\|_{L^p}\leq \frac{c}{1+\|u_0\|_{\dot B^{\frac dp-1}_{p,1}}}\cdotp
$$
Hence \eqref{eq:slightlybetter} ensures that \eqref{eq:uL} is satisfied for this choice of $T.$

\subsubsection*{Step 3: Convergence in the case $1\leq p<d$ and $d\geq3$}

Let $\da^n\eqdefa a^{n+1}-a^n$ and $\du^n\eqdefa u^{n+1}-u^n.$
The couple $(\da^{n+1},\du^{n+1})$ satisfies 
\begin{equation}\label{nslocconv}
\!\!\left\{
\begin{array}{l}
\partial_t\da^{n+1}+u^{n+1}\cdot\nabla\da^{n+1}=\sum_{i=1}^3\dF_i^n,\\
\partial_t\du^{n+1}-\cA\du^{n+1}=\sum_{i=1}^{5}\dG_i^n,
\end{array}\right.
\end{equation}
$$\displaylines{\hbox{with}\hfill
\dF_1^n\eqdefa-\du^n\cdot\nabla a^{n+1},\qquad
\dF_2^n\eqdefa -\da^n\,\div u^{n+1},\qquad
\dF_3^n\eqdefa -(1+a^n)\div\du^n,\hfill\cr
\dG_1^n\eqdefa\bigl(I(a^{n})-I(a^{n+1})\bigr)\cA u^{n+1},\quad
\dG_2^n\eqdefa -I(a^n)\,\cA\du^n,\quad
\dG_3^n\eqdefa\nabla(G(a^{n})-G(a^{n+1})),\hfill\cr
\dG_4^n\eqdefa-u^{n+1}\cdot\nabla\du^n,\qquad
\dG_5^n\eqdefa-\du^n\cdot\nabla u^n.}
$$
Owing to the first equation,
one can perform estimates for $(\da^n,\du^n)$ only in a space with one less derivative, namely 
in $\cC([0,T];\dot B^{\frac dp-1}_{p,1})\times F_p(T)$ with
$$
F_p(T)\eqdefa\cC([0,T];\dot B^{\frac dp-2}_{p,1})\cap L^1(0,T;\dot B^{\frac dp}_{p,1}).
$$ 
Now,  using the same type of computations  as in Step 3, we get
for all $t\in[0,T],$
$$\displaylines{
\|\da^{n+1}(t)\|_{\dot B^{\frac dp-1}_{p,1}}\leq \|\da^{n+1}(0)\|_{\dot B^{\frac dp-1}_{p,1}}
+C\int_0^t\|\nabla u^{n+1}\|_{\dot B^{\frac dp}_{p,1}}
\bigl(\|\da^n\|_{\dot B^{\frac dp-1}_{p,1}}+\|\da^{n+1}\|_{\dot B^{\frac dp-1}_{p,1}}\bigr)\,d\tau\hfill\cr\hfill
+\int_0^t\bigl(1+C\|a^n\|_{\dot B^{\frac dp}_{p,1}}+C\|a^{n+1}\|_{\dot B^{\frac dp}_{p,1}}\bigr)\|\du^n\|_{\dot B^{\frac dp}_{p,1}}\,d\tau.}
$$
Using the bounds provided by the previous step, we thus get, taking $c$ smaller if needed, 
\begin{equation}\label{eq:dan0}
\|\da^{n+1}\|_{L^\infty_t(\dot B^{\frac dp-1}_{p,1})}\leq \|\da^{n+1}(0)\|_{\dot B^{\frac dp-1}_{p,1}}
+\frac18\|\da^{n}\|_{L^\infty_t(\dot B^{\frac dp-1}_{p,1})}+2\|\du^n\|_{F_p(t)}.
\end{equation}
As in the previous step, bounding $\du^{n+1}$ in $F_p(T)$ follows from \eqref{eq:maxreglame} 
and product laws. 
However, as less regularity is available, one has to make the stronger assumption
\begin{equation}\label{eq:condpd}
d\geq3\quad\hbox{and}\quad 1\leq p<d.
\end{equation}
Taking $c$ smaller if needed,  we eventually get thanks to \eqref{eq:an0}, \eqref{eq:Un}
and \eqref{eq:un1} 
$$
\|\du^{n+1}\|_{F_p(T)}\leq C \|\du^{n+1}(0)\|_{\dot B^{\frac dp-2}_{p,1}}
+\frac18\bigl(\|\da^{n}\|_{L^\infty_T(\dot B^{\frac dp-1}_{p,1})}+\|\du^n\|_{F_p(T)}\bigr).
$$
Combining with \eqref{eq:dan0} yields
$$\displaylines{
\|\da^{n+1}\|_{L^\infty_T(\dot B^{\frac dp-1}_{p,1})} + 4\|\du^{n+1}\|_{F_p(T)}
\leq C\bigl(\|\da^{n+1}(0)\|_{\dot B^{\frac dp-1}_{p,1}} + 4\|\du^{n+1}(0)\|_{\dot B^{\frac dp-2}_{p,1}}\bigr)
\hfill\cr\hfill+\frac58\bigl(\|\da^{n}\|_{L^\infty_T(\dot B^{\frac dp-1}_{p,1})} + 4\|\du^{n}\|_{F_p(T)}\bigr).}
$$
Summing up over $n\in\N,$ we conclude that 
$(a^n-a^0)_{n\in\N}$ and $(u^n-u^0)_{n\in\N}$ converge in $\cC([0,T];\dot B^{\frac dp-1}_{p,1})$ and 
in $F_p(T),$ respectively.

\subsubsection*{Step 4: Checking that the limit is a solution and upgrading regularity}

{}From Step 3, we  know that there exists $a$ and $u$ so that
$$
a^n-a_0\to a-a_0\quad\hbox{in}\quad L^\infty(0,T;\dot B^{\frac dp-1}_{p,1})\quad\hbox{and}\quad
u^n-u_0\to u-u_0\quad\hbox{in}\quad F_p(T).
$$
The bounds of  Step 2  combined with Banach-Alaoglu theorem imply  that in addition
$$a^n\rightharpoonup a\ \hbox{ in }\ L^\infty(0,T;\dot B^{\frac dp}_{p,1})\ \hbox{ weak } *\quad\hbox{and}\quad
u^n\rightharpoonup u\ \hbox{ in }\ L^\infty(0,T;\dot B^{\frac dp-1}_{p,1})\ \hbox{ weak } *.
$$
Routine verifications thus allow to pass to the limit in \eqref{eq:nscn}.

\medbreak
The previous step tells us that  $\nabla^2u$ is in  $L^1(0,T;\dot B^{\frac dp-2}_{p,1}).$ 
To upgrade the regularity exponent by $1,$ let us write that  for all $J\in\N$:
$$\displaylines{
\sum_{|j|\leq J}\int_{0}^T2^{j(\frac dp-1)}\|\ddj\nabla^2u\|_{L^p}\,dt
\leq \sum_{|j|\leq J}\int_{0}^T2^{j(\frac dp-1)}\|\ddj\nabla^2u^n\|_{L^p}\,dt
\hfill\cr\hfill+2^J\sum_{|j|\leq J}\int_{0}^T2^{j(\frac dp-2)}\|\ddj\nabla^2u-\ddj\nabla^2u^n\|_{L^p}\,dt.}
$$
 The first term is  bounded  by the r.h.s. of  \eqref{eq:un1} while, by virtue of Step 3, the second one tends to $0$ when $n$ goes $0.$ 
Hence, letting $J$  tend to $+\infty$ ensures that $\|\nabla^2 u\|_{L^1_T(\dot B^{\frac dp-1}_{p,1})}$ is finite. Next, as 
 $(a,u)$ satisfies \eqref{eq:nsc}, Theorem \ref{th:transport} and Inequality  \eqref{eq:maxreglame} imply that $a\in\wt L^\infty_T(\dot B^{\frac dp}_{p,1})$ and that 
$u\in\wt L^\infty_T(\dot B^{\frac dp-1}_{p,1}),$ which, combined with the fact that 
$a-a_0\in \cC([0,T];\dot B^{\frac dp-1}_{p,1})$ and $u-u_0 \in\cC([0,T];\dot B^{\frac dp-2}_{p,1})$ implies that 
$a\in \cC([0,T];\dot B^{\frac dp}_{p,1})$ and $u\in \cC([0,T];\dot B^{\frac dp-1}_{p,1}).$

\subsubsection*{Step 5: Uniqueness}

Consider two solutions $(a^1,u^1)$ and $(a^2,u^2)$ of \eqref{eq:nsc} with the above
regularity. The difference 
 $(\da,\du)\eqdefa(a^2\!-\!a^1,u^2\!-\!u^1)$
satisfies
\begin{equation}\label{nslocuniq}
\!\!\left\{
\begin{array}{l}
\partial_t\da+u^2\cdot\nabla\da=\sum_{i=1}^3\dF_i,\\
\partial_t\du-\cA\du=\sum_{i=1}^{5}\dG_i,
\end{array}\right.
\end{equation}
$$\displaylines{\hbox{with}\hfill
\dF_1\eqdefa-\du\cdot\nabla a^1,\qquad
\dF_2\eqdefa -\da\,\div u^2,\qquad
\dF_3\eqdefa -(1+a^1)\div\du,\hfill\cr
\dG_1\eqdefa\bigl(I(a^1)-I(a^2)\bigr)\cA u^2,\quad
\dG_2\eqdefa -I(a^1)\,\cA\du,\quad
\dG_3\eqdefa\nabla(G(a^1)-G(a^2)),\hfill\cr
\dG_4\eqdefa-u^2\cdot\nabla\du,\qquad
\dG_5\eqdefa-\du\cdot\nabla u^1.}
$$
Mimicking the computations of Step 3, it is easy to see that if \eqref{eq:condpd} is fulfilled then 
$(\da,\du)\equiv 0$ in $\cC([0,T];\dot B^{\frac dp-1}_{p,1})\times F_p(T).$ 
\smallbreak
 It turns out that the limit case $d=2$ or $p=d$   tractable even though
\begin{equation}\label{eq:uniq1}
\Bigl(\da\in\dot B^0_{d,1}\ \hbox{ and }\ \cA u^2\in\dot B^0_{d,1}\Bigr)
\quad\hbox{implies}\quad \da\cA u^2\in\dot B^{-1}_{d,\infty}\quad\hbox{only}.
\end{equation}
Now, applying Theorem \ref{th:transport} and product laws (see Proposition \ref{p:op}) gives 
$$
\|\da\|_{L^\infty_T(\dot B^{0}_{d,\infty})}\leq\Bigl(
 \|\da(0)\|_{\dot B^{0}_{d,\infty}}+
 (1+\|a^1\|_{L_T^\infty(\dot B^1_{d,1})})\|\du\|_{L_T^1(\dot B^1_{d,1})}\Bigr)
 e^{\|u^2\|_{L_T^1(\dot B^2_{d,1})}}.
$$
Regarding $\du,$ owing to \eqref{eq:uniq1}, one has to apply  Remark \ref{r:heat} 
rather than   Theorem \ref{th:heat}, which enables us to  control  the following quantity:
$$
\|\du\|_{\tilde L_T^1(\dot B^1_{d,\infty})}\eqdefa
\sup_j 2^j\|\ddj\du\|_{L_T^1(L^d)},
$$
which is slightly weaker than
$\|\du\|_{L_T^1(\dot B^1_{d,1})}$. 
\smallbreak
Inserting   the following logarithmic interpolation inequality (see \cite{D2}):
\begin{equation}\label{eq:loginterpo}
\|\du\|_{L_T^1(\dot B^1_{d,1})}\lesssim
\|\du\|_{\tilde L_T^1(\dot B^1_{d,\infty})}
\log\biggl(e+\frac{\|\du\|_{\tilde L_T^1(\dot B^{0}_{d,\infty})}
+\|\du\|_{\tilde L_T^1(\dot B^{2}_{d,\infty})}}
{\|\du\|_{\tilde L_T^1(\dot B^1_{d,\infty})}}\biggr)
\end{equation}
 in the estimate for $\da$ and using Osgood lemma
(see e.g. \cite{BCD}, Chap. 3), we end up  with 
$$
\|\da\|_{L^\infty_t(\dot B^{0}_{d,\infty})}\!+\!
\|\du\|_{L^\infty_t(\dot B^{-1}_{d,\infty})\cap
\tilde L^1_t(\dot B^{1}_{d,\infty})}
\!\lesssim\! \Bigl(\|\da(0)\|_{\dot B^{0}_{d,\infty}}
\!+\!\|\du(0)\|_{\dot B^{-1}_{d,\infty}}\Bigr)^{{\rm exp}(-\int_0^t\alpha\,d\tau)}
$$
where $\alpha$ is in $L^1(0,T)$ and depends  only on the high norms 
of the two solutions. This yields uniqueness on $[0,T].$ 


\subsection{The Lagrangian approach}

We now propose another proof of the local well-posedness of \eqref{eq:nsc}, 
which will provide us with the statement of Theorem \ref{th:small1} in its full generality. 
It is based on the Lagrangian formulation of the system under consideration. 

To make it more precise, we  need to introduce more  notation.
First, we agree that for a $C^1$ function $F:\R^d\rightarrow \R^d\times\R^m$ then 
$\div F:\R^d\rightarrow\R^m$ is defined by
$$
(\div F)^j\eqdefa\sum_{i=1}^d\partial_iF_{ij}\quad\hbox{for }\ 1\leq j\leq m,
$$ 
and that  for $A=(A_{ij})_{1\leq i,j\leq d}$ and  $B=(B_{ij})_{1\leq i,j\leq d}$
 two $d\times d$ matrices, $$A:B=\Tr AB=\sum_{i,j} A_{ij} B_{ji}.$$ 
 The notation $\adj(A)$ designates the
 adjugate matrix that is the transposed cofactor matrix of $A.$ 
 Of course if $A$ is invertible then we have
 $\adj(A)=(\det A)\:A^{-1}.$
 Finally, given some matrix $A,$ we define the ``twisted'' deformation tensor
 and divergence operator (acting on vector fields $z$) by the formulae
 $$
 D_A(z)\eqdefa\frac12\bigl(Dz\cdot A+{}^T\!A\cdot\nabla z\bigr)\quad\hbox{and}\quad
 \divA z\eqdefa{}^T\!A:\nabla z=Dz:A.
  $$
  We recall the following classical result (see the proof in e.g. \cite{D7}).
  \begin{lemma}\label{l:div} Let  $K$ be   a $C^1$ scalar function over
$\R^d$ and $H,$  a $C^1$  vector-field.
  Let $X$ be a $C^1$ diffeomorphism such that  $J\eqdefa\det(D_yX)>0.$ Then we have
\begin{eqnarray}
\label{eq:div2}
&&\overline{\nabla_xK}=J^{-1}\divy(\adj(D_yX)\bar K),\\\label{eq:div1}
&&\overline{\divx  H}=J^{-1}\divy(\adj(D_yX)\bar H).\end{eqnarray}
\end{lemma}
  Let $X$ be the flow associated to the vector-field $u,$ that is
the solution to 
\begin{equation}\label{lag}
X(t,y)=y+\int_0^tu(\tau,X(\tau,y))\,d\tau.
\end{equation}
Let 
$\bar\varrho(t,y)\eqdefa\varrho(t,X(t,y))\ \hbox{ and }\ \bar u(t,y)=u(t,X(t,y)).$
Formally,   we see from  the chain rule and Lemma \ref{l:div} above that $(\varrho,u)$ satisfies \eqref{eq:nsbaro}
if and only if  $(\bar\varrho,\bar u)$  fulfills
 \begin{equation}\label{eq:lagrangian}
\left\{\begin{array}{l}
\d_t(J\bar\varrho)=0\\
\varrho_0\d_t\bar u-\div\Bigl(\adj(DX)\bigl(2\mu(\bar\varrho)D_A(\bar u)
+\lambda(\bar\varrho)\divA\bar u\,\Id+P(\bar\varrho)\Id\bigr)\Bigr)=0\end{array}\right.
\end{equation}
with  $J\eqdefa\det DX,$  $A\eqdefa(D_yX)^{-1}$  and  
\begin{equation}\label{eq:lag}
X(t,y)=y+\int_0^t\bar u(\tau,y)\,d\tau.
\end{equation}
The first equation means that $\varrho=J^{-1}\varrho_0,$ 
and the velocity equation thus recasts in: 
$$
L_{\varrho_0}(\bar u)= \varrho_0^{-1}\div\bigl(I_1(\bar u,\bar u)+I_2(\bar u,\bar u)
+I_3(\bar u,\bar u)+I_4(\bar u)\bigr)
$$
 with 
 \begin{equation}\label{eq:L}
 L_{\varrho_0}(u)\eqdefa\d_t u-\varrho_0^{-1}\div\bigl(2\mu(\varrho_0)D(u)+\lambda(\varrho_0)\div u\,\Id\bigr)
 \end{equation}
 and 
 $$
 \begin{array}{lll}
 I_1(v,w)&\eqdefa&(\adj(DX_v)-\Id)\bigl(\mu(J_v^{-1}\varrho_0)(D w\, A_v+{}^T\!A_v\,\nabla w)
+\lambda(J_v^{-1}\varrho_0)({}^T\!A_v:\nabla w)\Id\bigr)\\
 I_2(v,w)&\eqdefa&(\mu(J_v^{-1}\varrho_0)-\mu(\varrho_0))(D w\, A_v+{}^T\!A_v\,\nabla  w)
 +(\lambda(J_v^{-1}\varrho_0)-\lambda(\varrho_0))({}^T\!A_v:\nabla w)\Id\\
 I_3(v,w)&\eqdefa&\mu(\varrho_0)\bigl(D w(A_v-\Id)+{}^T\!(A_v-\Id)\nabla w\bigr)
 +\lambda(\varrho_0)({}^T\!(A_v-\Id):\nabla w)\Id \\
 I_4(v)&\eqdefa&-\adj(DX_v)P(\varrho_0J_v^{-1}),
 \end{array}
$$
where $X_v$ is given by \eqref{lag} with $v$ instead of $u,$ $A_v\eqdefa (DX_v)^{-1}$
and $J_v\eqdefa\det DX_v.$
\medbreak
So finally,  in order to solve \eqref{eq:lagrangian} locally, it suffices to show that the map 
\begin{equation}\label{eq:Phi}
\Phi: v\longmapsto  u
\end{equation}
with $u$ the solution to 
$$\left\{\begin{array}{l}
L_{\varrho_0}(u)= \varrho_0^{-1}\div\bigl(I_1( v, v)+I_2( v, v)
+I_3(v,v)+I_4(v)\bigr),\\
u|_{t=0}=u_0\end{array}\right.$$
 has a fixed point in $E_p(T)$ for small enough $T.$ 
 \medbreak
In order to treat  the case where $\varrho$  is just bounded away 
from zero,   we need to generalize \eqref{eq:maxreglame} to the following
 Lam\'e system with \emph{nonconstant} coefficients:
\begin{equation}\label{eq:lame-var}
\d_tu-2a\div(\mu D(u))-b\nabla(\lambda\div u)=f,
\end{equation}
where $a,$ $b,$ $\lambda$ and $\mu$ satisfy  the following uniform ellipticity condition: 
\begin{equation}\label{eq:ellipticity1}
\alpha\eqdefa\min\Bigl(\inf_{(t,x)\in[0,T]\times\R^d} (a\mu)(t,x),\inf_{(t,x)\in[0,T]\times\R^d}
(2a\mu+b\lambda)(t,x)\Bigr)>0.
\end{equation}
In \cite{D7}, the following statement has been proved.
\begin{proposition}\label{p:lamerough} 
Let  $a,$  $b,$ $\lambda$ and $\mu$  be  bounded functions  satisfying  \eqref{eq:ellipticity1}. Assume that
$a\nabla \mu,$ $b\nabla\lambda,$ $\mu\nabla a$ and $\lambda\nabla b$ are in $L^\infty(0,T;\dot B^{\frac dp-1}_{p,1})$ for some $1<p<\infty.$
There exist    $\eta>0$ and $\alpha>0$ such that if for some $m\in\Z$ we have 
\begin{align}\label{eq:positivebis}
&\min\Bigl(\inf_{(t,x)\in[0,T]\times\R^d} \dot S_m(2a\mu+b\lambda)(t,x),
\inf_{(t,x)\in[0,T]\times\R^d} \dot S_m(a\mu)(t,x)\Bigr)
\geq\Frac\alpha2,\\
\label{eq:smallabis}
&\|(\Id-\dot S_m)(\mu\nabla a,a\nabla \mu,\lambda\nabla b,b\nabla\lambda)\|_{L^\infty_T(\dot B^{\frac dp-1}_{p,1})}\leq \eta\alpha
\end{align}
then the solutions to \eqref{eq:lame-var} satisfy for all $t\in[0,T],$
$$\displaylines{\quad
\|u\|_{L^\infty_t(\dot B^s_{p,1})}+\alpha\|u\|_{L^1_t(\dot B^{s+2}_{p,1})}
\hfill\cr\hfill\leq C \bigl(\|u_0\|_{\dot B^s_{p,1}}+\|f\|_{L_t^1(\dot B^s_{p,1})}\bigr)
\exp\biggl(\frac C{\alpha}\int_0^t\|\dot S_m(\mu\nabla a,a\nabla \mu,\lambda\nabla b,b\nabla\lambda)\|^2_{\dot B^{\frac dp}_{p,1}}\,d\tau\biggr)\quad}
$$
whenever $-\min(d/p,d/p')<s\leq d/p-1.$
\end{proposition}

In order to show that $\Phi$ in \eqref{eq:Phi} admits a fixed point in $E_p(T),$ 
we introduce, as in the previous subsection,  the solution $u_L$ 
in $E_p(T)$  to
 $$
 L_1 u_L=0,\qquad  u|_{t=0}=u_0.
 $$
 We want to apply Banach fixed point theorem to $\Phi$  in some suitable 
 closed ball $\bar B_{E_p(T)}(u_L,R).$ 
 Let $v$ be in  $\bar B_{E_p(T)}(u_L,R)$ and $u\eqdefa\Phi(v).$  
 Denoting $\tilde u\eqdefa u-u_L,$ we see that 
\begin{equation}\label{eq:tildeu}
 \left\{\begin{array}{l}
 L_{\varrho_0}\tilde u=\varrho_0^{-1}\div\bigl(I_1(v,v)+I_2(v,v)
 +I_3(v,v)+I_4(v)\bigr)+(L_1-L_{\varrho_0})u_L,\\
 \tilde u|_{t=0}=0.\end{array}\right.
 \end{equation}
The existence of  some $m\in\Z$ so that 
 $$
\displaylines{\min\biggl(\inf_{\R^d}\dot S_m\Bigl(2\frac{\mu(\varrho_0)}{\varrho_0}+\frac{\lambda(\varrho_0)}{\varrho_0}\Bigr),
\inf_{\R^d}\dot S_m\Bigl(\frac{\mu(\varrho_0)}{\varrho_0}\Bigr)\biggr)>\frac\alpha2\cr
\hbox{and}\quad
\Bigl\|(\Id-\dot S_m)\Bigl(\frac{\mu(\varrho_0)}{\varrho_0^2}\nabla\varrho_0,\frac{\mu'(\varrho_0)}{\varrho_0}\nabla\varrho_0,\frac{\lambda(\varrho_0)}{\varrho_0^2}\nabla\varrho_0,\frac{\lambda'(\varrho_0)}{\varrho_0}\nabla\varrho_0\Bigr)\Bigr\|_{\dot B^{\frac dp-1}_{p,1}}\leq\eta\alpha}
$$
 is ensured by the fact that all the coefficients (minus some constant) 
 belong to the space $\dot B^{\frac dp}_{p,1}$  which is defined in terms of  a convergent series
 and embeds continuously in the set of bounded continuous functions.
 Hence, if one can show that  the right-hand side of \eqref{eq:tildeu} is 
 in  $L^1(0,T;\dot B^{\frac dp-1}_{p,1})$ (which will be carried out in the next step) 
 then we will be allowed to apply 
  Proposition \ref{p:lamerough}   to bound  $\tilde u$ in $E_p(T).$


\subsubsection*{First step: Stability of $\bar B_{E_p(T)}(u_L,R)$}
Let $v\in \bar B_{E_p(T)}(u_L,R)$ and $\wt u$ be given by \eqref{eq:tildeu}. 
Let $a_0\eqdefa\varrho_0-1.$  Proposition \ref{p:lamerough}, product laws in Besov spaces and Proposition \ref{p:comp} imply that 
 \begin{multline}\label{eq:ut1}
 \|\tilde u\|_{E_p(T)}\leq Ce^{C_{\varrho_0}T}
\Bigl(\|(L_1-L_{\varrho_0})u_L\|_{L_T^1(\dot B^{\frac dp-1}_{p,1})}\\
+\bigl(1+\|a_0\|_{\dot B^{\frac dp}_{p,1}}\bigr)
\bigl(\|I_4(v)\|_{L_T^1(\dot B^{\frac dp}_{p,1})}+\sum_{i=1}^3 \|I_i(v,v)\|_{L_T^1(\dot B^{\frac dp}_{p,1})}  \bigr)\!\Bigr)
 \end{multline}
 for some constant $C_{\varrho_0}$ depending only on $\varrho_0.$
\medbreak
In what follows, we assume that $T$ and $R$ have been chosen so that, 
for a small enough positive constant $c,$
\begin{equation}\label{eq:smallv}
\int_0^T\|\nabla v\|_{\dot B^{\frac dp}_{p,1}}\,dt\leq c.
\end{equation}
Now, using the decomposition
$$\displaylines{\quad
(L_1-L_{\varrho_0})u_L=(\varrho_0^{-1}-1)\div\bigl(2\mu(\varrho_0)D(u_L)
+\lambda(\varrho_0)\div u_L\,\Id\bigr)\hfill\cr\hfill
+\div\bigl(2(\mu(\varrho_0)-\mu(1))D(u)+(\lambda(\varrho_0)-\lambda(1))\div u\,\Id\bigr),\quad}
$$
and Proposition \ref{p:comp},  we see that 
$(L_1-L_{\varrho_0})u_L\in L^1(0,T;\dot B^{\frac dp-1}_{p,1})$ and 
\begin{equation}\label{eq:L1}
\|(L_1-L_{\varrho_0})u_L\|_{L_T^1(\dot B^{\frac dp-1}_{p,1})}\lesssim 
\|a_0\|_{\dot B^{\frac dp}_{p,1}}(1+\|a_0\|_{\dot B^{\frac dp}_{p,1}})
\|Du_L\|_{L_T^1(\dot B^{\frac dp}_{p,1})}.
\end{equation}
Likewise, flow and composition estimates (see the appendix) ensure  that
\begin{align}\label{eq:II2}
\|I_i(v,w)\|_{L_T^1(\dot B^{\frac dp-1}_{p,1})}&\lesssim
(1+\|a_0\|_{\dot B^{\frac dp}_{p,1}})\|Dv\|_{L_T^1(\dot B^{\frac dp}_{p,1})}
\|Dw\|_{L_T^1(\dot B^{\frac dp}_{p,1})}\quad\hbox{for $i=1,2,3,$}\\\label{eq:II5}
\|I_4(v)\|_{L_T^1(\dot B^{\frac dp}_{p,1})}&\lesssim T(1+\|a_0\|_{\dot B^{\frac dp}_{p,1}})(1+\|Dv\|_{L_T^1(\dot B^{\frac dp}_{p,1})}).
\end{align}
So plugging the above inequalities in \eqref{eq:ut1} and keeping in mind
that  $v$  satisfies \eqref{eq:smallv}, we get after decomposing
$v$ into $\tilde v+u_L$: 
$$
\displaylines{\|\tilde u\|_{E_p(T)}\leq Ce^{C_{\varrho_0,m}T}
(1+\|a_0\|_{\dot B^{\frac dp}_{p,1}})^2\Bigl((T+\|a_0\|_{\dot B^{\frac dp}_{p,1}}\|Du_L\|_{L_T^1(\dot B^{\frac dp}_{p,1})})
\hfill\cr\hfill+\|Du_L\|_{L_T^1(\dot B^{\frac dp}_{p,1})}^2
+\bigl(\|Du_L\|_{L_T^1(\dot B^{\frac dp}_{p,1})}+\|D\tilde v\|_{L_T^1(\dot B^{\frac dp}_{p,1})}\bigr)\|D\tilde v\|_{L_T^1(\dot B^{\frac dp}_{p,1})}\Bigr).}
$$
Now, because $\tilde v\in\bar B_{E_p(T)}(0,R),$
$$
\displaylines{
\|\tilde u\|_{E_p(T)}\leq Ce^{C_{\varrho_0}T}(1+\|a_0\|_{\dot B^{\frac dp}_{p,1}})^2
\Bigl((T+\|a_0\|_{\dot B^{\frac dp}_{p,1}}\|Du_L\|_{L_T^1(\dot B^{\frac dp}_{p,1})})\hfill\cr\hfill
+(R+\|Du_L\|_{L_T^1(\dot B^{\frac dp}_{p,1})})\|Du_L\|_{L_T^1(\dot B^{\frac dp}_{p,1})}+R^2\Bigr).}
$$
Therefore, if we first choose $R$ so that for a small enough constant $\eta,$
\begin{equation}\label{eq:eta}(1+\|a_0\|_{\dot B^{\frac dp}_{p,1}})^2R\leq \eta
\end{equation}
 and then take $T$ so that 
\begin{equation}\label{eq:smalltime}
C_{\varrho_0}T\leq\log2,\quad T\leq R^2,\quad
\|a_0\|_{\dot B^{\frac dp}_{p,1}}\|Du_L\|_{L_T^1(\dot B^{\frac dp}_{p,1})}\leq R^2,\quad
\|Du_L\|_{L_T^1(\dot B^{\frac dp}_{p,1})}\leq R,
\end{equation}
then we may conclude that $\Phi$ maps $\bar B_{E_p(T)}(u_L,R)$ into itself.


\subsubsection*{Second  step: contraction estimates}
Let us now establish that, under Condition \eqref{eq:smalltime}, the map 
$\Phi$ is contractive. 
We consider two vector-fields $v^1$ and $v^2$ in $\bar B_{E_p(T)}(u_L,R),$ and 
set $u^1\eqdefa\Phi(v^1)$ and $u^2\eqdefa\Phi(v^2).$ Let $\du\eqdefa u^2-u^1$ and $\dv\eqdefa v^2-v^1.$
We have 
$$\displaylines{
L_{\varrho_0}\du=\varrho_0^{-1}\div\Bigl((I_1(v^2,v^2)-I_1(v^1,v^1))\hfill\cr\hfill+(I_2(v^2,v^2)-I_2(v^1,v^1)) +(I_3(v^2,v^2)-I_3(v^1,v^1)) +(I_4(v^2)-I_4(v^1))
\Bigr)\cdotp}$$
So  applying Proposition \ref{p:lamerough} (recall that $C_{\varrho_0}T\leq\log2$), we get
 \begin{multline}\label{eq:ut2}
 \!\!\!\!\!\|\du\|_{E_p(T)}\leq C(1+\|a_0\|_{\dot B^{\frac dp}_{p,1}})
 \Bigl(\sum_{i=1}^3\|I_i(v^2,v^2)-I_i(v^1,v^1)\|_{L_T^1(\dot B^{\frac dp}_{p,1})}\\
 +\|I_4(v^2)-I_4(v^1)\|_{L_T^1(\dot B^{\frac dp}_{p,1})}\Bigr)\cdotp
 \end{multline}
In order to deal with the first  term of the right-hand side, we use the decomposition
$$
\begin{array}{lll}I_1(v^2,v^2)-I_1(v^1,v^1)&=&\lambda(J_{v^2}^{-1}\varrho_0)\bigl({}^T\!A_{v^2}:\nabla v^2\bigr)\bigl(\adj(DX_{v^2})-\adj(DX_{v^1})\bigr)\\
&+&\bigl(\adj(DX_{v^1})-\Id\bigr)\bigl(\lambda(J_{v^2}^{-1}\varrho_0)-\lambda(J_{v^1}^{-1}\varrho_0)\bigr)\bigl({}^T\!A_{v^2}:\nabla v^2\bigr)\\
&+&\bigl(\adj(DX_{v^1})-\Id\bigr)\lambda(J_{v^1}^{-1}\varrho_0)\bigl(({}^T\!A_{v^2}-{}^T\!A_{v^1}):\nabla v^1+{}^T\!A_{v^2}:\nabla\dv\bigr)\\
&+&\hbox{terms pertaining to }\ \mu.\end{array}
$$
Taking advantage of product laws in Besov spaces, of  Proposition \ref{p:comp} 
and of the flow estimates in the appendix, we deduce that for some constant $C_{\varrho_0}$
depending only on~$\varrho_0$:
$$
\|I_1(v^2,v^2)-I_1(v^1,v^1)\|_{L_T^1(\dot B^{\frac dp}_{p,1})}\leq C_{\varrho_0}\|(Dv^1,Dv^2)\|_{L_T^1(\dot B^{\frac dp}_{p,1})}
\|D\dv\|_{L_T^1(\dot B^{\frac dp}_{p,1})}.
$$
Similar estimates may be proved for the next two terms 
 of the right-hand side of \eqref{eq:ut2}. Concerning the last one, we use the decomposition
 $$
 I_4(v^2)-I_4(v^1)\!=\!\bigl(\adj(DX_{v^1})-\adj(DX_{v^2})\bigr)\!P(J_{v^2}^{-1}\varrho_0)-\adj(DX_{v^1})\!
 \bigl(P(J_{v^2}^{-1}\varrho_0)-P(J_{v^1}^{-1}\varrho_0)\bigr).
 $$
 Hence
 $$
 \|I_4(v^2)-I_4(v^1)\|_{L_T^1(\dot B^{\frac dp}_{p,1})}\leq C(1+\|a_0\|_{\dot B^{\frac dp}_{p,1}})T\|D\dv\|_{L_T^1(\dot B^{\frac dp}_{p,1})}.
 $$
 We end up with 
$$
\displaylines{
 \|\du\|_{E_p(T)}\leq C(1+\|a_0\|_{\dot B^{\frac dp}_{p,1}})^2\bigl(T+\|(Dv^1,Dv^2)\|_{L_T^1(\dot B^{\frac dp}_{p,1})}
 \bigr)\|D\dv\|_{L_T^1(\dot B^{\frac dp}_{p,1})}.
}$$
Given that $v^1$ and $v^2$ are in $\bar B_{E_p(T)}(u_L,R),$
our hypotheses over $T$ and $R$ (with  smaller $\eta$ in \eqref{eq:eta} if need be) thus ensure that, 
$$
 \|\du\|_{E_p(T)}\leq \frac12 \|\dv\|_{E_p(T)}.
 $$
Hence  $\Phi$ admits a unique fixed point in $\bar B_{E_p(T)}(u_L,R).$


\subsubsection*{Third step: Regularity of the density}

Set  $\varrho\eqdefa J_u^{-1}\varrho_0.$
By construction $(\varrho,u)$ satisfies
\eqref{eq:lagrangian} and  $a\eqdefa\varrho-1$ is given by  
$$
a=(J_u^{-1}-1)a_0+a_0.
$$
{}From the appendix,  as $Du\in L^1(0,T;\dot B^{\frac dp}_{p,1}),$  we have  $J_u^{-1}-1$
 belongs to $\cC([0,T];\dot B^{\frac dp}_{p,1}).$
Hence $a$  is in $\cC([0,T];\dot B^{\frac dp}_{p,1}),$ too.
Because $\dot B^{\frac dp}_{p,1}$ is continuously embedded in $L^\infty,$    the density remains 
bounded away from $0$ on $[0,T]$ (taking $T$ smaller if needed).


\subsubsection*{Last step: Uniqueness and continuity of the flow map}

Let the data $(\varrho_0^1,u_0^1)$ and $(\varrho_0^2,u_0^2)$ fulfill  the assumptions
of Theorem \ref{th:small1}, and let  $(\varrho^1,u^1)$ and $(\varrho^2,u^2)$  be 
the corresponding solutions. Setting $\du\eqdefa u^2-u^1,$ we see that
$$\displaylines{
L_{\varrho_0^1}(\du)=(L_{\varrho_0^1}-L_{\varrho_0^2})(u^2)
+(\varrho_0^1)^{-1}\div\Bigl(\sum_{j=1}^3\bigl((I_j^2(u^2,u^2)-I_j^2(u^1,u^1)\bigr)+(I_4^2(u^2)-I_4^2(u^1))\Bigr)\hfill\cr\hfill
+(\varrho_0^1)^{-1}\div\Bigl(\sum_{j=1}^3((I_j^2-I_j^1)(u^1,u^1)+(I_4^2-I_4^1)(u^1)\Bigr),}$$
where $I_1^i,$ $I_2^i,$ $I_3^i$ and $I_4^i$ correspond to the quantities that have been defined just above \eqref{eq:Phi}, 
with density $\varrho_0^i.$ Compared to the  second step, the only definitely new terms are  $(L_{\varrho_0^1}-L_{\varrho_0^2})(u^2)$ 
and the last line. As regards $(L_{\varrho_0^1}-L_{\varrho_0^2})(u^2),$ we have for $t\leq T,$
$$
\|(L_{\varrho_0^1}-L_{\varrho_0^2})(u^2)\|_{L_t^1(\dot B^{\frac dp-1}_{p,1})}
\leq C_{\varrho_0^1,\varrho_0^2}\|\dr_0\|_{\dot B^{\frac dp}_{p,1}} \|Du^2\|_{L_t^1(\dot B^{\frac dp}_{p,1})}.
$$
The other new terms satisfy analogous estimates. 
Hence, applying Proposition \ref{p:lamerough} 
yields if  $\dr_0$ is small enough: $$
 \displaylines{
 \|\du\|_{E_p(t)}\leq C_{\varrho_0^1}\bigl((t+\|Du^1\|_{L_t^1(\dot B^{\frac dp}_{p,1})}
 +\|\du\|_{E_p(t)})\|\du\|_{E_p(t)}
 \hfill\cr\hfill+\|\du_0\|_{\dot B^{\frac dp}_{p,1}}+\|\dr_0\|_{\dot B^{\frac dp}_{p,1}}(t+ \|Du^1\|_{L_t^1(\dot B^{\frac dp}_{p,1})})\bigr).}
$$
 An obvious bootstrap argument thus shows that if $t,$  $\du_0$ and $\dr_0$ are small enough then 
 $$
  \|\du\|_{E_p(t)}\leq 2C_{\varrho_0}\bigl(\|\du_0\|_{\dot B^{\frac dp}_{p,1}}+\|\dr_0\|_{\dot B^{\frac dp}_{p,1}}\bigr).
  $$
    As regards the density, we have
$\da=J_{u^1}^{-1} \da_0+(J_{u^2}^{-1}-J_{u^1}^{-1})a_0^2.$
 Hence for all $t\in[0,T],$ 
 $$
 \|\da(t)\|_{\dot B^{\frac dp}_{p,1}}
 \leq C(1+\|Du^1\|_{L_t^1(\dot B^{\frac dp}_{p,1})})\|\da_0\|_{\dot B^{\frac dp}_{p,1}}\|D\du\|_{L_t^1(\dot B^{\frac dp}_{p,1})}.
 $$
So we  get  uniqueness and continuity of the flow map on a small  time interval. 
  Then iterating the proof yields uniqueness on the initial time interval $[0,T],$ as well as
  Lipschitz continuity of the flow map.

 It is now easy to conclude to  Theorem \ref{th:small1} in its full generality, as
a mere corollary of the following proposition which states the equivalence
of  Systems \eqref{eq:nsc} and \eqref{eq:lagrangian} in our functional  setting (see the proof in \cite{D7}).
 \begin{proposition}\label{p:equiv}
  Let $1\leq p< 2d.$ Assume that the couple $(\varrho,u)$ with $(\varrho-1)\in\cC([0,T];\dot B^{\frac dp}_{p,1})$
and $u\in E_p(T)$ is a solution to \eqref{eq:nsc} such that
\begin{equation}\label{eq:smallu}
\int_0^T\|\nabla u\|_{\dot B^{\frac dp}_{p,1}}\,dt\leq c.
\end{equation}  
Let $X$ be the flow of $u$ defined in \eqref{lag}. 
Then the couple $(\bar\varrho,\bar u)\eqdefa(\varrho\circ X,u\circ X)$ belongs
to the same functional space as $(\varrho,u),$ and satisfies \eqref{eq:lagrangian}. 
\medbreak
Conversely, if  $(\bar\varrho-1,\bar u)$ belongs to 
$\cC([0,T];\dot B^{\frac dp}_{p,1})\times E_p(T)$ and $(\bar\varrho,\bar u)$ satisfies \eqref{eq:lagrangian} 
and, for a small enough constant $c,$
\begin{equation}\label{eq:smallbaru}
\int_0^T\|\nabla \bar u\|_{\dot B^{\frac dp}_{p,1}}\,dt\leq c
\end{equation} then 
the map
$X_t\eqdefa X(t,\cdot)$ defined in \eqref{eq:lag} is a $ C^1$ diffeomorphism on $\R^d$
and the couple $(\varrho,u)(t)\eqdefa(\bar\varrho(t)\circ X_t^{-1},\bar u(t)\circ X_t^{-1})$ satisfies
\eqref{eq:nsc} and has the same regularity as $(\bar\varrho,\bar u).$
\end{proposition}


\section{The global existence issue}\label{s:global}

This section is devoted to the proof of global existence 
of strong solutions for small perturbations of the constant state $(\varrho,u)=(1,0),$ 
 under the stability assumption $P'(1)>0.$ For simplicity,  we  assume that
the viscosity functions $\lambda$ and $\mu$ are constant. 
 
 Let us emphasize that the approach  we used so far to solve \eqref{eq:nsc}
 cannot provide us with  global-in-time estimates  (even if both $a_0$ and $u_0$ are small)
 because we completely ignored the coupling between the mass and momentum 
 equation through the pressure term and looked at it as a low order source term, 
 just writing 
  $$
\d_tu-\cA u=-u\cdot\nabla u-I(a)\cA u-\nabla(G(a)).
$$ 
Then, applying   Inequality \eqref{eq:maxreglame} and product laws in Besov spaces led to 
\begin{equation}\label{eq:uu}
\|u\|_{E_p(t)}\leq C\biggl(\|u_0\|_{\dot B^{\frac dp-1}_{p,1}}
+\|u\|_{E_p(t)}\bigl(\|a\|_{L^\infty_t(\dot B^{\frac dp}_{p,1})}+\|u\|_{E_p(t)}\bigr)
+\Int_0^t\|a\|_{\dot B^{\frac dp}_{p,1}}\,d\tau\biggr)\cdotp
\end{equation}
At the same time, as $a$ is  a solution to a transport equation, 
we can only get   bounds on  
$\|a\|_{L_t^\infty(\dot B^{\frac dp}_{p,1})}$ and  the last term
 of \eqref{eq:uu} is thus out of control for  $t\to+\infty.$


\subsection{The linearized compressible Navier-Stokes system, and main result}

The key to proving global results is a refined analysis of the linearized system \eqref{eq:nsc}   about  $(a,u)=(0,0)$
taking the coupling between the mass and momentum equation through the pressure term  into account. 
The system in question reads: 
\begin{equation}\label{eq:LPH}
\left\{\begin{array}{l}\d_ta+\div u=f,\\
\d_tu-\mu\Delta u-(\lambda+\mu)\nabla\div u+P'(1)\nabla a=g.
\end{array}\right.
\end{equation}
 Applying the orthogonal projectors    $\cP$ and $\cQ$  over divergence-free and potential vector-fields, respectively, 
 to the second equation,
and setting $\alpha\eqdefa P'(1)$ and $\nu\eqdefa\lambda+2\mu$, System \eqref{eq:LPH} translates into
\begin{equation}\label{eq:LPHa}
\left\{\begin{array}{l}
\d_ta+\div\cQ u=f,\\
\d_t \cQ u-  \nu\Delta \cQ u+\alpha\nabla a=\cQ g,\\
\d_t \cP u-\mu \Delta \cP u=\cP g.
\end{array}\right.
\end{equation}
We see that $\cP u$ satisfies an ordinary heat equation, which is  uncoupled from $a$ and $\cQ u.$
For studying  the coupling between $a$ and $\cQ u,$ it
is convenient to set $v\eqdefa|D|^{-1}\div u$ (with $\cF(|D|^s u)(\xi)\eqdefa |\xi|^s\wh u(\xi)$), 
keeping in mind  that, according  to  \eqref{eq:fm},  bounding $v$ or $\cQ u$  is equivalent,
as one can go from $v$ to $\cQ u$ or from $\cQ u$ to $v$ by means of
a $0$ order homogeneous Fourier multiplier. 

For notational simplicity, we assume from now on 
that\footnote{Which is not restrictive as  the rescaling
\begin{equation}\label{eq:change}\textstyle
a(t,x)= \tilde a\bigl(\frac\alpha\nu\, t, \frac{\sqrt{\alpha}}\nu\: x\bigr) \quad\hbox{and}\quad
 u(t,x)=\sqrt{\alpha}\:\tilde u \bigl(\frac\alpha\nu\, t, \frac{\sqrt{\alpha}}\nu\: x\bigr) \end{equation}
 ensures that $(\wt a,\wt u)$ satisfies \eqref{eq:LPHa} with $\alpha=\nu=1.$}  $\alpha=\nu=1.$
Hence  $(a,v)$ satisfies the following $2\times 2$ system:
\begin{equation}\label{eq:nscl}
\left\{\begin{array}{l}
\d_ta+|D| v=f,\\
\d_tv-\Delta v-|D|a=h\eqdefa|D|^{-1}\div g.
\end{array}\right.\end{equation}
Taking the Fourier transform with respect to $x,$ and denoting  $\rho\eqdefa|\xi|$ with $\xi\in\R^d$ the 
Fourier variable, System  \eqref{eq:nscl} translates  into 
\begin{equation}\label{eq:nsclfourier}
\frac d{dt}\left(\begin{array}{c}\hat{a}\\\hat{v}\end{array}\right)=
M_\rho\left(\begin{array}{c}\hat{a}\\\hat{v}\end{array}\right)+\left(\begin{array}{c}\hat{f}\\\hat{h}\end{array}\right)  \quad\hbox{with}\quad
M_\rho\eqdefa \left(\begin{array}{cc}0&-\rho\\ \rho&-\rho^2\end{array}\right)\cdotp
\end{equation}
\begin{itemize}
\item In the low frequency regime $\rho<2,$ $M_\rho$ has  two \emph{complex conjugated} eigenvalues:
$$
\lambda_\pm(\rho)\eqdefa -\frac{\rho^2}{2}(1\pm i S(\rho))\quad\hbox{with}\quad 
 S(\rho)\eqdefa\sqrt{\frac{4}{\rho^2}-1}
$$
which have real part $-\rho^2/2,$ exactly as for the heat equation with diffusion $1/2.$
\item  In the  high frequency regime $\rho>2,$ there are two distinct real eigenvalues:
$$
\lambda_\pm(\rho)\eqdefa -\frac{\rho^2}{2}(1\pm R(\rho))\quad\hbox{with}\quad  R(\rho)\eqdefa\sqrt{1-\frac{4}{\rho^2}}\,\cdotp
$$
 As $1-R(\rho)\sim  2/\rho^2$ for $\rho\to+\infty,$ we have 
 $\lambda_+(\rho)\sim -\rho^2$ and $\lambda_-(\rho)\sim -1\cdotp$
In other words, a parabolic and a damped mode coexist. 
\end{itemize}
Optimal a priori estimates may be easily derived by computing 
the explicit solution of  \eqref{eq:nscl} explicitly in the Fourier space. 
Below, we present an alternative method which is  generalizable to much more complicated systems where explicit computations are no longer possible (see e.g. \cite{DD}).

Fix some $\rho\geq0$ and consider the corresponding solution $(A,V)$ of  \eqref{eq:nsclfourier}  in the case $\wh f=\wh h=0$. 
We easily get the following three identities:
\begin{eqnarray}
&&\frac12\frac d{dt}|A|^2+\rho \Re(A\bar V)=0,\\
&&\frac12\frac d{dt}|V|^2+\rho^2|V|^2-\rho\Re(A\bar V)=0,\\
&&\frac d{dt} \Re(A\bar V)+\rho|V|^2-\rho|A|^2+\rho^2\Re(A\bar V)=0,
\end{eqnarray}
from which we deduce
\begin{equation}\label{eq:Lrho0}
\frac12\frac d{dt}\cL_\rho^2+\rho^2|(A,V)|^2=0
\quad\hbox{with}\quad
\cL_\rho^2\eqdefa 2|(A,V)|^2+|\rho A|^2-2\rho\Re(A\bar V).
\end{equation}
Using Young inequality, we discover that there exists some constant $C_0>0$ independent of $\rho$ so that
\begin{equation}\label{eq:equivLrho}
C_0^{-1}\cL_\rho^2\leq |(A,\rho A,V)|^2\leq C_0\cL_\rho^2.
\end{equation}
Combining with \eqref{eq:Lrho0}, we conclude that there exists a universal constant $c_0>0$
so that 
\begin{equation}\label{eq:Lrho}
\cL_\rho^2(t)\leq e^{-c_0\min(1,\rho^2)t}\cL_\rho^2(0)\quad\hbox{for all }\  t\geq0.
\end{equation}
In the case of general  source terms $\wh f$ and $\wh h$ in \eqref{eq:nsclfourier}, 
using Duhamel formula and applying the above inequality to 
$A(t)=\wh a(t,\rho)$ and $V(t)=\wh v(t,\rho)$ thus leads to 
\begin{equation}\label{eq:bhf}
|(\wh a,\rho\wh a,\wh v)(t)|+\min(1,\rho^2)\!\int_0^t\! |(\wh a,\rho\wh a,\wh v)|\,d\tau\leq
C\biggl(|(\wh a_0,\rho\wh a_0,\wh v_0)|+\!\int_0^t \! |(\wh f,\rho\wh f,\wh h)|\,d\tau\biggr)\cdotp
\end{equation}
Note that as 
$$
\d_t\wh v+\rho^2\wh v= \wh h+\rho\wh a,
$$
we also have 
$$
\rho^2\int_0^t|\wh v(\tau)|\,d\tau\leq |\wh v_0|+\int_0^t|\wh h(\tau)|\,d\tau+\int_0^t|\rho \wh a(\tau)|\,d\tau.
$$
And thus, bounding the last term  according to \eqref{eq:bhf}, 
we get  the following  inequality which provides the full parabolic smoothing 
for $v$:
\begin{multline}\label{eq:bhfbis}
|(\wh a,\rho\wh a,\wh v)(t)|+\min(\rho,\rho^2)\int_0^t |\wh a|\,d\tau
+\rho^2\int_0^t|\wh v|\,d\tau\\\leq
C\biggl(|(\wh a_0,\rho\wh a_0,\wh v_0)|+\int_0^t  |(\wh f,\rho\wh f,\wh h)|\,d\tau\biggr)\cdotp
\end{multline}
{}From Inequality \eqref{eq:Lrho}  and Fourier-Plancherel theorem, 
it is easy to obtain estimates of $L^2$ type for the solutions to \eqref{eq:LPH}. 
Optimal informations will be obtained if splitting
the unknowns into  frequency packets of comparable sizes. 
To this end, one may apply $\ddk$ to \eqref{eq:LPH} and get 
\begin{equation}\label{eq:LPHk}
\left\{\begin{array}{l}\d_t\ddk a+\div\ddk \cQ u=\ddk f,\\
\d_t\ddk\cQ u-\Delta \ddk \cQ u+\nabla\ddk a=\ddk \cQ g,\\ 
\d_t\ddk \cP u-\mu\Delta\ddk \cP u=\ddk\cP g.
\end{array}\right.
\end{equation}
In the case with no source term then using \eqref{eq:Lrho} combined  with  Fourier-Plancherel 
theorem readily yields for some universal constant $C_0,$ and $c_0$ depending only on $\mu$,
$$
\|(\ddk a,\ddk\nabla a,\ddk u)(t)\|_{L^2}\leq C_0e^{-c_0\min(1,2^{2k})t}
\|(\ddk a,\ddk\nabla a,\ddk u)(0)\|_{L^2}.
$$
Then, for general source terms, using Duhamel's formula and repeating the computations leading
to \eqref{eq:bhfbis}, we end up with
\begin{multline}\label{eq:bhfk}
\|(\ddk a,\ddk\nabla a,\ddk u)(t)\|_{L^2}+\min(1,2^k)\int_0^t\|\ddk\nabla a \|_{L^2}\,d\tau
+2^{2k}\int_0^t\|\ddk u\|_{L^2}\,d\tau\\\leq C\biggl(\|(\ddk a_0,\ddk\nabla a_0,\ddk u_0)\|_{L^2}
+\int_0^t\|(\ddk f,\ddk\nabla f,\ddk g)\|_{L^2}\,d\tau\biggr)\cdotp
\end{multline}
Multiplying both sides by $2^{ks},$ taking the supremum on $[0,t]$ then 
summing up on $k\geq k_0$ or $k\leq k_0,$ we conclude to the following:
\begin{proposition} \label{p:lph} Let $s\in\R$ and $(a,u)$ satisfy \eqref{eq:LPH} with $P'(1)=\nu=1.$
Let  $k_0\in\Z.$
Then we have for some constant $C$ depending only on $k_0$ and $\mu,$ and all $t\geq0,$
$$\|(a,u)\|^{\ell}_{\wt L_t^\infty(\dot B^{s}_{2,1})}
+\|(a,u)\|^{\ell}_{L_t^1(\dot B^{s+2}_{2,1})}
\leq C\bigl(\|(a_0,u_0)\|^{\ell}_{\dot B^{s}_{2,1}}+\|(f,g)\|^{\ell}_{L_t^1(\dot B^{s}_{2,1})}\bigr),
$$
$$\displaylines{
\|a\|^{h}_{\wt L_t^\infty(\dot B^{s+1}_{2,1})}+\|a\|^{h}_{L_t^1(\dot B^{s+1}_{2,1})}+
\|u\|^{h}_{\wt L_t^\infty(\dot B^{s}_{2,1})}+\|u\|^{h}_{L_t^1(\dot B^{s+2}_{2,1})}
\hfill\cr\hfill\leq\bigl(
\|a_0\|^{h}_{\dot B^{s+1}_{2,1}}+\|u_0\|^{h}_{\dot B^{s}_{2,1}}+
\|f\|^{h}_{L_t^1(\dot B^{s+1}_{2,1})}+\|g\|^{h}_{L_t^1(\dot B^{s}_{2,1})}\bigr),}
$$
where we  used the notation
 \begin{equation}\label{eq:notation}\|z\|^{\ell}_{\dot B^\sigma_{p,1}}=
\sum_{k\leq k_0}2^{k\sigma}\|\ddk z\|_{L^p}\quad\hbox{and}\quad 
\|z\|^{h}_{\dot B^\sigma_{p,1}}=\sum_{k\geq k_0}2^{k\sigma}\|\ddk z\|_{L^p}.
\end{equation}
\end{proposition}
The high frequencies  inequality means that  in order to get
optimal estimates, it is suitable  to work with the same 
regularity for $\nabla a$ and $u.$ In contrast, for low frequencies, one has to work in the same space for $a$ and $u,$ a fact which does not follow from  our rough  scaling considerations \eqref{scaling}  but is  fundamental to keep the pressure term under control in \eqref{eq:nsc}.
\medbreak
Granted with the above proposition, it is now natural to look at \eqref{eq:nsc}
as System \eqref{eq:LPH} with right-hand side
$$
f=-\div(au)\quad\hbox{and}\quad
g=-u\cdot\nabla u-I(a)\cA u-k(a)\nabla a
\quad\hbox{where }\ k(a)\eqdefa G'(a)-G'(0).
$$
The problem is that $f$ will cause 
a loss of one derivative as there is no smoothing effect for $a$ in high frequency. 
A second limitation of Proposition \ref{p:lph} is that
it concerns Besov spaces related to  $L^2$  whereas we know the system to be locally well-posed
in more general Besov spaces (see Theorem \ref{th:small1}). 
To overcome the first problem, let us include the convection terms in our linear analysis, 
thus considering:
\begin{equation}\label{eq:LPHbis}
\left\{\begin{array}{l}\d_ta+v\cdot\nabla a+\div u=f,\\
\d_tu+v\cdot\nabla u-\mu\Delta u-(\lambda+\mu)\nabla\div u+\alpha\nabla a=g,
\end{array}\right.
\end{equation}
where $v$ stands for a given time-dependent vector field. 
\begin{proposition} \label{p:lphbis} Let $-d/2<s\leq d/2$ and $(a,u)$ satisfy \eqref{eq:LPHbis} 
with $\alpha=\nu=1.$
Let  $k_0\in\Z.$
Then we have for some constant $C$ depending only on $k_0$ and $\mu,$ and all $t\geq0,$
$$\displaylines{\|(a,\nabla a,u)\|_{\wt L_t^\infty(\dot B^{s}_{2,1})}+\|a\|^{\ell}_{L_t^1(\dot B^{s+2}_{2,1})}
+\|\nabla a\|^{h}_{L_t^1(\dot B^{s}_{2,1})}
+\|u\|_{L_t^1(\dot B^{s+2}_{2,1})}\hfill\cr\hfill\leq C\biggl(\|(a_0,\nabla a_0,u_0)\|_{\dot B^{s}_{2,1}}
+\|(f,\nabla f,g)\|_{L_t^1(\dot B^{s}_{2,1})}+\int_0^t\|\nabla v\|_{\dot B^{\frac d2}_{2,1}}\|(a,\nabla a,u)\|_{\dot B^{s}_{2,1}}\,d\tau
\biggr)\cdotp}$$
\end{proposition}
\begin{proof}
Applying $\ddk$ to \eqref{eq:LPHbis} yields:
$$
\left\{\begin{array}{l}\d_ta_k+\ddk(v\cdot\nabla a)+\div u_k=f_k,\\
\d_tu_k+\ddk(v\cdot\nabla u)-\mu\Delta u_k-(\lambda+\mu)\nabla\div u_k+\nabla a_k=g_k,
\end{array}\right.
$$
with $a_k\eqdefa\ddk a,$ $u_k\eqdefa\ddk u,$ $f_k\eqdefa\ddk f$  and $g_k\eqdefa\ddk g.$
\medbreak
Keeping in mind the proof of Proposition \ref{p:lph}, we introduce
$$
\cL_k^2\eqdefa2\|(a_k,u_k)\|_{L^2}^2+\|\nabla a_k\|_{L^2}^2+ (u_k\mid\nabla a_k)_{L^2}.
$$
Now, remembering  that $\lambda+2\mu=1,$ we get
\begin{multline}\label{eq:Lk0}\frac12\frac d{dt}\cL_k^2+\mu\|\nabla\cP u_k\|_{L^2}^2+\|(\nabla\cQ u_k,\nabla a_k)\|_{L^2}^2=\bigl(g_k\mid (2u_k\!+\!\nabla a_k)\bigr)_{L^2}+2(f_k\mid a_k)_{L^2}\\
+\bigl(\nabla f_k\mid \nabla a_k\bigr)_{L^2}-2\bigl(\ddk(v\cdot\nabla a) \mid a_k\bigr)_{L^2}
-2\bigl(\ddk(v\cdot\nabla u)|u_k\bigr)_{L^2}\\\bigl(\ddk\nabla(v\cdot\nabla a)\mid\nabla a_k\bigr)_{L^2}
-\bigl(\ddk(v\cdot\nabla u)\mid\nabla a_k\bigr)_{L^2}-\bigl(\ddk\nabla(v\cdot\nabla a)\mid u_k\bigr)_{L^2}.
\end{multline}
Let us  explain how to bound  the convection terms. To handle 
the second and third terms of the second line, we proceed as explained below, taking $b\in\{a,u^1,\cdots,u^d\}$. 
\smallbreak
Integrating by parts and setting   $$R_k(v,b)\eqdefa\ddk(v\cdot\nabla b)-v\cdot\nabla \ddk b,$$  we discover  that 
\begin{align*}
\bigl(\ddk(v\cdot\nabla b)\mid b_k\bigr)_{L^2}&=\int (v\cdot\nabla b_k)\cdot b_k\,dx
+\int R_k(v,b)\,b_k\,dx\\&\leq -\frac12\int |b_k|^2\div v\,dx +\|R_k(v,b)\|_{L^2}\|b_k\|_{L^2}.
\end{align*}
Bounding the last term according to \eqref{eq:transport2},  we thus get 
$$
\bigl|\bigl(\ddk(v\cdot\nabla b)\mid b_k\bigr)_{L^2}\bigr|\leq Cc_k2^{-ks}\|\nabla v\|_{\dot B^{\frac d2}_{2,1}}
\|b\|_{\dot B^s_{2,1}}\|b_k\|_{L^2}
$$
with $(c_k)_{k\in\Z}$ in the unit sphere  of $\ell^1(\Z).$
\medbreak\noindent
Next, we  use the fact that for $i\in\{1,\cdots,d\},$ 
$$
\d_i\ddk(v\cdot\nabla a)=v\cdot\nabla \d_i a_k + \wt R_k^i(v,a)
\quad\hbox{with }\  \wt R_k^i(v,a)\eqdefa [\d_i\ddk,v]\cdot\nabla a.
$$
By adapting the proof of \eqref{eq:transport2}, it is easy to prove that 
$$
\| \wt R_k(v,a)\|_{L^2} \leq Cc_k2^{-ks}\|\nabla v\|_{\dot B^{\frac d2}_{2,1}}
\|\nabla a\|_{\dot B^s_{2,1}}.
$$
Then using an integration by parts, exactly as above, we conclude that 
$$
\Bigl|\bigl(\ddk\nabla(v\cdot\nabla a)|\nabla a_k\bigr)_{L^2}\Bigr|\leq Cc_k2^{-ks}\|\nabla v\|_{\dot B^{\frac d2}_{2,1}}
\|\nabla a\|_{\dot B^s_{2,1}}\|\nabla a_k\|_{L^2}.
$$
Finally, to handle the last two convection terms, we use the fact that 
$$
\displaylines{
\bigl(\ddk(v\cdot\nabla u)\mid\nabla a_k\bigr)_{L^2}+\bigl(\ddk\nabla(v\cdot\nabla a)\mid u_k\bigr)_{L^2}
\hfill\cr\hfill=\bigl(v\cdot\nabla u_k|\nabla a_k\bigr)_{L^2}+\bigl((v\cdot\nabla)\nabla a_k)|u_k\bigr)_{L^2}
+\bigl(R_k(v,u)\mid\nabla a_k\bigr)_{L^2} +\bigl(\wt R_k(v,a)\mid u_k\bigr)_{L^2}.}
$$
Integrating by parts  in the first two terms of the second line and using again 
\eqref{eq:transport2} to bound the last two terms eventually leads to 
$$\displaylines{
\Bigl|\bigl(\ddk(v\cdot\nabla u)\mid\nabla a_k\bigr)_{L^2}+\bigl(\ddk\nabla(v\cdot\nabla a)\mid u_k\bigr)_{L^2}\Bigr|
\hfill\cr\hfill\leq C c_k2^{-ks}\|\nabla v\|_{\dot B^{\frac d2}_{2,1}}\bigl(\|\nabla a\|_{\dot B^s_{2,1}}\|u_k\|_{L^2}
+\|u\|_{\dot B^s_{2,1}}\|\nabla a_k\|_{L^2}\bigr).}
$$
Because $\cL_k\approx\|(a_k,\nabla a_k,u_k)\|_{L^2},$  we thus conclude that  
\begin{multline}\label{eq:Lk1}
\frac12\frac d{dt}\cL_k^2+\mu\|\nabla\cP u_k\|_{L^2}^2+\|(\nabla\cQ u_k,\nabla a_k)\|_{L^2}^2
\\\leq  \bigl(\|(f_k,\nabla f_k,g_k)\|_{L^2}+Cc_k2^{-ks}\|\nabla v\|_{\dot B^{\frac d2}_{2,1}}\|(a,\nabla a,u)\|_{\dot B^s_{2,1}}\bigr)\cL_k,
\end{multline}
which after time integration and multiplication by $2^{ks}$ yields 
$$\displaylines{
2^{ks}\cL_k(t)+c_02^{ks}\min(1,2^{2k})\int_0^t\|(a_k,\nabla a_k,u_k)\|_{L^2}\,d\tau
\leq 2^{ks}\cL_k(0)+\int_0^t2^{ks}\|g_k\|_{L^2}\,d\tau\hfill\cr\hfill+\int_0^tc_k\|\nabla v\|_{\dot B^{\frac d2}_{2,1}}\|(a,\nabla a,u)\|_{\dot B^s_{2,1}}\,d\tau.}
$$
Taking the supremum on $[0,t]$ then summing up over $k,$ we thus get
\begin{multline}\label{eq:Lk2}
\|(a,\nabla a,u)\|_{\wt L^\infty_t(\dot B^s_{2,1})}+\int_0^t\|(a,u)\|^\ell_{\dot B^{s+2}_{2,1}}\,d\tau
+\int_0^t\|(\nabla a,u)\|^h_{\dot B^{s}_{2,1}}\,d\tau\\\lesssim \|(a,\nabla a,u)(0)\|_{\dot B^s_{2,1}}
+\int_0^t\|(f,\nabla f,g)\|_{\dot B^{s}_{2,1}}\,d\tau
+\int_0^t\|\nabla v\|_{\dot B^{\frac d2}_{2,1}}\|(a,\nabla a,u)\|_{\dot B^s_{2,1}}\,d\tau.
\end{multline}
Finally, using the fact that 
$$
\d_tu+v\cdot\nabla u-\mu\Delta u-(\lambda+\mu)\nabla\div u=g-\nabla a,
$$
localizing according to $\ddk,$ and arguing  as above, we find out that
$$
\|u\|_{\wt L_t^\infty(\dot B^s_{2,1})}+\int_0^t\|u\|_{\dot B^{s+2}_{2,1}}\lesssim \|u(0)\|_{\dot B^s_{2,1}}
+\int_0^t\|g-\nabla a\|_{\dot B^{s}_{2,1}}\,d\tau
+\int_0^t\|\nabla v\|_{\dot B^{\frac d2}_{2,1}}\|u\|_{\dot B^s_{2,1}}\,d\tau.
$$
Then bounding $\nabla a$ according to \eqref{eq:Lk2} completes the proof of the proposition.\qed
\end{proof}
It turns out to be possible to extend the above proposition  to more general Besov spaces
related to the $L^p$ spaces with $p\not=2.$
The proof relies  on a  paralinearized version 
of System \eqref{eq:LPHbis} combined with a Lagrangian change of variables  (see  \cite{CD,CMZ2}). 
Here, in order to solve \eqref{eq:nsc} globally,  we shall follow  a more elementary approach based on the
paper by B. Haspot \cite{H2} : we use Proposition \ref{p:lph} only for bounding
low frequencies,  and perform a suitable quasi-diagonalization of the system to handle high frequencies. 
This eventually leads to the following statement\footnote{The reader may refer
to \cite{DH} for a slightly more general result.} 
 that will be proved in the rest of this section:
\begin{theorem}\label{th:main} Let $d\geq2.$
Let $p\in [2,\min(4,2d/(d-2)]$ with, additionally, $p\not=4$ if $d=2.$ 
Assume with no loss of generality 
that $P'(1)=1$ and $\nu=1.$  There exists a universal integer $k_0\in\N$  and a small  constant $c=c(p,d,\mu,G)$ such that
if    $a_0\in \dot B^{\frac d{p}}_{p,1}$
and $u_0\in \dot B^{\frac d{p}-1}_{p,1}$ with   
    besides  $(a_0^\ell,u_0^\ell)$  in $\dot B^{\frac d2-1}_{2,1}$
    (with the notation $z^\ell=\dot S_{k_0+1}z$ and $z^h=z-z^\ell$) satisfy 
    \begin{equation}\label{eq:globalsmall}
X_{p,0}\eqdefa \|(a_0,u_0)\|^\ell_{\dot B^{\frac d2-1}_{2,1}}+\|a_0\|^h_{\dot B^{\frac dp}_{p,1}}
+\|u_0\|^h_{\dot B^{\frac d{p}-1}_{p,1}}\leq c\end{equation}
then \eqref{eq:nsc}  has a unique global-in-time  solution $(a,u)$  in the space $X_p$ defined by 
$$\displaylines{
(a,u)^\ell\in \wt\cC_b(\R_+;\dot B^{\frac d2-1}_{2,1})\cap  L^1(\R_+;\dot B^{\frac d2+1}_{2,1}),\quad
a^h\in \wt\cC_b(\R_+;\dot B^{\frac dp}_{p,1})\cap L^1(\R_+;\dot B^{\frac dp}_{p,1}),
\cr u^h\in  \wt\cC_b(\R_+;\dot B^{\frac dp-1}_{p,1})
\cap L^1(\R_+;\dot B^{\frac dp+1}_{p,1})}
$$
where we agree that $\, \wt\cC_b(\R_+;\dot B^{s}_{q,1})\eqdefa \cC(\R_+;\dot B^{s}_{q,1})\cap
 \wt L^\infty(\R_+;\dot B^{s}_{q,1}),$ $\, s\in\R,$ $\,1\leq q\leq\infty.$
 \medbreak
Furthermore, we have for some constant $C=C(p,d,\mu,G),$ 
\begin{equation}\label{eq:globalbound}
\|(a,u)\|_{X_p}\leq C X_{p,0}.
\end{equation}
\end{theorem}
 \begin{remark} Condition  \eqref{eq:globalsmall} is satisfied for small $a_0$ and  \emph{large} highly oscillating  velocities:
take  $u_0^\ep:x\mapsto \phi(x) \sin (\ep^{-1} x\cdot\omega)\, n$ with  $\omega$ and
$n$  in ${\mathbb S}^{d-1}$ and $\phi\in\cS(\R^d).$ Then 
$$
\|u_0^\ep\|_{\dot B^{\frac d{p}-1}_{p,1}}\leq C\ep^{1-\frac d{p}}\quad\hbox{if }\ p>d,
$$
and $\|u_0^\ep\|_{\dot B^{\frac d{2}-1}_{2,1}}^\ell$ has fast decay with respect to $\ep.$
Hence such data with small enough~$\ep$  generate global unique solutions
in dimension $d=2,3$.
\end{remark}
\begin{remark}
One may  extend the above global result to  $2d/(d+2)\leq p<2$   provided the following smallness condition is fulfilled: 
$$\|a_0\|_{\dot B^{\frac d2-1}_{2,1}\cap\dot B^{\frac d2}_{2,1}}+\|u_0\|_{\dot B^{\frac d2-1}_{2,1}}\leq \eta.$$
Indeed,  Theorem \ref{th:main1} provides a global small solution in $X_{2}.$ Therefore it is only a matter
of checking that the constructed solution has  additional regularity $X_p.$
This may be achieved by following Steps 3 and 4 of the proof below, knowing already that 
the solution is in $X_{2}.$ The condition that $2d/(d+2)\leq p$ comes from 
the part $u^{\ell}\cdot\nabla a$ of the convection term in the mass equation, as $\nabla u^{\ell}$
is only in $L^1(\R_+;\dot B^{\frac d2}_{2,1}),$ and the regularity to be transported is $\dot B^{\frac dp}_{p,1}.$ 
Hence we need to have $d/p\leq d/2+1$ (see Theorem \ref{th:transport}). 
The same condition appears when handling $k(a)\nabla a$. 
\end{remark}
\begin{remark} Using space $\wt C_b(\R_+;\dot B^s_{2,1})$ rather than just
$C_b(\R_+;\dot B^s_{2,1})$ is not essential in the proof of Theorem \ref{th:main}.
We chose to present that slightly more accurate result, as it will be needed
when investigating time decay estimates, at the end of the survey.
\end{remark}


\subsection{Global a priori estimates}

Consider a smooth solution $(a,u)$ to \eqref{eq:nsc}
satisfying, say,  
\begin{equation}\label{eq:pasdevide}
\|a\|_{L^\infty(\R_+\times\R^d)}\leq 1/2.
\end{equation}
We want to find conditions under which the following quantity: 
$$\displaylines{\quad
X_p(t)\eqdefa\|(a,u)\|^\ell_{\wt L_t^\infty(\dot B^{\frac d2-1}_{2,1})}+\|(a,u)\|^\ell_{L_t^1(\dot B^{\frac d2+1}_{2,1})}
\hfill\cr\hfill+\|a\|^h_{\wt L_t^\infty(\dot B^{\frac dp}_{p,1})}+\|a\|^h_{L_t^1(\dot B^{\frac dp}_{p,1})}
+\|u\|^h_{\wt L_t^\infty(\dot B^{\frac dp-1}_{p,1})}+\|u\|^h_{L_t^1(\dot B^{\frac dp+1}_{p,1})}\quad}
$$
satisfies \eqref{eq:globalbound} for all $t\in\R_+.$
\medbreak
Rewriting System \eqref{eq:nsc} as follows:
 $$
\left\{\begin{array}{l}
\d_ta+\div u= f\eqdefa -\div(au),\\
\d_tu-\cA u+\nabla a =g\eqdefa -u\cdot\nabla u-I(a)\cA u-k(a)\nabla a,
\end{array}\right.
$$
we shall   take advantage of Proposition \ref{p:lph} with $s'=d/2-1$
 to  bound the low frequency part  of $(a,u).$ 
 To handle high frequencies, following \cite{H2}, we shall  use the facts that,  up to low order terms:
\begin{itemize}
\item $\cP u$ satisfies a heat equation (hence parabolic smoothing in any Besov space);
\item The \emph{effective velocity} 
\begin{equation}\label{eq:modified}
w\eqdefa\nabla(-\Delta)^{-1}(a-\div u)\end{equation}
satisfies a heat equation;
\item The high frequencies of  $a$ have exponential decay.
 \end{itemize}
 
\subsubsection*{First step:  Low frequencies}

{}From  Proposition \ref{p:lph},  we readily infer that
\begin{equation}\label{eq:LF}
\|(a,u)\|^\ell_{\wt L_t^\infty(\dot B^{\frac d2-1}_{2,1})}
+\|(a,u)\|^\ell_{L_t^1(\dot B^{\frac d2+1}_{2,1})}\lesssim
\|(a_0,u_0)\|^\ell_{\dot B^{\frac d2-1}_{2,1}}+\|(f,g)\|^\ell_{L_t^1(\dot B^{\frac d2-1}_{2,1})}.
\end{equation}

\subsubsection*{Second step: high frequencies, the incompressible part of the velocity}

To handle $\cP u,$ we just use the fact that 
$$ 
\d_t\cP u -\mu\Delta\cP u=\cP g.
$$
Hence, according to Remark \ref{r:heat} (restricted to high frequencies)
  \begin{equation}\label{eq:HFPu}
\|\cP u\|^h_{\wt L^\infty_t(\dot B^{\frac dp-1}_{p,1})}+\mu \|\cP u\|^h_{L^1_t(\dot B^{\frac dp+1}_{p,1})}
 \leq C \bigl(\|\cP u_0\|_{\dot B^{\frac dp-1}_{p,1}}^h+\|\cP g\|^h_{L^1_t(\dot B^{\frac dp-1}_{p,1})}\bigr).
\end{equation}

\subsubsection*{Third step: high frequencies, the effective  velocity and the density}

On the one hand, the   effective velocity $w$ defined in \eqref{eq:modified} fulfills 
  $$
\d_tw-\Delta w=\nabla(-\Delta)^{-1}(f-\div g) +w-(-\Delta)^{-1}\nabla a.
 $$
 Therefore, Theorem \ref{th:heat} and the fact that $\nabla(-\Delta)^{-1}$ is an homogeneous
 Fourier multiplier of degree $-1$   imply  that 
 \begin{multline}\label{eq:w}
 \|w\|_{\wt E_p(t)}^h\eqdefa\|w\|_{\wt L^\infty_t(\dot B^{\frac dp-1}_{p,1})}^h+\|w\|_{L^1_t(\dot B^{\frac dp+1}_{p,1})}^h\leq C\bigl(\|w_0\|_{\dot B^{\frac dp-1}_{p,1}}^h \\+
 \|f-\div g\|_{L^1_t(\dot B^{\frac dp-2}_{p,1})}^h
 +\|w-(-\Delta)^{-1}\nabla a\|_{L^1_t(\dot B^{\frac dp-1}_{p,1})}^h\bigr).
 \end{multline}
 On the other hand, we have 
 \begin{equation}\label{eq:eqa}
 \d_ta+\div(au)+a=-\div w.
 \end{equation}
 We claim that 
 \begin{equation}\label{eq:a}
 \|a\|_{\wt L^\infty_t(\dot B^{\frac dp}_{p,1})}^h+\|a\|_{L^1_t(\dot B^{\frac dp}_{p,1})}^h
 \leq C\biggl(\|a_0\|^h_{\dot B^{\frac dp}_{p,1}}+\|\div w\|_{L_t^1(\dot B^{\frac dp}_{p,1})}^h
 +\int_0^t\|\nabla u\|_{\dot B^{\frac dp}_{p,1}}\|a\|_{\dot B^{\frac dp}_{p,1}}\,d\tau\biggr)\cdotp
 \end{equation}
 Indeed,  as in the proof of  Theorem \ref{th:transport},  let us apply $\ddk$ to \eqref{eq:eqa}. We get
 $$
 \d_t\ddk a+u\cdot\nabla\ddk a+\ddk a=   -\ddk(a\div u)-\ddk\div w+\dot R_k,                 
 $$
 where, according to \eqref{eq:transport2}, the remainder term $\dot R_k$ satisfies:
 $$
\forall k\in\Z,\; \|\dot R_k\|_{L^p}\leq C c_k 2^{-k\frac dp}\|\nabla u\|_{\dot B^{\frac dp}_{p,1}}\|a\|_{\dot B^{\frac dp}_{p,1}}
 \quad\hbox{with}\quad \sum_{k\in\Z}c_k=1         
 $$
 and 
 $$
 \|a\,\div u\|_{\dot B^{\frac dp}_{p,1}}\leq C\|\div u\|_{\dot B^{\frac dp}_{p,1}}\|a\|_{\dot B^{\frac dp}_{p,1}}. 
 $$
 Therefore evaluating the $L^p$ norm of $\ddk a$ seen as the solution to a transport equation, 
 multiplying by $2^{k\frac dp}$  and summing up over $k\geq k_0$ yields \eqref{eq:a}.
  \medbreak
 Next, let us observe that, owing to the high frequency cut-off, we have for some universal constant $C,$
 \begin{equation}\label{eq:hfdecay}
 \|w\|^h_{\dot B^{\frac dp-1}_{p,1}}\leq C2^{-2k_0} \|w\|^h_{\dot B^{\frac dp+1}_{p,1}}
 \quad\hbox{and}\quad\|(-\Delta)^{-1}\nabla a\|_{\dot B^{\frac dp-1}_{p,1}}^h\leq C2^{-2k_0}
 \|a\|_{\dot B^{\frac dp}_{p,1}}^h.
 \end{equation}
 In consequence, combining \eqref{eq:w} and \eqref{eq:a}, and choosing $k_0$ large enough yields
 \begin{multline}\label{eq:HFQu}
  \|w\|_{\wt E_p(t)}^h+ \|a(t)\|_{\dot B^{\frac dp}_{p,1}}^h+\|a\|_{L^1_t(\dot B^{\frac dp}_{p,1})}^h
  \leq\biggl(\|w_0\|_{\dot B^{\frac dp-1}_{p,1}}^h +
  \|a_0\|^h_{\dot B^{\frac dp}_{p,1}}\\+ \|f-\div g\|_{L^1_t(\dot B^{\frac dp-2}_{p,1})}^h
  +\int_0^t\|\nabla u\|_{\dot B^{\frac dp}_{p,1}}\|a\|_{\dot B^{\frac dp}_{p,1}}\,d\tau\biggr)\cdotp
 \end{multline}
 
 \subsubsection*{Fourth step: end of the proof of the linear estimate}

 Putting Inequality \eqref{eq:HFQu} together with \eqref{eq:LF} and \eqref{eq:HFPu} and observing that
 $$
 \|u\|_{\wt E_p(t)}^h\leq  \|\cP u\|^h_{E_p(t)}+ \|w\|_{\wt E_p(t)}^h+C\bigl( \|a\|_{\wt L^\infty_t(\dot B^{\frac dp-2}_{p,1})}^h+\|a\|_{L^1_t(\dot B^{\frac dp}_{p,1})}^h\bigr),
 $$ 
 we come to the conclusion (if $k_0$ has been taken large enough) that
 $$
 X_p(t)\leq C_{k_0}\biggl(X_p(0)+\int_0^t\bigl(\|(f,g)\|_{\dot B^{\frac d2-1}_{2,1}}^\ell
 +\|f\|_{\dot B^{\frac dp-2}_{p,1}}^h+\|g\|_{\dot B^{\frac dp-1}_{p,1}}^h+
 \|\nabla u\|_{\dot B^{\frac dp}_{p,1}}\|a\|_{\dot B^{\frac dp}_{p,1}}\bigr)\,d\tau\biggr)\cdotp
$$
 
\subsubsection*{Fifth step : nonlinear estimates}

It is only a matter of proving that under hypothesis \eqref{eq:pasdevide}, we have
\begin{equation}\label{eq:rhs}
\int_0^t\bigl(\|(f,g)\|_{\dot B^{\frac d2-1}_{2,1}}^\ell
 +\|f\|_{\dot B^{\frac dp-2}_{p,1}}^h+\|g\|_{\dot B^{\frac dp-1}_{p,1}}^h+
 \|\nabla u\|_{\dot B^{\frac dp}_{p,1}}\|a\|_{\dot B^{\frac dp}_{p,1}}\bigr)\,d\tau
\leq CX_p^2(t).
\end{equation}
As $p\geq2,$ it is clear that the last term of the r.h.s. of \eqref{eq:rhs} is bounded by $CX_p^2(t).$
Next, arguing exactly as in the proof of the local existence, we easily get for $1\leq p<2d,$
$$
\|f\|_{L^1_t(\dot B^{\frac dp-1}_{p,1})}\leq C\|a\|_{L^2_t(\dot B^{\frac dp}_{p,1})}\|u\|_{L^2_t(\dot B^{\frac dp}_{p,1})},
$$
$$\displaylines{
\quad\|g\|_{L^1_t(\dot B^{\frac dp-1}_{p,1})}\leq C\bigl(
 \|u\|_{L^\infty_t(\dot B^{\frac dp-1}_{p,1})}\|\nabla u\|_{L^1_t(\dot B^{\frac dp}_{p,1})}
  \hfill\cr\hfill+\|a\|_{L^\infty_t(\dot B^{\frac dp}_{p,1})}\|\nabla^2u\|_{L^1_t(\dot B^{\frac dp-1}_{p,1})}
  + \|a\|_{L^2_t(\dot B^{\frac dp}_{p,1})}\|\nabla a\|_{L^2_t(\dot B^{\frac dp-1}_{p,1})}\bigr)\cdotp\quad}
$$
Therefore, using the definition of $X_p(t)$ and  embedding (recall that $p\geq2$),   
we get $$\|(f,g)\|^h_{L_t^1(\dot B^{\frac dp-1}_{p,1})}\leq CX_p^2(t).$$ 
So we are left with the proof of 
\begin{equation}\label{eq:nl}
\|(f,g)\|_{L^1_t(\dot B^{\frac d2-1}_{2,1})}^\ell\leq CX_p^2(t).
\end{equation}
Let us admit  the following two inequalities (the first one being proved in \cite{DH} and
the second one being a particular case of Proposition \ref{p:op} followed by suitable embedding, owing to  $1\leq p/2\leq 2$)~:
\begin{equation}\label{eq:paracont}
\|T_{a}b\|_{\dot B^{s-1+\frac d2-\frac dp}_{2,1}}\leq C\|a\|_{\dot B^{\frac dp-1}_{p,1}}\|b\|_{\dot B^s_{p,1}}
\quad\hbox{if }\ d\geq2 \ \hbox{ and }\  \textstyle \frac d{d-1}\leq p\leq\min\bigl(4,\frac{2d}{d-2}\bigr),
\end{equation} 
\begin{equation}\label{eq:restecont}
\|R(a,b)\|_{\dot B^{s-1+\frac d2-\frac dp}_{2,1}}\leq C\|a\|_{\dot B^{\frac dp-1}_{p,1}}\|b\|_{\dot B^s_{p,1}}
\quad\hbox{if }\  \textstyle s>1-\min\bigl(\frac dp,\frac d{p'}\bigr)\ \hbox{ and }\  1\leq p\leq4.
\end{equation}

In order to  prove \eqref{eq:nl} for $f,$ it suffices to bound $(au)^\ell$
in $L^1(0,t;\dot B^{\frac d2}_{2,1})$.
Now,  using Bony's decomposition and the fact that $a=a^\ell+a^h,$ we see that
\begin{equation}\label{eq:deca}
(au)^\ell=\bigl(T_au)^\ell+\bigl(R(a,u)\bigr)^\ell+\bigl(T_ua^\ell\bigr)^\ell+\bigl(T_ua^h\bigr)^\ell.
\end{equation}
The first three terms may be bounded thanks to Prop. \ref{p:op}  and Inequalities \eqref{eq:paracont}, \eqref{eq:restecont}  with $s=\frac dp-1.$ 
Observing that $\|z\|^\ell_{\dot B^\sigma_{r,1}}\leq C\|z\|_{\dot B^\sigma_{r,1}}$ for  any Besov norm, 
we get
$$\begin{array}{lll}
\|(T_au)^\ell\|_{L^1_t(\dot B^{\frac d2}_{2,1})}\leq C\|a\|_{L^\infty_t(\dot B^{\frac dp-1}_{p,1})}\|u\|_{L^1_t(\dot B^{\frac dp+1}_{p,1})},\\
\|(R(a,u))^\ell\|_{L^1_t(\dot B^{\frac d2}_{2,1})}\leq C\|a\|_{L^\infty_t(\dot B^{\frac dp-1}_{p,1})}\|u\|_{L^1_t(\dot B^{\frac dp+1}_{p,1})},\\
\|(T_ua^\ell)^\ell\|_{L^1_t(\dot B^{\frac d2}_{2,1})}\leq C\|u\|_{L^\infty_t(\dot B^{-1}_{\infty,\infty})}\|a^\ell\|_{L^1_t(\dot B^{\frac d2+1}_{2,1})}.
\end{array}
$$
Because $\dot B^{\frac dp-1}_{p,1}$ is embedded in $\dot B^{-1}_{\infty,\infty},$ the above right-hand sides may  be bounded by $CX_p^2(t).$
To handle the last term of \eqref{eq:deca}, we just have to observe that owing to the spectral cut-off, there exists a universal integer $N_0$ so that
$$
\bigl(T_ua^h\bigr)^\ell=\dot S_{k_0+1}\Bigl(\sum_{|k-k_0|\leq N_0} \dot S_{k-1}u\,\ddk a^h\Bigr)\cdotp
$$
Hence $\|T_ua^h\|_{\dot B^{\frac d2}_{2,1}}\approx 2^{k_0\frac d2}\sum_{|k-k_0|\leq N_0} \|\dot S_{k-1}u\,\ddk a^h\|_{L^2}.$
Now, if $2\leq p\leq \min(d,2d/(d-2))$ then we may use for $|k-k_0|\leq N_0$ 
$$
2^{k_0\frac d2} \|\dot S_{k-1}u\,\ddk a^h\|_{L^2}\leq C 2^{k_0}\, \|\dot S_{k-1}u\|_{L^d}\,\bigl(2^{k\frac dp}\|\ddk a^h\|_{L^p}\bigr),
$$
 and if $d\leq p\leq 4$ then 
$$ 2^{k_0\frac d2}\|\dot S_{k-1}u\,\ddk a^h\|_{L^2}\leq C2^{k_0}\bigl(2^{k(\frac dp-1)}\|\dot S_{k-1}u\|_{L^p}\bigr)\bigl(2^{k\frac dp}\|\ddk a^h\|_{L^p}\bigr).
$$
Hence one may conclude that $f$ satisfies \eqref{eq:nl}.
Bounding  $g$  is similar (see \cite{DH}).

\subsubsection*{Last step:  Global estimate}

Putting all the previous estimates together, we get
\begin{equation}\label{eq:blue}
X_p(t)\leq C\bigl(X_p(0)+X_p^2(t)).
\end{equation}
Now it is clear that 
as long as 
\begin{equation}\label{eq:purple}
2CX_p(t)\leq 1,
\end{equation}
 Inequality \eqref{eq:blue}  ensures that 
\begin{equation}\label{eq:red}
X_p(t)\leq 2CX_p(0).
\end{equation}
Using a bootstrap argument, one may
conclude that if $X_p(0)$ is small enough  then \eqref{eq:pasdevide} and 
\eqref{eq:purple}  are satisfied as long as the solution exists. 
Hence  \eqref{eq:red} holds globally in time.

\subsubsection{The proof of Theorem \ref{th:main}}

We just give the important steps. We fix some initial data so that $X_0$ is small enough. First, Theorem \ref{th:small1} implies that 
there exists a unique maximal  solution   $(a,u)$ to \eqref{eq:nsc} on some time interval $[0,T^*[,$ 
with $a\in\cC([0,T^*);\dot B^{\frac dp-1}_{p,1}),$ $\|a\|_{L^\infty(0,T^*\times\R^d)}\leq1/2$ and $u \in \cC([0,T^*);\dot B^{\frac dp-1}_{p,1})                                                                                                                                                                                                                                                                                                                                                                                                                                                                                                                   \cap L^1_{loc}(0,T^*;\dot B^{\frac dp+1}_{p,1}).$
{}From \eqref{eq:nsc} and Proposition \ref{p:lph}, one may check that the additional low frequency information is preserved on $[0,T^*)$: we have
 $$a^\ell\in\cC([0,T^*);\dot B^{\frac d2-1}_{2,1})\cap L^1(0,T^*;\dot B^{\frac d2+1}_{2,1})\quad\hbox{and}\quad u^\ell \in \cC([0,T^*);\dot B^{\frac d2-1}_{2,1})                                                                                                                                                                                                                                                                                                                                                                                                                                                                                                                   \cap L^1_{loc}(0,T^*;\dot B^{\frac d2+1}_{2,1}).$$
Let us assume (by contradiction) that $T^*<\infty.$ 
Then applying \eqref{eq:red} for all $t<T^*$ yields 
$$
\|a\|_{\wt L^\infty_{T^*}(\dot B^{\frac dp}_{p,1})}+\|u\|_{\wt L^\infty_{T^*}(\dot B^{\frac dp-1}_{p,1})}\leq C X_0.
$$
If  $X_0$ is  so small as    \eqref{eq:red}  to imply  that both \eqref{eq:a0}
and \eqref{eq:smallu} are fulfilled on $[0,T^*)$ then, for all $t_0\in[0,T^*),$ 
 one can  solve \eqref{eq:nsc} starting with data $(a(t_0),u(t_0))$ at time $t=t_0$ 
 and get a solution according to  Theorem \ref{th:small1} on the interval $[t_0,T+t_0]$ with 
 \emph{$T$ independent of $t_0$}. Choosing $t_0>T^*-T$ thus shows that  the solution can be continued beyond $T^*,$ a contradiction.\qed


\section{Asymptotic results}

In this section, we focus on two types of asymptotic issues for  small global solutions to \eqref{eq:nsbaro} that
received a lot of attention since the eighties: the low Mach number asymptotic, and 
the long time behavior.  We shall see that essentially optimal  results
may be obtained by very simple arguments from the global result   we established in the previous section.

\subsection{The low Mach number limit}

This subsection is devoted to the rigorous justification of  the convergence of \eqref{eq:nsbaro} to the  incompressible Navier-Stokes equations 
\begin{equation}\label{eq:ins}
\left\{\begin{array}{l}\d_tu+u\cdot\nabla u-\mu\Delta u+\nabla\Pi=0,\\ \div u=0,\end{array}\right.
\end{equation}
in the so-called \emph{ill-prepared} data case, where we only assume that
$\ep^{-1}(\rho_0^\ep-1)$ and $u_0^\ep$ are suitably bounded. 
In particular,   if we set $a^\ep\eqdefa\ep^{-1}(\rho^\ep-1), $
 this means that $(\d_ta^\ep,\d_tu^\ep)|_{t=0}$ is of order $1/\ep,$
and that one cannot exclude  highly oscillating acoustic waves. 
 More concretely, we  have to pass to the limit $\ep\to0$ in:
\begin{equation}  \label{eq:NSeps}
  \left\{    \begin{aligned}
      & \d_t a^\eps+\frac{\div u^\eps}{\eps}=-\div (a^\eps u^\eps),\\
      &\d_t u^\eps + u^\eps\cdot\nabla u^\eps-\frac{\cA u^\eps}{1+\eps a^\eps} +\frac{\nabla  a^\eps}\eps=\frac{k(\eps a^\eps)}\eps\nabla a^\eps
      \\&\hspace{3.5cm}+\frac1{1+\eps a^\eps}\div\bigl(2\wt\mu(\eps a^\eps) D(u^\eps)+\wt\lambda(\eps a^\eps)
       \div u^\eps\Id\bigr).
        \end{aligned}\right.
\end{equation}
Before stating our main results, let us introduce some notation. In this section,
 we agree that for $z\in\cS'(\R^d),$  
\begin{equation}
\label{eq:decompo}  z^{\ell,\beta}\eqdefa\sum_{2^j\beta\leq 2^{j_0}} \ddj z\quad\mbox{and}\quad z^{h,\beta}\eqdefa\sum_{2^j\beta>2^{j_0}} \ddj z, \end{equation}
 for some large enough nonnegative integer $j_0$  depending only on $p,$ $d,$ and on the functions $k,$  $\lambda/\nu,$ $\mu/\nu$ with $\nu\eqdefa\lambda+2\mu.$
 The corresponding  ``truncated'' semi-norms are defined  by 
\begin{equation} \|z\|^{\ell, \beta}_{\dot B^{\sigma}_{p,r}}\eqdefa  \|z^{\ell,\beta}\|_{\dot B^{\sigma}_{p,r}} 
\ \hbox{ and }\   \|z\|^{h,\beta}_{\dot B^{\sigma}_{p,r}}\eqdefa  \|z^{h,\beta}\|_{\dot B^{\sigma}_{p,r}}. \end{equation}
 Keeping in mind the linear analysis we performed for \eqref{eq:LPH}
 in the case $\nu=1$ and $\eps=1,$ and combining with the change of variable
 \begin{equation}\label{eq:changebis}
(a,u)(t,x)\eqdefa\eps (a^\eps,u^\eps)(\eps^2\nu t,\eps\nu x),
\end{equation}
we expect the threshold between low and high frequencies to be at $1/\wt\eps$  with $\wt\eps\eqdefa\eps\nu$,
 and it is thus natural to  consider families of data $(a_0^\eps,u_0^\eps)$ such  that
$$
\|(a_0^\eps,u_0^\eps)\|_{\dot B^{\frac d2-1}_{2,1}}^{\ell,\wt\eps}+
\|a_0^\eps\|^{h,\wt\eps}_{\dot B^{\frac dp}_{p,1}}+\|u_0^\eps\|^{h,\wt\eps}_{\dot B^{\frac dp-1}_{p,1}}
$$
is bounded independently of $\eps.$ 
We expect the corresponding  solutions  of  \eqref{eq:NSeps}
to be uniformly in the space $X^{p}_{\eps,\nu}$  defined by 
\begin{itemize}
\item  $(a^{\ell,\wt\eps} ,u^{\ell,\wt\eps} )\in \wt\cC_b(\R_+;\dot B^{\frac d2-1}_{2,1})\cap L^1(\R_+;\dot B^{\frac d2+1}_{2,1}),$
\item  $a^{h,\wt\eps}\in\wt\cC_b(\R_+;\dot B^{\frac dp}_{p,1})\cap L^1(\R_+; \dot B^{\frac dp}_{p,1}),$
\item $u^{h,\wt\eps}\in\wt\cC_b(\R_+;\dot B^{\frac dp-1}_{p,1})\cap L^1(\R_+;\dot B^{\frac dp+1}_{p,1}),$
\end{itemize}
and  endowed  with the norm:
$$\displaylines{
\|(a,u)\|_{X^{p}_{\eps,\nu}}\eqdefa \|(a,u)\|^{\ell,\wt\eps}_{\wt L^\infty(\dot B^{\frac d2-1}_{2,1})}+\|u \|^{h,\wt\eps}_{\wt L^\infty(\dot B^{\frac dp-1}_{p,1})}
+\wt\eps\|a\|^{h,\wt\eps}_{\wt L^\infty(\dot B^{\frac dp}_{p,1})}
\hfill\cr\hfill+\nu\|(a,u)\|^{\ell,\wt\eps}_{L^1(\dot B^{\frac d2+1}_{2,1})}+\nu\|u\|^{h,\wt\eps} _{L^1(\dot B^{\frac dp+1}_{p,1})}+\eps^{-1}\|a\|^{h,\wt\eps}_{L^1(\dot B^{\frac dp}_{p,1})}.}
$$
One can now state our main result of convergence in the small data case, the reader being referred to \cite{D1a,D1b}
for the large data case and stronger results of convergence. 
\begin{theorem}\label{th:main1} Assume that the fluid domain is either $\R^d$ or $\T^d,$ 
that the initial data $(a_0^\eps, u_0^\eps)$ are as above  and that $p$ is as in Theorem \ref{th:main}. 
 There exists a constant $\eta$ independent of $\eps$ and of $\nu$ such that  if
\begin{equation}\label{eq:smalldata1}
C_0^{\eps,\nu}\eqdefa\|(a_0^\eps,u_0^\eps)\|^{\ell,\wt\eps}_{\dot B^{\frac d2-1}_{2,1}}+\|u_0^\eps\|^{h,\wt\eps}_{\dot B^{\frac dp-1}_{p,1}}+\wt\eps\|a_0^\eps\|^{h,\wt\eps}_{\dot B^{\frac dp}_{p,1}}\leq \eta \nu,
\end{equation}
then System \eqref{eq:NSeps} with initial data $(a_0^\eps, u_0^\eps)$ has a unique global
solution $(a^\eps,u^\eps)$   in the space $X^{p}_{\eps,\nu}$ with, for some constant $C$ independent of $\eps$ and $\nu,$
\begin{equation}\label{eq:ue}
\|(a^\eps,u^\eps)\|_{X^{p}_{\eps,\nu}}\leq CC_0^{\eps,\nu}.
\end{equation}
In addition, $\cQ u^\eps$ converges weakly to $0$  when $\eps$ goes to $0,$ and, if $\cP u^\eps_0\rightharpoonup v_0$ then  $\cP u^\eps$ converges in the sense of distributions  to the unique solution of
\eqref{eq:ins} supplemented with initial data $v_0.$
\end{theorem}
\begin{proof}
Performing the change of unknowns given in \eqref{eq:changebis}
and the  change of data
\begin{equation}\label{eq:changedata}
(a_0,u_0)(x)\eqdefa \eps(a^\eps_0,u^\eps_0)(\eps\nu x)
 \end{equation}
reduces the proof of the global existence to the case   $\nu=1$ and $\eps=1,$
which was done in Theorem \ref{th:main}. 
Back  to the original variables will yield the desired uniform estimate \eqref{eq:ue}
under Condition \eqref{eq:smalldata1}. Indeed, we notice that we have
up to some harmless constant:
$$
\|(a_0^\eps,u_0^\eps)\|^{\ell,\wt\eps}_{\dot B^{\frac d2-1}_{2,1}}+\|u_0^\eps\|^{h,\wt\eps}_{\dot B^{\frac dp-1}_{p,1}}
+\wt\eps\|a_0^\eps\|^{h,\wt\eps}_{\dot B^{\frac dp}_{p,1}}=\nu\bigl(\|(a_0,u_0)\|^{\ell,1}_{\dot B^{\frac d2-1}_{2,1}}
+\|u_0\|^{h,1}_{\dot B^{\frac dp-1}_{p,1}}+\|a_0\|_{\dot B^{\frac dp}_{p,1}}^{h,1}\bigr)
$$
and
$$
\|(a^\eps,u^\eps)\|_{X^{p}_{\eps,\nu}}=\nu\|(a,u)\|_{X^{p}_{1,1}}.
$$

Granted with the uniform estimates established in the previous section,
it is now easy to pass to the limit in the system in the sense of distributions, 
by adapting the compactness arguments of  P.-L. Lions and N. Masmoudi  in \cite{LM}. 

More precisely,  consider a family $(a_0^\eps,u_0^\eps)$ of data satisfying \eqref{eq:smalldata1}
and $\cP u^\eps_0\rightharpoonup v_0$ when $\eps$ goes to $0.$
Let $(a^\eps,u^\eps)$  be the corresponding  solution of \eqref{eq:NSeps}
given by  the first part of Theorem \ref{th:main1}.
Because
\begin{equation}\label{eq:conv0}
\|a_0^{\eps}\|_{\dot B^{\frac dp-1}_{p,1}}^{h,\wt\eps}\lesssim \wt\eps\|a_0^{\eps}\|_{\dot B^{\frac dp}_{p,1}}^{h,\wt\eps},
\end{equation}
the data  $(a^{\eps}_0,u_{0}^{\eps})$ are uniformly bounded
in $\dot B^{\frac dp-1}_{p,1}\times \dot B^{\frac dp}_{p,1}.$ 
Likewise, \eqref{eq:ue} ensures that 
  $(a^{\eps},u^{\eps})$ is bounded in the space  $\cC_b(\R_+;\dot B^{\frac dp-1}_{p,1}).$
Therefore there exists a sequence $(\eps_n)_{n\in\N}$ decaying to $0$ so that
$(a^{\eps_n}_0,u_0^{\eps_n})\rightharpoonup (a_0,u_0)
\quad\!\!\hbox{in}\!\!\quad \dot B^{\frac dp-1}_{p,1}$  (with $\cP u_0= v_0$) and 
\begin{equation}\label{eq:conv00} 
(a^{\eps_n},u^{\eps_n})\rightharpoonup (a,u)
\quad\!\!\hbox{in}\!\!\quad L^\infty(\R_+;\dot B^{\frac dp-1}_{p,1})\quad\hbox{weak}\ *.\end{equation}
The strong convergence of the density to $1$  is obvious: we have
$\varrho^{\eps_n}=1+\eps_na^{\eps_n},$
and  $(a^{\eps_n})_{n\in\N}$ is bounded in $L^2(\R_+;\dot B^{\frac dp}_{p,1}).$
In order to justify  that $\div u=0,$ we rewrite the  mass equation as follows:
$$
\div u^{\eps_n}=-\eps_n\div(a^{\eps_n} u^{\eps_n})-\eps_n\d_ta^{\eps_n}.
$$
Given that $a^{\eps_n}$ and $u^{\eps_n}$ are bounded in $L^2(\R_+;\dot B^{\frac dp}_{p,1})$
(use the definition of $X^{p}_{\eps,\nu}$ and interpolation), 
the first term in the right-hand side is  $\cO(\eps_n)$ in $L^1(\R_+;\dot B^{\frac dp-1}_{p,1}).$
As for the last term, it tends to $0$ in the sense of distributions, by virtue of \eqref{eq:conv00}. 
We thus have $\div u^{\eps_n}\rightharpoonup 0,$ whence $\div u=0.$
\medbreak
To  establish that  $u$ is a solution to \eqref{eq:ins}, let us
 project the velocity equation onto divergence-free vector fields:
\begin{equation}\label{eq:conv1}
\d_t\cP u^{\eps_n}-\mu\Delta\cP u^{\eps_n}=-\cP( u^{\eps_n}\cdot\nabla u^{\eps_n})
-\cP\biggl(\frac1{1+\eps_na^{\eps_n}}\cA u^{\eps_n}\biggr)\cdotp
\end{equation}
Because $\cQ u=0,$ the left-hand side weakly converges to $\d_t u-\mu\Delta u.$
To prove that the last term tends to $0,$ we use the fact that having $\wt\eps (a^{\eps})^{h,\wt\eps}$
and $(a^{\eps})^{\ell,\wt\eps}$ bounded in $L^\infty(\dot B^{\frac dp}_{p,1})$ and $L^\infty(\dot B^{\frac dp-1}_{p,1}),$ respectively, implies that, for all $\alpha\in[0,1],$
\begin{equation}\label{eq:aaa}
\wt\eps^\alpha a^{\eps}\quad\hbox{is bounded in }\   L^\infty(\dot B^{\frac dp-1+\alpha}_{p,1}).
\end{equation}
Now,  $\cA u^\eps$ is bounded in $L^1(\dot B^{\frac dp-1}_{p,1})$ and   $p<2d.$ 
Hence,  according to product laws in Besov spaces, composition inequality and \eqref{eq:aaa}, we get
   $(1+\eps a^{\eps})^{-1}\cA u^{\eps}=\cO(\wt\eps^{1-\alpha})$
in $L^1(\dot B^{\frac dp-2+\alpha}_{p,1}),$ whenever $2\max(0,1-\frac dp)<\alpha\leq1.$
Hence the last term of \eqref{eq:conv1} goes strongly to $0$ for some appropriate norm.
\smallbreak
In order to  prove that $\cP(u^{\eps_n}\cdot\nabla u^{\eps_n})\rightharpoonup \cP(u\cdot\nabla u),$ we  note that
$$u^{\eps_n}\cdot\nabla u^{\eps_n}=\frac12\nabla|\cQ u^{\eps_n}|^2+\cP u^{\eps_n}\cdot\nabla u^{\eps_n}
+\cQ u^{\eps_n}\cdot\nabla\cP u^{\eps_n}.$$
Projecting the first term onto divergence free vector fields gives $0,$ and we also know that $\cP u=u.$
Hence we just have to prove that
\begin{equation}\label{eq:conv2}
\cP(\cP u^{\eps_n}\cdot\nabla u^{\eps_n})\rightharpoonup \cP(\cP u\cdot\nabla u)
\quad\hbox{and}\quad
\cP(\cQ u^{\eps_n}\cdot\nabla\cP u^{\eps_n})\rightharpoonup0.
\end{equation}
This requires our proving results of strong convergence for $\cP u^{\eps_n}.$ 
To this end, we shall  exhibit uniform  bounds  for $\d_t\cP u^{\eps_n}$ in a suitable space. 
First, arguing by interpolation, we see that $(\nabla^2 u^{\eps_n})$ is bounded
in $L^m(\dot B^{\frac dp+\frac 2m-3}_{p,1})$ for any $m\geq1.$ 
Choosing $m>1$ so that $\frac2m-3>-d\min(\frac{2}p,1)$ (this is possible as $p<2d$) and remembering
that $(\eps^na^{\eps_n})$ is bounded in $L^\infty(\dot B^{\frac dp}_{p,1}),$ we thus 
get $((1+\eps_na^{\eps_n})^{-1}\cA u^{\eps_n})$ bounded in $L^m(\dot B^{\frac dp+\frac2m-3}_{p,1}).$
Similarly, combining the facts that $(u^{\eps_n})$ and $(\nabla u^{\eps_n})$ are bounded in 
 $L^\infty(\dot B^{\frac dp-1}_{p,1})$ and $L^m(\dot B^{\frac dp+\frac2m-2}_{p,1}),$ respectively, we see that 
$(u^{\eps_n}\cdot\nabla u^{\eps_n})$ is  bounded in $L^m(\dot B^{\frac dp+\frac2m-3}_{p,1}),$ too. 
Computing $\d_t\cP u^{\eps_n}$ from \eqref{eq:conv1}, it is now clear that 
 $(\d_t\cP u^{\eps_n})$ is bounded in  $L^{m}(\dot B^{\frac dp+\frac2m-3}_{p,1}).$
 Hence  $(\cP u^{\eps_n}-\cP u^{\eps_n}_0)$ is bounded in $\cC^{1-\frac1m}(\R_+; \dot B^{\frac dp+\frac2m-3}_{p,1}).$
As $\cP u^{\eps_n}$ is also bounded in $\cC_b(\R_+;\dot B^{\frac dp-1}_{p,1}),$
and as the embedding of $\dot B^{\frac dp-1}_{p,1}$ in $\dot B^{\frac dp+\frac2m-3}_{p,1}$ is locally compact
(see e.g. \cite{BCD}, page 108), we conclude by means of Ascoli theorem that, up to a new extraction,  for all $\phi\in\cS(\R^d)$ and $T>0,$
\begin{equation}\label{eq:conv3}
\phi\cP u^{\eps_n}\longrightarrow \phi\cP u \quad\hbox{in}\quad  \cC([0,T];\dot B^{\frac dp+\frac2m-3}_{p,1}).
\end{equation}
Interpolating  with the bounds in $\cC_b(\R_+;\dot B^{\frac dp-1}_{p,1}),$ we can upgrade
the strong convergence in \eqref{eq:conv3}
 to the space $\cC([0,T];\dot B^{\frac dp-1-\alpha}_{p,1})$ for all small enough $\alpha>0,$
and all $T>0.$ Combining with the properties of weak convergence
for $\nabla u^{\eps_n}$ to $\nabla u,$ and $\cQ u^{\eps_n}$ to $0$ that may be deduced
from the bounds of $u^{\eps_n},$ it is now easy to conclude to \eqref{eq:conv2}.
One can use for instance the fact that for all $m>1,$ we have
$$
\nabla u^{\ep_n}\rightharpoonup\nabla u \ \hbox{ in }\ L^m(\dot B^{\frac dp+\frac2m-2}_{p,1})
\ \hbox{weak }\ *\quad\hbox{and}\quad
\cQ u^{\ep_n}\rightharpoonup 0 \ \hbox{ in }\ L^m(\dot B^{\frac dp+\frac2m-1}_{p,1})
\ \hbox{weak }\ *.
$$

\end{proof}


\subsection{Time decay estimates}

In the present subsection, we show that under a mild additional 
decay assumption that is satisfied if the data are in $L^1(\R^d)$ for instance, the $L^2$ norm (the $\dot B^0_{2,1}$ norm in fact) 
of the global solutions constructed in Theorem \ref{th:main}  decays like $t^{-\frac d4}$ for $t\to+\infty,$
exactly as  for the linearized equations.
This fact has been first observed by A. Matsumura and T. Nishida in \cite{MN}
in the case of solutions with high Sobolev regularity. 
The adaptation to the $L^2$ type critical regularity framework has been carried out recently by 
M. Okita in \cite{Okita}, in dimension  $d\geq3.$
Below, we give a more accurate description of the time decay, emphasizing a better
decay for   high frequencies. This is the key to   handling  any dimension $d\geq2.$
For simplicity, we concentrate on the $L^2$ type framework, even though we expect 
similar results to be true in the
more general  $L^p$ framework of Theorem \ref{th:main}.
\begin{theorem}\label{th:decay}
Let the data $(a_0,u_0)$ satisfy the assumptions of Theorem \ref{th:main} with $p=2$ and assume with 
no loss of generality that $P'(1)=1$ and that $\nu=1.$ Denote $\langle \tau\rangle\eqdefa\sqrt{1+\tau^2}$
and $\alpha\eqdefa\min(\frac d4+2,\frac d2+\frac12-\ep)$ with $\ep>0$ arbitrarily small. 
 There exists a positive constant $c$ so that if in addition 
\begin{equation}\label{eq:D0}
D_0\eqdefa\sup_{k\leq k_0}\bigl(\|\cF(\ddk a_0)\|_{L^\infty}+\|\cF(\ddk u_0)\|_{L^\infty}\bigr)\leq c
\end{equation}
then  the global solution $(a,u)$ 
given by Theorem \ref{th:main} satisfies for all $t\geq0,$
\begin{equation}\label{eq:D}
D(t)\leq C\bigl(D_0+\|(a_0,\nabla a_0,u_0)\|_{\dot B^{\frac d2-1}_{2,1}}\bigr)
\end{equation}
$$\displaylines{\hbox{with }\  
D(t)\eqdefa\sup_{s\in(-\frac d2,2]}\|\langle\tau\rangle^{\frac d4+\frac s2}(a,u)\|_{L^\infty_t(\dot B^s_{2,1})}^\ell
+\|\langle\tau\rangle^{\alpha}(\nabla a,u)\|_{\wt L^\infty_t(\dot B^{\frac d2-1}_{2,1})}^h
+\|\tau\nabla  u\|_{\wt L^\infty_t(\dot B^{\frac d2}_{2,1})}^h.}$$
\end{theorem}
\begin{proof}
Throughout the proof, we shall use repeatedly  that 
for $0<\sigma_1\leq\sigma_2,$ we have:
\begin{equation}\label{eq:fact1}
\int_0^t\langle t-\tau\rangle^{-\sigma_1}\langle \tau\rangle^{-\sigma_2}\,d\tau
\lesssim\langle t\rangle^{-\sigma_1}\quad\hbox{if in addition }\ \sigma_2>1.
\end{equation}

\subsubsection*{Step 1: Bounds for the low frequencies} 
Denoting by $E(D)$ the semi-group associated to \eqref{eq:LPH}, we have for all $k\in\Z,$
\begin{equation}\label{eq:decay0}
\left(\begin{array}{c}\ddk a(t)\\\ddk u(t)\end{array}\right)
=e^{tE(D)}\left(\begin{array}{c}\ddk a_0\\\ddk u_0\end{array}\right)
-\int_0^te^{(t-\tau)E(D)}\left(\begin{array}{c}\ddk f_1(\tau)\\\ddk(f_2\!+\!f_3\!+\!f_4)(\tau)\end{array}\right)d\tau
\end{equation}
with $f_1\eqdefa \div(au),$ $f_2\eqdefa u\cdot\nabla u,$ $f_3\eqdefa k(a)\nabla a$ and $f_4\eqdefa I(a)\cA u.$
\medbreak
{}From an explicit computation of  the action of $e^{tE(D)}$ in Fourier variables (see e.g. \cite{CD}), we
discover that there exist positive constants $c$ and $C$ depending 
only on $k_0$ and such that 
$$|\cF(e^{tA(D)} U)(\xi)|\leq Ce^{-c_0t|\xi|^2}|\cF U(\xi)|\quad\hbox{for all }\  |\xi|\leq2^{k_0}.$$ 
Therefore,  for all $k\leq k_0,$
\begin{align*}\|e^{tE(D)}\ddk U\|_{L^2}^2&\lesssim \int e^{-2c_0|\xi|^2t}|\cF \ddk U(\xi)|^2\,d\xi\\
&\lesssim\|\cF\ddk U\|_{L^\infty}^2 2^{kd}\,e^{-c_02^{2k}t}.\end{align*}
We thus get up to a change of $c_0,$
\begin{equation}\label{eq:decay1}
t^{\frac d4+\frac s2}\sum_{k\leq k_0}2^{ks}\|e^{tE(D)}\ddk U\|_{L^2}\lesssim 
\Bigl(\sup_{k\leq k_0}\|\cF\ddk U\|_{L^\infty}\Bigr)\sum_{k\leq k_0}(\sqrt t\,2^k)^{\frac d2+s}\,e^{-c_02^{2k}t}.
\end{equation}
As for any $\sigma>0$ there  exists a constant $C_\sigma$ so that
\begin{equation}\label{eq:decay1a}
\sup_{t\geq0}\sum_{k\in\Z}t^{\frac\sigma2}2^{k\sigma}e^{-c_02^{2k}t}\leq C_\sigma,
\end{equation}
we get from \eqref{eq:decay1} that for $s>-d/2,$
$$
\sup_{t\geq0}\, t^{\frac d4+\frac s2}\|e^{tE(D)}U\|_{\dot B^s_{2,1}}^\ell\leq C_s  \sup_{k\leq k_0} \|\cF\ddk U\|_{L^\infty}.
$$
It is also obvious that  for $s>-d/2,$
$$
\|e^{tE(D)}U\|_{\dot B^s_{2,1}}^\ell\lesssim \|U\|_{\dot B^s_{2,1}}^\ell\lesssim \sup_{k\leq k_0} \|\cF\ddk U\|_{L^\infty}.
$$
Hence we conclude  that 
\begin{equation}\label{eq:decay2}
\sup_{t\geq0}\langle t\rangle^{\frac d4+\frac s2}\|e^{tE(D)}U\|_{\dot B^s_{2,1}}^\ell\lesssim  \sup_{k\leq k_0} \|\cF\ddk U\|_{L^\infty}.
\end{equation}
Next, we claim that for all $s\in(-d/2,2]$ and $i\in\{1,\cdots,4\},$ we have
\begin{equation}\label{eq:decay3}
\int_0^t\langle t-\tau\rangle^{-\frac d4-\frac s2} \sup_{k\leq k_0} \|\cF\ddk f_i(\tau)\|_{L^\infty}\,d\tau\lesssim\langle t\rangle^{-\frac d4-\frac s2}
\bigl(D^2(t)+X^2(t)\bigr)
\end{equation}
with $\,X(t)\eqdefa\|(a,\nabla a,u)\|_{\wt L^\infty_t(\dot B^{\frac d2-1}_{2,1})}+\Int_0^t\bigl(\|a\|_{\dot B^{\frac d2+1}_{2,1}}^\ell
+\|a\|_{\dot B^{\frac d2}_{2,1}}^h+\|u\|_{\dot B^{\frac d2+1}_{2,1}}\bigr)\,d\tau.$\smallbreak
Of course, as the Fourier transform maps $L^1$ to $L^\infty,$ it suffices to prove \eqref{eq:decay3} with $\|f_i\|_{L^1}$ instead of $\sup_{k\leq k_0} \|\cF\ddk f_i\|_{L^\infty}.$
\medbreak
To bound the term with $f_1,$ we use the following decomposition:
$$
f_1=u\cdot\nabla a+ a\, \div u^\ell + a\,\div u^h.
$$
Now, from Cauchy-Schwarz inequality and the definition of $D(t),$ one may write
\begin{align*}
\int_0^t\langle t-\tau\rangle^{-\frac d4-\frac s2}\|(u\cdot\nabla a)(\tau)\|_{L^1}\,d\tau
&\leq \Bigl(\sup_{0\leq\tau\leq t}\langle \tau\rangle^{\frac d4}\|u(\tau)\|_{L^2}\Bigr) 
\Bigl(\sup_{0\leq\tau\leq t}\langle \tau\rangle^{\frac d4+\frac12}\|\nabla a(\tau)\|_{L^2}\Bigr)\\
&\hspace{4cm}\times\int_0^t\langle t-\tau\rangle^{-\frac d4-\frac s2}\langle\tau\rangle^{-\frac{d}2-\frac12}\,d\tau\\
&\lesssim\langle t\rangle^{-\frac d4-\frac s2} D^2(t),
\end{align*}
where we used   \eqref{eq:fact1} and the fact that $0<\frac d4+\frac s2\leq \frac d2+\frac12\cdotp$
\medbreak
Bounding the term with $a\,\div u^\ell$ is totally similar.
Regarding the term with $a\,\div u^h,$ we use that if $t\geq2,$
$$\displaylines{
\int_0^t\langle t-\tau\rangle^{-\frac d4-\frac s2}\|(a\div u^h)(\tau)\|_{L^1}\,d\tau\lesssim 
\int_0^1\langle t-\tau\rangle^{-\frac d4-\frac s2} \|a(\tau)\|_{L^2}\|\div u^h(\tau)\|_{L^2}\,d\tau
\hfill\cr\hfill+\int_1^t\langle t-\tau\rangle^{-\frac d4-\frac s2}\langle \tau\rangle^{-1-\frac{d}4}
\bigl(\langle\tau\rangle^{\frac d4}\|a(\tau)\|_{L^2}\bigr)\bigl(\tau\|\div u^h(\tau)\|_{L^2}\bigr)\,d\tau.}
$$
Therefore, as $-d/2<s\leq2,$ we get 
$$\displaylines{
\langle t\rangle^{\frac s2+\frac d4}\int_0^t\langle t-\tau\rangle^{-\frac d4-\frac s2}\|(a\div u^h)(\tau)\|_{L^1}\,d\tau\lesssim 
\Bigl(\sup_{\tau\in[0,t]}  \|a(\tau)\|_{L^2}\Bigr) \int_0^t \|\div u^h(\tau)\|_{L^2}\,d\tau
\hfill\cr\hfill+\Bigl(\sup_{\tau\in[0,t]} \langle\tau\rangle^{\frac d4}\|a(\tau)\|_{L^2}\Bigr)
\Bigl(\sup_{\tau\in[0,t]} \tau\|\div u^h(\tau)\|_{L^2}\Bigr),}
$$
and \eqref{eq:decay3} is thus satisfied by the term with  $a\,\div u^h$ if $t\geq2,$
the case $t\leq2$ being obvious as $\langle t\rangle\approx1$ and 
$\langle t-\tau\rangle\approx1$ for $0\leq\tau\leq t\leq 2$ and one may write
$$
\int_0^t\|a\,\div u^h\|_{L^1}\,d\tau\leq \|a\|_{L^2_t(L^2)}\|\div u^h\|_{L_t^2(L^2)}\lesssim X^2(t).
$$
Handling the terms with $f_2$ and $f_3$ is totally similar:  $k(a)\nabla a$ and  $u\cdot\nabla u^\ell$ may be treated as $u\cdot\nabla a,$
  and $u\cdot\nabla u^h,$ as $a\,\div u^h.$
For $f_4,$ we write that
$$
f_4=I(a)\cA u^\ell+ I(a)\cA u^h.
$$
Now, we have
 $$\displaylines{
 \int_0^t\langle t-\tau\rangle^{-\frac d4-\frac s2}\|I(a)\cA u^\ell\|_{L^1}\,d\tau\hfill\cr\hfill
 \lesssim \Bigl(\sup_{\tau\in[0,t]} \langle\tau\rangle^{\frac d4}\|a(\tau)\|_{L^2}\Bigr)
  \Bigl(\sup_{\tau\in[0,t]} \langle\tau\rangle^{\frac d4+1}\|\nabla^2 u^\ell(\tau)\|_{L^2}\Bigr)
 \int_0^t \langle t-\tau\rangle^{-\frac d4-\frac s2}\langle\tau\rangle^{-1-\frac d2}\,d\tau.}
 $$
 Hence, thanks to \eqref{eq:fact1}, the term with $I(a)\cA u^\ell$ fulfills \eqref{eq:decay3}. 
 Finally,  for $t\geq2,$ $$\displaylines{
 \int_0^t\langle t-\tau\rangle^{-\frac d4-\frac s2}\|I(a)\cA u^h\|_{L^1}\,d\tau
 \lesssim  \langle t\rangle^{-\frac d4-\frac s2}\int_0^1\|a\|_{L^2}\|\nabla^2 u^h\|_{L^2}\,d\tau
 \hfill\cr\hfill+\int_1^t  \langle t-\tau\rangle^{-\frac d4-\frac s2} \langle\tau\rangle^{-1-\frac{d}4}
 \bigl( \langle\tau\rangle^{\frac d4}\|a(\tau)\|_{L^2}\bigr)
\bigl(\tau\|\nabla^2u^h(\tau)\|_{L^2}\bigr)d\tau,}
$$
 hence,  because  $-d/2<s\leq2$ and $\| \tau\nabla^2 u^h\|_{L_t^\infty(L^2)}\lesssim \|\tau\nabla u\|^h_{\wt L^\infty_t(\dot B^{\frac d2}_{2,1})},$
 $$
  \int_0^t\langle t-\tau\rangle^{-\frac d4-\frac s2}\|I(a)\cA u^h\|_{L^1}\,d\tau
 \lesssim \langle t\rangle^{-\frac d4-\frac s2}\bigl(D^2(t)+X^2(t)\bigr)\quad\hbox{for }\ t\geq2.$$
Obviously,  as $\langle t\rangle\approx1$ and $\langle t-\tau\rangle\approx1$ for $0\leq\tau\leq t\leq 2,$ 
we have the following inequality:
$$
  \int_0^t\langle t-\tau\rangle^{-\frac d4-\frac s2}\|I(a)\cA u^h\|_{L^1}\,d\tau
 \lesssim \langle t\rangle^{-\frac d4-\frac s2}X^2(t)\quad\hbox{for }\ t\leq2,$$
which completes the proof of \eqref{eq:decay3}. 
Combining with \eqref{eq:decay2} and using Duhamel's formula, we conclude that for all $t\geq0$ and $s\in(-d/2,2],$
\begin{equation}\label{eq:decay4a}
\langle t\rangle^{\frac d4+\frac s2}\|(a,u)\|_{\dot B^s_{2,1}}^\ell \lesssim 
D_0+D^2(t)+X^2(t).
\end{equation}


\subsubsection*{Step 2: Decay estimates for the high frequencies of $(\nabla a, u)$}

We here want to bound the second term of $D(t).$
Recall that Theorem \ref{th:main} ensures that 
$$
\|(\nabla a,u)\|_{\wt L^\infty_T(\dot B^{\frac d2-1}_{2,1})}\leq CX(0)\quad\hbox{for all }\ T\geq0.
$$
Therefore it suffices to bound $\|t^\alpha(\nabla a,u)\|_{\wt L^\infty_T(\dot B^{\frac d2-1}_{2,1})}$
for, say,   $T\geq2.$ 

Now,  the starting point is Inequality \eqref{eq:Lk0} which implies that for $k\geq k_0$ and for some
$c_0=c(k_0)>0,$ we have
$$\displaylines{
\frac12\frac d{dt}\cL_k^2+c_0\cL_k^2\leq\Bigl(\|(\nabla f_k,g_k)\|_{L^2}\hfill\cr\hfill+\|R_k(u,a)\|_{L^2}
+\|R_k(u,u)\|_{L^2}+\|\wt R_k(u,a)\|_{L^2}+\|\nabla u\|_{L^\infty}\cL_k\Bigr)\cL_k}
$$
with $f\eqdefa -a\div u,$ $g=-k(a)\nabla a-I(a)\cA u,$
$R_k(u,b)\eqdefa \ddk(u\cdot\nabla b)-u\cdot\nabla\ddk b$ for $b\in\{a,u\},$
and $\wt R_k^i(u,a)\eqdefa \d_i\ddk(u\cdot\nabla a)-u\cdot\nabla\d_i\ddk a.$
\smallbreak
After time integration, we discover that
$$\displaylines{
e^{c_0t}\cL_k(t)\leq\cL_k(0)+\int_0^te^{c_0\tau}\Bigl(\|(\nabla f_k,g_k)\|_{L^2}+\|R_k(u,a)\|_{L^2}
\hfill\cr\hfill+\|R_k(u,u)\|_{L^2}+\|\wt R_k(u,a)\|_{L^2}+\|\nabla u\|_{L^\infty}\cL_k\Bigr)d\tau,}
$$
whence, remembering that $\cL_k\approx\|(\ddk\nabla a,\ddk u)\|_{L^2}$ for $k\geq k_0,$
$$
\displaylines{
t^\alpha\|(\ddk\nabla a,\ddk u)(t)\|_{L^2}\lesssim t^\alpha e^{-c_0t}
\|(\ddk\nabla a,\ddk u)(0)\|_{L^2}
+t^\alpha\int_0^te^{c_0(\tau-t)}\Bigl(\|(\nabla f_k,g_k)\|_{L^2}\hfill\cr\hfill
+\|R_k(u,a)\|_{L^2}
+\|R_k(u,u)\|_{L^2}+\|\wt R_k(u,a)\|_{L^2}+\|\nabla u\|_{L^\infty}\|(\ddk\nabla a,\ddk u)\|_{L^2}\Bigr)d\tau}$$
and thus, multiplying both sides by $2^{k(\frac d2-1)},$ taking the supremum on $[0,T],$
and summing up over $k\geq k_0,$
\begin{equation}\label{eq:decay4b}
\|t^\alpha(\nabla a,u)\|_{\wt L^\infty_T(\dot B^{\frac d2-1}_{2,1})}\lesssim 
\|(\nabla a_0,u_0)\|_{\dot B^{\frac d2-1}_{2,1}}^h
\!+\sum_{k\geq k_0}\sup_{0\leq t\leq T}\biggl(t^\alpha\!\int_0^t\!e^{c_0(\tau-t)}2^{k(\frac d2-1)}S_k\,d\tau\biggr)
\end{equation}
with $S_k\eqdefa\sum_{i=1}^5 S_k^i$ and 
$$\displaylines{
S_k^1\eqdefa \|(\nabla f_k,g_k)\|_{L^2},\quad
S_k^2\eqdefa\|R_k(u,a)\|_{L^2},\quad S_k^3\eqdefa\|R_k(u,u)\|_{L^2},\cr
S_k^4\eqdefa\|\wt R_k(u,a)\|_{L^2}, \quad
S_k^5\eqdefa\|\nabla u\|_{L^\infty}\|(\ddk\nabla a,\ddk u)\|_{L^2}.}$$
In order to bound the sum, we first notice that 
$$\displaylines{
\sum_{k\geq k_0}\sup_{0\leq t\leq 2}\biggl(t^\alpha\!\int_0^t\!e^{c_0(\tau-t)}2^{k(\frac d2-1)}S_k(\tau)\,d\tau\biggr)
\lesssim\int_0^2 \sum_{k\geq k_0}2^{k(\frac d2-1)}S_k(\tau)\,d\tau.}
$$
Hence taking advantage of \eqref{eq:transport2} and of a similar inequality 
for $\wt R_k(u,a),$ we end up with 
$$\displaylines{
\sum_{k\geq k_0}\sup_{0\leq t\leq 2}t^\alpha\!\int_0^t\!e^{c_0(\tau-t)}2^{k(\frac d2-1)}S_k\,d\tau
\lesssim\int_0^2 \Bigl(\|(\nabla f,g)\|_{\dot B^{\frac d2-1}_{2,1}}+
\|\nabla u\|_{\dot B^{\frac d2}_{2,1}}\|(a,\nabla a,u)\|_{\dot B^{\frac d2-1}_{2,1}}\Bigr)d\tau.}
$$
Bounding $\nabla f$ and $g$ as in the proof of Theorem \ref{th:main} leads to 
\begin{equation}\label{eq:decay5}
\sum_{k\geq k_0}\sup_{0\leq t\leq 2}t^\alpha\!\int_0^t\!e^{c_0(\tau-t)}2^{k(\frac d2-1)}S_k\,d\tau
\lesssim X^2(2).
\end{equation}
To bound the supremum on $[2,T],$ we split the integral on $[0,t]$ into  integrals 
on  $[0,1]$ and $[1,t],$ respectively.   
The $[0,1]$ part of the integral is easy to handle: we have $$\begin{aligned}
\sum_{k\geq k_0}\sup_{2\leq t\leq T}t^\alpha\!\int_0^1\!e^{c_0(\tau-t)}2^{k(\frac d2-1)}S_k(\tau)\,d\tau
&\leq  \sum_{k\geq k_0}\sup_{2\leq t\leq T} t^\alpha e^{-\frac{c_0}2t} \int_0^1 2^{k(\frac d2-1)}S_k\,d\tau\\
&\lesssim \int_0^1  \sum_{k\geq k_0}  2^{k(\frac d2-1)}S_k\,d\tau. 
\end{aligned}
$$
Hence
\begin{equation}\label{eq:decay5a}
\sum_{k\geq k_0}\sup_{2\leq t\leq T}\biggl(t^\alpha\!\int_0^1\!e^{c_0(\tau-t)}2^{k(\frac d2-1)}S_k(\tau)\,d\tau\biggr)\lesssim X^2(1).
\end{equation}
Let us finally consider the $[1,t]$ part of the integral for $2\leq t\leq T.$  We shall use repeatedly the following inequality
\begin{equation}\label{eq:decay9}
\|\tau\nabla u\|_{\wt L^\infty_t(\dot B^{\frac d2}_{2,1})}\lesssim D(t),
\end{equation}
which is straightforward as regards the high frequencies of $u$ and stems from 
$$
\|\tau\nabla u\|_{\wt L^\infty_t(\dot B^{\frac d2}_{2,1})}^\ell\lesssim
\|\langle\tau\rangle^{\frac d4+\frac 12}\nabla u\|_{\wt L^\infty_t(\dot B^{\frac d2}_{2,1})}^\ell
\lesssim\|\langle\tau\rangle^{\frac d4+\frac 12} u\|_{L^\infty_t(\dot B^{1}_{2,1})}^\ell\leq D(t)
$$
for the low frequencies of $u.$
\medbreak
Regarding the contribution of $S_k^1,$ we first notice that, by virtue of \eqref{eq:fact1},
\begin{equation}\label{eq:decay6}
\sum_{k\geq k_0}\sup_{2\leq t\leq T}t^\alpha\!\int_1^t\!e^{c_0(\tau-t)}2^{k(\frac d2-1)}S_k^1(\tau)\,d\tau
\lesssim \|\tau^\alpha(\nabla f, g)\|_{\wt L^\infty_T(\dot B^{\frac d2-1}_{2,1})}^h.
\end{equation}
Now, product laws in tilde spaces ensures that
$$
\|\tau^\alpha\nabla f\|_{\wt L^\infty_T(\dot B^{\frac d2-1}_{2,1})}^h\lesssim\|\tau^{\alpha-1} a\|_{\wt L^\infty_T(\dot B^{\frac d2}_{2,1})}
\|\tau\div u\|_{\wt L^\infty_T(\dot B^{\frac d2}_{2,1})}.
$$
The high frequencies of the first term of the r.h.s. is obviously bounded by $D(T).$ As for the low frequencies, we notice that
if $d\leq4$ then for all small enough $\ep>0,$ 
\begin{equation}\label{eq:decay6a}
\|\tau^{\frac d2-\ep}a\|^\ell_{\wt L_T^\infty(\dot B^{\frac d2}_{2,1})} \lesssim
 \|\tau^{\frac d2-\ep}a\|^\ell_{L_T^\infty(\dot B^{\frac d2-2\ep}_{2,1})}\leq D(T)\end{equation}
 and if $d\geq5,$ taking $s=2$ in the first term of $D(T),$
  \begin{equation}\label{eq:decay6b}\|\tau^{\frac d4+1}a\|^\ell_{\wt L_T^\infty(\dot B^{\frac d2}_{2,1})} \lesssim
 \|\tau^{\frac d4+1}a\|^\ell_{L_T^\infty(\dot B^{2}_{2,1})}\leq D(T).\end{equation}
Therefore, using \eqref{eq:decay9} and remembering  the definition of $\alpha,$ we get  
$$
\|\tau^\alpha\nabla f\|_{\wt L^\infty_T(\dot B^{\frac d2-1}_{2,1})}^h\lesssim D^2(T).
$$
Next, we have 
$$
\|\tau^\alpha(k(a)\nabla a^h)\|_{\wt L^\infty_T(\dot B^{\frac d2-1}_{2,1})}\lesssim \|a\|_{\wt L^\infty_T(\dot B^{\frac d2}_{2,1})}
\|\tau^\alpha a\|_{\wt L^\infty_T(\dot B^{\frac d2}_{2,1})}^h\leq X(T)D(T)
$$
and, according to \eqref{eq:decay6a} and \eqref{eq:decay6b}, 
$$
\|\tau^\alpha(k(a)\nabla a^\ell)\|_{\wt L^\infty_T(\dot B^{\frac d2-1}_{2,1})}\lesssim \|\tau^{1-\ep}a\|_{\wt L^\infty_T(\dot B^{\frac d2}_{2,1})}
\|\tau^{\alpha-1+\ep} a\|_{\wt L^\infty_T(\dot B^{\frac d2}_{2,1})}^\ell\lesssim D^2(T).
$$
We also see that
$$
\|\tau^\alpha I(a)\cA u\|_{\wt L^\infty_T(\dot B^{\frac d2-1}_{2,1})}
\lesssim\|\tau\nabla^2u\|_{\wt L^\infty_T(\dot B^{\frac d2-1}_{2,1})}\bigl(\|\tau^{\alpha-1}a\|_{\wt L^\infty_T(\dot B^{\frac d2}_{2,1})}^\ell
+\|\tau^{\alpha-1}a\|_{\wt L^\infty_T(\dot B^{\frac d2}_{2,1})}^h\bigr).
$$
The first term of the r.h.s. may be bounded by virtue of \eqref{eq:decay9}, and it is also clear that
the last term is bounded by $D(T).$ 
As for  the second one, we use again \eqref{eq:decay6a} and \eqref{eq:decay6b}.
Resuming to \eqref{eq:decay6}, we end up with 
$$
\sum_{k\geq k_0}\sup_{2\leq t\leq T}t^\alpha\!\int_1^t\!e^{c_0(\tau-t)}2^{k(\frac d2-1)}S_k^1(\tau)\,d\tau
\lesssim D^2(T).$$
To bound the term with $S_k^2,$ we use the fact that 
$$
\int_1^te^{c_0(\tau-t)}\|R_k(u,a)\|_{L^2}\,d\tau\leq
\|R_k(\tau u,\tau^{\alpha-1}a)\|_{L^\infty_t(L^2)}
\int_1^te^{c_0(\tau-t)} \tau^{-\alpha}\,d\tau.
$$
Hence, thanks to  \eqref{eq:fact1} and to \eqref{eq:transport2} (adapted to tilde spaces),
$$
\begin{aligned}
\sum_{k\geq k_0}\sup_{2\leq t\leq T}
\biggl(t^\alpha\!\int_1^t\!e^{c_0(\tau-t)}2^{k(\frac d2-1)}S_k^2(\tau)\,d\tau\biggr)
&\lesssim \sum_{k\geq k_0} 2^{k(\frac d2-1)}  \|R_k(\tau u,\tau^{\alpha-1}a)\|_{L^\infty_T(L^2)}\\
&\lesssim \|\tau\nabla u\|_{\wt L^\infty_T(\dot B^{\frac d2}_{2,1})} \|\tau^{\alpha-1}a\|_{\wt L^\infty_T(\dot B^{\frac d2-1}_{2,1})}.
\end{aligned}$$
The first term of the r.h.s. may be bounded thanks to \eqref{eq:decay9}, 
and the high frequencies of the last one are obviously bounded by $D(T).$
To bound   $\|\tau^{\alpha-1}a\|_{\wt L^\infty_T(\dot B^{\frac d2-1}_{2,1})}^\ell,$
we use the following two inequalities
$$\begin{aligned}
 \|\tau^{\alpha-1}a\|_{\wt L^\infty_T(\dot B^{\frac d2-1}_{2,1})}^\ell&\lesssim \|\tau^{\alpha-1}a\|_{L^\infty_T(\dot B^{\frac d2-1-2\ep}_{2,1})}^\ell&\quad\hbox{if }\ d\leq6,\\
  \|\tau^{\alpha-1}a\|_{\wt L^\infty_T(\dot B^{\frac d2-1}_{2,1})}^\ell&\lesssim \|\tau^{\alpha-1}a\|_{L^\infty_T(\dot B^{2}_{2,1})}^\ell&\quad\hbox{if }\ d\geq7.\end{aligned}
 $$
Because  $\alpha-1=\frac d2-\frac12-\ep$ if $d\leq6,$ and 
 $\alpha-1=\frac d4+1$   if $d\geq7,$ the r.h.s. above are bounded by $D(T).$
 We eventually get 
$$\sum_{k\geq k_0}\sup_{2\leq t\leq T}t^\alpha\!\int_1^t\!e^{c_0(\tau-t)}2^{k(\frac d2-1)}S_k^2(\tau)\,d\tau
\lesssim D^2(T).$$
 The terms $S_k^3$ and $S_k^4$  may be treated along the same lines.
\medbreak
Finally, using product laws and \eqref{eq:fact1}, we get
$$
\displaylines{
\sum_{k\geq k_0}\sup_{2\leq t\leq T}t^\alpha\!\int_1^t\!e^{c_0(\tau-t)}2^{k(\frac d2-1)}S_k^5(\tau)\,d\tau\hfill\cr\hfill
\lesssim \|\tau\nabla u\|_{\wt L^\infty_T(\dot B^{\frac d2}_{2,1})}\|\tau^{\alpha-1}(\nabla a,u)\|^h_{\wt L^\infty_T(\dot B^{\frac d2-1}_{2,1})}
\sup_{2\leq t\leq T} t^\alpha \int_1^te^{c_0(\tau-t)}\tau^{-\alpha}\,d\tau\lesssim D^2(T).}
$$ 
Putting all the above inequalities together, we conclude that 
$$
\sum_{k\geq k_0}\sup_{2\leq t\leq T}\biggl(t^\alpha\!\int_1^t\!e^{c_0(\tau-t)}2^{k(\frac d2-1)}S_k(\tau)\,d\tau\biggr)\lesssim D(T)X(T)+D^2(T). 
$$
Then plugging this latter inequality, \eqref{eq:decay5}   and \eqref{eq:decay5a} in  \eqref{eq:decay4b} yields
\begin{equation}\label{eq:decay7}
\|\langle\tau\rangle^\alpha(\nabla a,u)\|_{\wt L^\infty_T(\dot B^{\frac d2-1}_{2,1})}\lesssim
\|(\nabla a_0,u_0)\|^h_{\dot B^{\frac d2-1}_{2,1}}+X^2(T)+D^2(T).
\end{equation}


\subsubsection*{Step 3: Decay estimates with gain of regularity  for the high frequencies of $\nabla u$}

In order to bound the last term of $D(t),$ we shall use the fact that the velocity $u$ satisfies
$$
\d_tu-\cA u=F\eqdefa -(1+k(a))\nabla a-u\cdot\nabla u-I(a)\cA u,
$$
whence
$$
\d_t(t\cA u)-\cA(t\cA u)=\cA u+t\cA F.
$$
Because the maximal regularity estimates for the Lam\'e semi-group are the same as for the heat semi-group
(see the beginning of Section \ref{s:local}), 
we deduce from Remark \ref{r:heat}  that 
$$
\|t\cA u\|_{\wt L_t^\infty(\dot B^{\frac d2-1}_{2,1})}^h \lesssim \|\cA u\|_{L_t^1(\dot B^{\frac d2-1}_{2,1})}^h +\|t\cA F\|_{\wt L^\infty_t(\dot B^{\frac d2-3}_{2,1})}^h,
$$
whence, using the bounds given by Theorem \ref{th:main1}, 
\begin{equation}\label{eq:decay8}
\|t\nabla u\|_{\wt L_t^\infty(\dot B^{\frac d2}_{2,1})}^h \lesssim  X(0)
 +\|\tau F\|_{\wt L^\infty_t(\dot B^{\frac d2-1}_{2,1})}^h.
\end{equation}
In order to bound the last term, we notice that, because $\alpha\geq1,$ we have
$$
\|\tau\nabla a\|_{\wt L^\infty_t(\dot B^{\frac d2-1}_{2,1})}^h\lesssim
\|\langle\tau\rangle^\alpha a\|_{\wt L^\infty_t(\dot B^{\frac d2}_{2,1})}^h.
$$
Next, product and composition estimates adapted to tilde spaces give
$$
\|\tau\,k(a)\nabla a\|_{\wt L^\infty_t(\dot B^{\frac d2-1}_{2,1})}^h\lesssim \|\tau^{\frac12} a\|_{\wt L^\infty_t(\dot B^{\frac d2}_{2,1})}^2\leq D^2(t),
$$
as well as 
$$
\|\tau \, u\cdot\nabla u\|_{\wt L^\infty_t(\dot B^{\frac d2-1}_{2,1})}^h\lesssim \|u\|_{\wt L^\infty_t(\dot B^{\frac d2-1}_{2,1})}
\|\tau \nabla u\|_{\wt L^\infty_t(\dot B^{\frac d2}_{2,1})}
$$
and 
$$
\|\tau I(a)\cA u\|_{\wt L^\infty_t(\dot B^{\frac d2-1}_{2,1})}^h\lesssim \|a\|_{\wt L^\infty_t(\dot B^{\frac d2}_{2,1})}
\|\tau \nabla^2 u\|_{\wt L^\infty_t(\dot B^{\frac d2-1}_{2,1})}.$$
Therefore, resuming to \eqref{eq:decay8} and remembering \eqref{eq:decay9}, we get
$$
\|t\nabla u\|_{\wt L_t^\infty(\dot B^{\frac d2}_{2,1})}^h \lesssim  X(0) +D(t)X(t)+D^2(t)
+\|\langle\tau\rangle^\alpha a\|_{\wt L^\infty_t(\dot B^{\frac d2}_{2,1})}^h.
$$
Finally,  bounding the last term according
to \eqref{eq:decay7},  and adding up the obtained  inequality to \eqref{eq:decay4a} and \eqref{eq:decay7} yields
$$
D(t)\lesssim D_0+ \|(\nabla a_0, u_0)\|_{\dot B^{\frac d2-1}_{2,1}}^h+ X^2(t)+D^2(t).
$$
As Theorem \ref{th:main1} ensures that $X(t)$ is small, on can now conclude 
that \eqref{eq:D} is fulfilled for all time if $D_0$ and $ \|(\nabla a_0, u_0)\|_{\dot B^{\frac d2-1}_{2,1}}^h$
are small enough.
\end{proof}


\section{Appendix}

Here we  recall various estimates for the flow
that have been used repeatedly in the proof of Theorem \ref{th:small1}. 
More details may be found in \cite{D7} or \cite{DM}. 

Recall that if  $v:[0,T)\times\R^d\to\R^d$ is measurable, such that $t\mapsto v(t,x)$
is in $L^1(0,T)$ for all $x\in\R^d$ and in addition $\nabla v\in L^1(0,T;L^\infty)$ then it has, by virtue of the Cauchy-Lipschitz theorem, a unique $C^1$
flow $X_v$ satisfying
$$
X_v(t,y)=y+\int_0^t v(\tau,X_v(\tau,y))\,d\tau\quad\hbox{for all }\ t\in[0,T).
$$
In addition, for all $t\in[0,T),$ the map  $X_v(t,\cdot)$ is a $C^1$-diffeomorphism over $\R^d.$ 
\begin{lemma}\label{eq:flow} Let $p\in[1,\infty).$ 
Let $\bar v(t,y)\eqdefa v(t,X(t,y)).$ Under Assumption \eqref{eq:smallv},
we have for all $t\in[0,T],$
\begin{eqnarray}\label{eq:U1}
&&\|\Id-\adj(DX_v(t))\|_{\dot  B^{\frac dp}_{p,1}}\lesssim \|D\bar v\|_{L_t^1(\dot B^{\frac dp}_{p,1})},\\
\label{eq:U2}
&&\|\Id-A_v(t)\|_{\dot B^{\frac dp}_{p,1}}\lesssim \|D\bar v\|_{L_t^1(\dot B^{\frac dp}_{p,1})},\\
\label{eq:U4}&&\|\adj(DX_v(t)){}^T\!A_v(t)-\Id\|_{\dot B^{\frac dp}_{p,1}}
\lesssim \|D\bar v\|_{L_t^1(\dot B^{\frac dp}_{p,1})},\\
\label{eq:J}
&&\|J_v^{\pm1}(t)-1\|_{\dot B^{\frac dp}_{p,1}}\lesssim \|D\bar v\|_{L_t^1(\dot B^{\frac dp}_{p,1})}.
\end{eqnarray}
\end{lemma}
\begin{proof}
As an example, let us prove the last item. We have thanks to the chain rule, 
\begin{equation}\label{eq:Jv}
J_v(t,y)=1+\int_0^t\div v(\tau,X_v(\tau,y))\,J_v(\tau,y)\,d\tau
=1+\int_0^t (D\bar v:\adj(DX_v))(\tau,y)\,d\tau.
\end{equation}
Hence, if Condition  \eqref{eq:smallv} holds then we have \eqref{eq:J} for $J_v,$
a consequence of the fact that $\dot B^{\frac dp}_{p,1}$ is an algebra, and of \eqref{eq:U1}.
In order to get the inequality for $J_v^{-1},$ it suffices to notice that
$$
J_v^{-1}(t,y)-1=(1+(J_v(t,y)-1))^{-1}-1=\sum_{k\geq1}(-1)^k\int_0^tD\bar v:\adj(DX_v)\,d\tau.
$$
\end{proof}

\begin{lemma} 
Let $\bar v_1$ and $\bar v_2$ be two vector-fields satisfying \eqref{eq:smallv},
and $\dv\eqdefa\bar v_2-\bar v_1.$ 
Then we have for all  $p\in[1,\infty)$ and $t\in[0,T]$:
\begin{equation}\label{eq:dA}
\|A_{v_2}-A_{v_1}\|_{L_t^\infty(\dot B^{\frac dp}_{p,1})} \lesssim 
 \|D\dv\|_{L_t^1(\dot B^{\frac dp}_{p,1})},
\end{equation}
\begin{equation}\label{eq:dAdj}
\|\adj(DX_{v_2})-\adj(DX_{v_1})\|_{L_t^\infty(\dot B^{\frac dp}_{p,1})} \lesssim 
 \|D\dv\|_{L_t^1(\dot B^{\frac dp}_{p,1})},
\end{equation}
\begin{equation}\label{eq:dJ}
\|J_{v_2}-J_{v_1}\|_{L_t^\infty(\dot B^{\frac dp}_{p,1})} \lesssim 
 \|D\dv\|_{L_t^1(\dot B^{\frac dp}_{p,1})}.
\end{equation}
 \end{lemma}
 \begin{proof}
 In order to prove the first inequality, we use the fact
 that, for $i=1,2,$ we have
 $$
 A_{v_i}=(\Id+C_i)^{-1}=\sum_{k\geq0}(-1)^kC_i^k\quad\hbox{with}\quad
 C_i(t)=\int_0^t D\bar v_i\,d\tau.
 $$
 Hence
$$
A_{v_2}-A_{v_1}= \sum_{k\geq 1} \Bigl(C_2^k-C_1^k\Bigr)
=\biggl(\int_0^tD\dv\,d\tau\biggr)\sum_{k\geq1}\sum_{j=0}^{k-1}
C_1^jC_2^{k-1-j}.
$$
So using the fact that $\dot B^{\frac dp}_{p,1}$ is a Banach algebra, it is easy to conclude
to \eqref{eq:dA}.
Proving the second inequality is similar.
\end{proof}



\end{document}